\documentclass[10pt]{article}         

\usepackage{mathrsfs}
\usepackage{amsmath}
\usepackage{amsfonts}
\usepackage{amssymb}
\usepackage{amsthm}
\usepackage{caption}
\usepackage{subfigure}
\usepackage{cite}
\usepackage{color}
\usepackage{graphicx}
\usepackage{subfigure}
\usepackage{float}
\usepackage{paralist}
\usepackage[colorlinks,linkcolor=blue,anchorcolor=blue,citecolor=blue,  urlcolor=blue]{hyperref}
\usepackage{indentfirst}

\usepackage{authblk}
\usepackage[T1]{fontenc}
\usepackage{amscd}
\usepackage{latexsym}
\usepackage{verbatim}
\usepackage{enumerate}
\usepackage{fullpage}
\usepackage{xcolor}
\usepackage{titlesec}
\usepackage{titletoc}
\usepackage{nameref}
\usepackage{appendix}
\usepackage{doi}

\newtheorem{theorem}{Theorem}[section]
\newtheorem{lemma}{Lemma}[section]
\newtheorem{definition}{Definition}[section]

\newtheorem{remark}{Remark}[section]
\newtheorem{proposition}{Proposition}[section]
\numberwithin{equation}{section}
\newcommand{\RNum}[1]{\uppercase\expandafter{\romannumeral #1\relax}}
\title{Porous Medium Type Reaction-Diffusion Equation: Large Time Behaviors and Regularity of Free Boundary}
\author{Qingyou He\thanks{E-mail: qyhe.cnu.math@qq.com (Q. Y. He)}}
   \affil{School of Mathematical Sciences,
	Capital Normal University, Beijing 100048, P.R. China}
\date{}           

\begin{document}
\maketitle
\begin{abstract}
We consider the Cauchy problem of the porous medium type reaction-diffusion equation 
\begin{equation*}
\partial_t\rho=\Delta\rho^m+\rho g(\rho),\quad (x,t)\in \mathbb{R}^n\times \mathbb{R}_+,\quad n\geq2,\quad m>1,
\end{equation*}
where $g$ is  the given monotonic decreasing function with the  density critical threshold $\rho_M>0$ satisfying $g(\rho_M)=0$.  We prove that the pressure  $P:=\frac{m}{m-1}\rho^{m-1}$  in $L_{loc}^{\infty}(\mathbb{R}^n)$ tends  to the pressure critical threshold $P_M:=\frac{m}{m-1}(\rho_M)^{m-1}$ at the  time decay rate $(1+t)^{-1}$. If the initial density $\rho(x,0)$ is compactly supported,  we justify that the support $\{x: \rho(x,t)>0\}$ of the  density $\rho$ expands exponentially in time. Furthermore, we  show that there exists a time $T_0>0$ such that the pressure $P$ is Lipschitz continuous for $t>T_0$, which is the optimal (sharp) regularity of the pressure, and the free surface $\partial \{(x,t): \rho(x,t)>0\}\cap \{t>T_0\}$ is locally Lipschitz continuous. In addition,  under the same initial assumptions of compact support,  we verify  that the free boundary $\partial \{(x,t): \rho(x,t)>0\}\cap \{t>T_0\}$ is a local $C^{1,\alpha}$ surface. 
\end{abstract} 

\\
2020 {\bf MSC.} 35A24; 35B35; 35B40; 35B51; 35B65; 35D30; 35E10; 35G55; 35K10; 35K55; 35K57; 35K65; 35R35; 35Q92.
\tableofcontents
\section{Introduction}
The nonlinear reaction-diffusion equation is the governing equation for many very important nonlinear models that are used to describe various processes in biology,  ecology, physics,  chemistry, and so on, readers can refer to~\cite{BD2009, CD2017B,FKA1937,PERTHAME2014,WTP1997} and references therein. For instance in tumor growth, the reaction-diffusion equation with degenerate diffusion in~\cite{BD2009,PERTHAME2014} simulates that the cell proliferation is governed by the contact inhibition of bio-mechanical form, and the cell division  stops as the cell density reaches the critical threshold.  In this article, we consider the Cauchy problem of the porous medium type reaction-diffusion equation, written as
\begin{equation}\label{de}
\begin{cases}
\partial_t\rho=\Delta \rho^m+\rho g(\rho),\quad &(x,t)\in \mathbb{R}^n\times\mathbb{R}_+,\ n\geq 2,\\
\rho(x,0)=\rho_0(x)\geq 0,&x\in\mathbb{R}^n,
\end{cases}
\end{equation}
where  $\rho\geq0$ denotes the density, $m>1$ (slow diffusion) stands for the diffusion exponent, and $g$ is a given monotonic decreasing function with the density critical threshold $\rho_M>0$ satisfying $g(\rho_M)=0$. The another important variable for the porous medium type equation~$\eqref{de}_1$ is the pressure  $P$ defined by 
\begin{equation}\label{pressure}
P:=\frac{m}{m-1}\rho^{m-1},\quad  P(x,0):=P_{0}(x)=\frac{m}{m-1}\rho_0^{m-1}\ (\text{initial pressure}).
\end{equation}  For convenience, we set
\begin{equation}\label{Gg}
G(P):=g(\rho)
\text{ with }
G(P_M)=0,
\end{equation}
where $P_M$ is the corresponding pressure critical threshold given by 
\begin{equation}
P_M:=\frac{m}{m-1}(\rho_M)^{m-1}.
\end{equation}
By a direct computation, the equation of the pressure $P$ is expressed as 
\begin{equation}\label{pe}
\partial_t P=(m-1)P\big(\Delta P+G(P)\big)+\nabla P\cdot\nabla P,\quad (x,t)\in \mathbb{R}^n\times\mathbb{R}_+.
\end{equation}

If the initial density $\rho_0\geq0$ is compactly supported, the weak solution $\rho\geq 0$ to the Cauchy problem~\eqref{de}  remains compactly supported for all the time $t>0$; cf.~\cite{PERTHAME2014}. Naturally, an  interesting and important problem is to study the regularity of free (support) boundary $\partial\{(x,t): \rho(x,t)>0\}$ for the Cauchy problem~\eqref{de}. Indeed, we formally derive from~$\eqref{de}_1$ that
\begin{equation*}
\partial_t \rho=\nabla \rho\cdot \nabla P=|\nabla \rho| |\nabla P|\quad \text{on }\partial\{(x,t):\ \rho(x,t)>0\},
\end{equation*}
and  the motion of the free boundary  $\partial\{(x,t):\ \rho(x,t)>0\}$ is determined by Darcy's law, i.e., 
\begin{equation}\label{dracy}
V(x,t)=\frac{\partial_t \rho}{|\nabla \rho|}=|\nabla P(x,t)|\quad\text{on }\partial\{(x,t): \rho(x,t)>0\},
\end{equation}
where $V(x,t)$ is the outer normal velocity on the free boundary, which was verified in~\cite{JKT2021} by the method of optimal transport.  \\

Gavrilut and Morosanuin~\cite{GM2018} established the global well-posedness of the porous medium type reaction-diffusion equation~$\eqref{de}_1$ with  the Cauchy-Neumann boundary conditions, and the uniqueness  was obtained in~\cite{PERTHAME2014} through the comparison principle. In addition, for the porous medium type reaction-diffusion equation~$\eqref{de}_1$,  there were also many important previous works on the traveling wave or the spreading speed; cf.~\cite{HLM2010,LP2020,MMBA2010,SM1995,SJA2010,WB2016,XJMY2018,DQZ2020}. \\

The main purpose of this article is to study  the regularity of free boundary $\partial\{(x,t) : \rho(x,t)>0\}$ for the Cauchy problem~\eqref{de} of the porous medium type reaction-diffusion equation including the tumor growth model~\cite{PERTHAME2014} with $G(P)=P_M-P$ and the Fisher-KPP equation~\cite{DQZ2020} with $g(\rho)=1-\rho$. To this end,  as initiated in~\cite{CW1990}, we need not only  to obtain the large time asymptotic behaviors, but also to get Lipschitz's continuity of the pressure $P$~\eqref{pressure} for the Cauchy problem~\eqref{de}. \\

         Let us recall some important works on the large time behaviors~\cite{DQZ2020,PMEs,ZB1937},  and Lipschitz's continuity~\cite{ACV1985,CVW1987}  of the pressure $P$~\eqref{pressure} for the Cauchy problem of the  porous medium type  equation like~\eqref{PME}. The large time asymptotic behaviors of the non-negative weak solution to the porous medium equation(PME),
\begin{equation}\label{PME}
\partial_t\rho=\Delta\rho^m,\quad (x,t)\in\mathbb{R}^n\times\mathbb{R}_+,
\end{equation}
 was described by the self-similar (Barenblatt's) solution; cf.~\cite{PMEs,ZB1937}.  The porous medium type equation~$\eqref{de}_1$ with source term $\rho g(\rho):=\rho(1-\rho)$ is  the Fishier-KPP equation.  Authors in~\cite{DQZ2020} proved that the solution to the Cauchy problem for the Fisher-KPP equation converged locally to the density critical threshold $\rho_M=1$ as the time went to infinity,  expressed precisely as follows:
\begin{equation}\label{fkpppme}
\lim\limits_{t\to\infty}\rho(x,t)=1\quad \text{uniformly in }\{|x|\leq c_*t-c^*\log t\},
\end{equation}
where $c_*>0$ was the critical wave speed meaning the existence of wavefronts for all $c\geq c_*$ and $c^*>0$ is the corrector constant.  The pioneer work on Lipschitz's regularity for one-dimensional PME~\cite{ACV1985} told us that Lipschitz's continuity of the pressure for the Cauchy problem~\eqref{de} was the optimal (sharp) regularity. To the best of our knowledge,  \cite{CVW1987,DHL2001} were the only two  results including Lipschitz's continuity of the pressure for  the Cauchy problem of the n-dimensional ($n\geq2$) porous medium type equation~\eqref{PME} up to now. Indeed, from the perspective of the linearly non-degeneracy estimate near the free boundary (refer to Lemma~\ref{lin-degenerate}); cf.~\cite{CW1990,KZ2021}, we were also able to conclude  that Lipschitz's continuity was the optimal (sharp) regularity for the pressure of the PME~\eqref{PME}.  \\
      
There were many fundamental works on the regularity of the free boundary $\partial\{(x,t): \rho(x,t)>0\}$ for the $n$-dimensional ($n\geq2$) porous medium type equation~\eqref{PME}; refer to~\cite{ACV1985,CW1990,DH1997,DH1998,DHL2001,KZ2021} and references therein. Caffarelli et al  showed the $C^{1,\alpha}$ regularity of free boundary  for the multi-dimensional PME~\cite{CW1990}.   In addition,  P. Daskalopoulos et al developed a geometric approach and proved that the solution to the PME was smooth to the surface and the free boundary was $C^\infty$ smooth for all the time under  the root-concavity assumption of the initial pressure; cf.~\cite{DH1997,DH1998,DHL2001}. For the PME with a drift, the $C^{1,\alpha}$ regularity of free boundary  was  obtained under the assumptions of cone monotonicity and so-called \emph{weaker Lipschitz's regularity}; cf.~\cite{KZ2021}.  It should be emphasized that authors in~\cite{KZ2021} developed a new (i.e., inf-convolution) method to establish a local linearly non-degeneracy estimate near the free surface, which was also applicable to the PME~\eqref{PME}. \\

As  the diffusion exponent $m$ tends to infinity, the limit of the weak solution $\rho$ is called as the Hele-Shaw limit. Indeed, as $m\to+\infty$, authors in ~\cite{PERTHAME2014,KM2022}  derived from $\eqref{de}_1$  to a Hele-Shaw type system
\begin{equation}\label{HS}
\begin{cases}
\partial_t\rho_\infty=\Delta P_\infty+\rho_\infty G(P_\infty),&\\
P_\infty(1-\rho_\infty)=0,\quad 0\leq\rho_\infty\leq 1,&\\
P_\infty(\Delta P_\infty+G(P_\infty))=0,\quad \text{in }\mathbb{R}^n\times\mathbb{R}_+,&
\end{cases}
\end{equation}
where $\rho_\infty,P_\infty$ are the Hele-Shaw limits of the density $\rho$ and the pressure $P$ respectively. There were many important previous works on the Hele-Shaw  limits~\cite{DAVID2021,DPSV2021,DS2020,KT2018,KM2022,GKM2022,LX2021,PERTHAME2014,PV2015}, and on the Hele-Shaw problem~\cite{FTXZ2022,JKT2022,KT2021,MELLET2017}  (i.e., the Cauchy problem for~\eqref{HS} with the compactly supported initial density).  Perthame et al studied the Hele-Shaw asymptotic (limit) of the tumor growth model~\cite{PERTHAME2014}, and also did it for this model with nutrients~\cite{DAVID2021} by the method of weak solution. The Hele-Shaw  limit of the tumor growth with a drift was justified in~\cite{DS2021}, and the decay rate on the diffusion exponent~$m$ was obtained in~\cite{DDB2022}. For the tumor growth model with the motion governed by Brinkman's pressure law,   an optimal  uniform decay rate of the density and the pressure in $m$ was established in~\cite{KT2018}  through the viscosity solution approach,   the Hele-Shaw  limits of the two-species case  were proved in~\cite{DPSV2021,DS2020} by the compactness technique. Kim et al verified the Hele-Shaw  limit of the porous medium equation with the non-monotonic reaction terms through the variation of the obstacle problem; cf.~\cite{GKM2022,KM2022}. The existence of the weak solution and the free boundary (Hele-Shaw) limit of a tissue growth model with autophagy or necrotic core were established in~\cite{BS2022,LX2021} respectively.  In addition, the Hele-Shaw limits for the  chemotaxis were shown in \cite{CKY2018,HLP2022,HLP2023}. Mellet, Perthame and Quir\'os in~\cite{MELLET2017} obtained the regularity properties of the solution and the smoothness of the free boundary for the Hele-Shaw problem of tumor growth. Furthermore, the stability and  regularity of the solution, and the regularity of free boundary for the Hele-Shaw   problem  of the tumor growth with nutrients were proven in~\cite{JKT2022, KL2022,KT2021}. Recently, authors in~\cite{FTXZ2022} studied the tumor boundary instability of the Hele-Shaw   problem with nutrient consumption and supply.  \\

 On the one hand, compared with the case of the PME~\eqref{PME}, both the absence of the scaling invariant and the mass conservation,  and the presence of  the logistic growth effect for the Cauchy problem~\eqref{de} give rise to some new challenges. On the other hand, after a time integration for the pressure, the method of obstacle problem is well applied to study the regularity of the free boundary for the Hele-Shaw  problem like~\eqref{HS} of the tumor growth; cf.~\cite{JKT2022,MELLET2017,KL2022}.  However, this method from the Hele-Shaw problem~\eqref{HS} is not  applied to the case of the porous medium  type reaction-diffusion equation~\eqref{de}, that's because the stronger degeneration near the free boundary $\partial\{(x,t): \rho(x,t)>0\}$ means that the classical second order blow-up analysis technique originating from the obstacle problem~\cite{CLA1977} is too complicated to be applicable.\\

In the present paper, we first prove the large time asymptotic behaviors of the pressure $P$ for the Cauchy problem \eqref{de}. In particular, the pressure $P$ in $L_{loc}^\infty(\mathbb{R}^n)$ converges to the pressure critical threshold $P_M$ at the algebraic time decay rate $(1+t)^{-1}$ in time. Then, if the initial density $\rho_0$ is compactly supported,  there exits a time $T_0>0$, we show that the pressure $P$ is time-space Lipschitz continuous after the time $T_0$, where Lipschitz's continuity of the pressure is optimal (sharp), and furthermore the free boundary $\partial \{(x,t),\rho(x,t)>0\}\cap\{t>T_0\}$ is locally time-space Lipschitz continuous. Finally, we mainly follow the approach of~\cite{CW1990,KZ2021} to verify the local $C^{1,\alpha}$ smoothness of the free boundary $\partial\{(x,t): \rho(x,t)>0\}\cap \{\ t>T_0\}$. To the best of our knowledge, this paper is the first work on studying the time convergence rate of the solution, and the regularity of the pressure and its free (support) boundary for the porous medium type reaction-diffusion equation~\eqref{de}.\\

We give the definitions of weak solutions for the Cauchy problem~\eqref{de}.
\begin{definition}[Weak solution] Let $\rho_0 \in L^1(\mathbb{R}^n)$ with $0\leq \rho_{0}(x)\leq \rho_H$ ($\rho_H\geq\rho_M$). We call that the non-negative function $\rho: \mathbb{R}^n\times (0,\infty)\to[0,\infty)$ is a weak sub-solution (resp. sup-solution) to the Cauchy problem~\eqref{de} for the porous medium type reaction-diffusion equation if it holds for any $T>0$ and any non-negative test function $\varphi\in C_c^\infty(\mathbb{R}^n\times(0,T)) $ that
\begin{equation}\label{weaksolution}
\begin{cases}
\rho\in C([0,T];L^1(\mathbb{R}^n)),\quad\rho^m\in L^2([0,T]; \dot{H}^1(\mathbb{R}^n))\quad \text{and }\ \  g(\rho)\in L^\infty_{loc}(\mathbb{R}^n\times[0,T]),&\\
\int_{0}^{T}\hskip-4pt\int_{\mathbb{R}^n}\rho\partial_t\varphi dxdt\geq (\text{resp. } \leq) -\int_{0}^{T}\hskip-4pt\int_{\mathbb{R}^n}\rho^m\Delta \varphi dxdt-\int_{0}^{T}\hskip-4pt\int_{\mathbb{R}^n}\rho g(\rho)\varphi dxdt.&\\
\end{cases}
\end{equation}
Moreover, if \eqref{weaksolution} holds,  we claim that the corresponding pressure $P=\frac{m}{m-1}\rho^{m-1}$ is a weak sub-solution (resp. sup-solution) to the Cauchy problem of the pressure equation~\eqref{pe}. We call $\rho$ a weak solution to the Cauchy problem~\eqref{de} if and only if $\rho$ is both a weak sup-solution and a weak sub-solution to the Cauchy problem~\eqref{de}.
\end{definition}

In the rest of this paper, we use the definition of ball, denoted precisely by 
\begin{equation*}
B_R(x_0):=\{x\in\mathbb{R}^n:\ |x-x_0|\leq R\},\quad B_R:=\{x\in\mathbb{R}^n:\ |x|\leq R\}  \text{ for }x_0\in\mathbb{R}^n\text{ and }R>0.
\end{equation*}

 \noindent \emph{Organization of this paper.}\ The rest of the present paper is  organized as follows. In Section~\ref{mr}, we introduce the main results of this article. In Section~\ref{LTA}, we show the large time asymptotic behaviors of the solution to the Cauchy problem~\eqref{de}.   In Section~\ref{LC}, we prove Lipschitz's continuity of the pressure for the Cauchy problem~\eqref{de}.  In Section~\ref{fbr}, we justify that the free boundary~$\partial\{(x,t): \rho(x,t)>0\}$ for the Cauchy problem~\eqref{de} is a local $C^{1,\alpha}$ surface.
\\

\section{Main results}\label{mr}
We state main theorems of this paper. To ensure the strict expansion of the support as the time increases strictly,  we need the following initial assumptions:
\begin{equation}\label{i1}
\Omega_{0}:=\{x:\ P_0(x)>0\} \text{ is a bounded domain with Lipschitz's boundary},
\end{equation}
\begin{equation}\label{i2}
P_{0}(x)\geq C_0(d(x))^{2-\delta},\quad 0<\delta<1,\quad  C_0>0,
\end{equation}
where $d(x)$  denotes the distance between $x\in\Omega_0$ and $\partial \Omega_0$, readers can refer to~\cite{ACK1983} and references therein for the detailed  mechanism of the strict expansion.

\vspace{2mm}

 Lipschitz's continuity of the pressure for the Cauchy problem~\eqref{de}, which is the optimal (sharp) regularity for the porous medium type equation inspired by~\cite{ACV1985}, plays an important role in establishing the regularity of the free boundary $\partial\{(x,t): \rho(x,t)>0\}$.

\begin{theorem}[Sharp continuity]\label{t3}  Under the initial assumptions~\eqref{i1}-\eqref{i2} and $0\leq P_0\leq P_H$ for some $P_H\geq P_M$ satisfying $\inf\limits_{P\in[0,P_H]}G(P)-PG'(P)\geq 0$, $G'(P)\leq 0$ for $P\in[0,P_H]$, and $G(\cdot)\in \mathcal{C}^1(0,P_H]$, and assume that the pressure $P$ to Cauchy problem~\eqref{de} uniformly converges to the threshold $P_M$ in $L^\infty_{loc}(\mathbb{R}^n)$ as the time goes to infinity. Let    $R_0:=\inf\{R>0:supp(P_0)\subset B_{R}(0)\}$, then there exists
\begin{equation}\label{T_0}
T_0:=\inf\{t>0:B_{R_0}(0)\subset supp\big({P(t)\big)}\}
\end{equation}such that, for any $\delta>0$, $P$ is Lipschitz (sharp) continuous in $\mathbb{R}^n\times(T_0+\delta,\infty)$. To be more precise, it holds for the constant $C>0$ depending on $G(0),T_0, \delta,  d_G, P_M, m$ and the initial assumptions that
\begin{align*}
|\nabla P(x,t)|\leq C\max\{1,|x|\},\quad |\partial_t P(x,t)|\leq C\max\{1,|x|^2\},\quad (x,t)\in \mathbb{R}^n\times(T_0+\delta,\infty).
\end{align*}
\end{theorem}
\begin{remark}We discuss two classical examples like  Fisher-KPP equation $g(\rho)=1-\rho$ and tumor growth $G(P)=P_M-P$. Since $g(\rho)=1-\rho$ of the Fisher-KPP equation doesn't satisfy the assumption~\eqref{G} for the diffusion exponent $1<m<2$,  the results on the large time behaviors as in Theorem~\ref{t1}  don't hold.  However,  we still can get Lipschitz's  (sharp)continuity of the pressure for the Fisher-KPP equation for any $m>1$ as in Theorem~\ref{t3}  by means of the large time behaviors~\eqref{fkpppme} obtained in~\cite{DQZ2020}. For tumor growth $G(P)=P_M-P$, Theorem~\ref{t1} can cover all the diffusion exponent $m>1$. In one word, if we assume that \eqref{G} holds, then the pressure for the Fisher-KPP equation and tumor growth is Lipschitz continuous after a finite time by Theorem \ref{t3}. 
\end{remark}

\vspace{2mm}

 Based on Theorem~\ref{t3},  we can verify the $C^{1,\alpha}$ regularity of free (support) boundary for the Cauchy problem~\eqref{de}.  To obtain the regularity of the free (support) boundary, we assume that there exists a constant $K_G>0$ such that 
\begin{equation}\label{KG}
K_G:=\sup\limits_{0\leq P_1\leq P_2\leq \frac{P_M}{2}}-P_1G'(P_2)\geq \sup\limits_{0\leq P\leq \frac{P_M}{2}}-PG'(P)>0,
\end{equation}
which holds not only for $G(P)=P_M-P$ from the tumor growth~\cite{PERTHAME2014} but also for $g(\rho)=1-\rho$ from the Fisher-KPP equation~\cite{DQZ2020}.  We need to define  \begin{equation*}
 A_{r}(x,t):=B_{r}(x)\times[t-r,t+r],\quad(x,t)\in\mathbb{R}^n\times (T_0,\infty).
 \end{equation*}

\begin{theorem}[Regularity of free boundary]\label{t2}  Under both the initial assumptions and the assumptions on $g(\rho)$ as in Theorem~\ref{t3}, and $g(\rho)$ additionally satisfies $\eqref{KG}$. Then, for any $(x_0,t_0)\in \Gamma:=\partial\{(x,t) : \rho(x,t)>0\}\cap\{t_0>T_0\}$, there exists $\delta',C_0',L,c_2>0$ such that the pressure $P$ satisfies the following estimates:
\begin{itemize}
\item[(1)] $|\nabla P|,\ |\partial_tP|\leq L$\quad  in $A_{\delta'}(x_0,t_0)$,
\item[(2)] $\nabla_{\mu}P,\ \partial_t P\geq c_2>0$\quad in $\{(x,t): P(x,t)>0\}\cap A_{\delta'}(x_0,t_0)$,
\item[(3)] $\Delta P\geq -C_0'$\quad  in $A_{\delta'}(x_0,t_0)$.
\end{itemize}
Furthermore, there exist $\alpha\in (0,1)$ and $\delta_0\in (0,\delta')$ depending only on the constants $L,C_0',c_2$ such that $\Gamma\cap A_{\delta_0}(x_0,t_0)$ is a $C^{1,\alpha}$ surface. 
\end{theorem}

\begin{remark}
 The free (support) boundary of the tumor growth model $\&$ the Fisher-KPP equation for all diffusion exponent $m>1$ is  a local $C^{1,\alpha}$  surface after a finite time  due to Theorem~\ref{t3}.
\end{remark}

\begin{remark}
Based on Theorem~\ref{t2}, the boundary velocity  as shown in~\cite{CW1990} is the pressure gradient $\nabla P(x,t)$ on the free boundary $\Gamma(t):=\partial\{x: P(x,t)>0\} \cap\{t>T_0\}$, which is corresponding to Darcy's law~\eqref{dracy} originating from describing  the mechanism of  the gas through the porous medium. 
\end{remark}
\section{Large time behaviors}\label{LTA}
This section is devoted to studying the large time behaviors of the non-negative weak solution to the Cauchy problem \eqref{de}. Different from the case of the PME~\eqref{PME},  the absence of scaling invariance means the missing of self-similar solution,   and the presence of the nonlinear Lotka-Volterra source term $\rho g(\rho)$ means that there exist two constant equilibrium  states $0,\rho_M$. 
 \paragraph{Large time behaviors.}
 The compound effect of both the degenerate diffusion and the logistic growth effect is the main characteristic of the porous medium type reaction-diffusion equation~\eqref{de}.  We aim to prove that the solution $\rho$ tends to its critical threshold $\rho_M$ in time. To this end, we construct a new type weak sub-solution to the porous medium type reaction-diffusion equation~\eqref{de}, which locally tends to the critical threshold as the time goes to infinity.  For the reason of technique, we need to assume that there exists a constant $d_{G}>0$ for $m>1$ such that
\begin{equation}\label{G}
d_{G}:=\inf\limits_{\rho\in[0,\rho_M]}-g'(\rho)\frac{\rho^{2-m}}{m-1}=\inf\limits_{P\in[0,P_M]}[-G'(P)]>0,
\end{equation}
where both $G(P):=P_M-P$ from the tumor growth model~\cite{PERTHAME2014} and $g(\rho)=1-\rho$  from the Fisher-KPP equation with $m\geq2$~\cite{DQZ2020} satisfy~\eqref{G}. To construct the sub-solution, we need to introduce  the initial value problem for the Lotka-Volterra system~\cite[Section 12]{KM2001B}  as follows
 \begin{equation}\label{dy}
 \begin{cases}
 \dot{\alpha}(t)=(m-1)d_G\alpha(t)\big(P_M-\alpha(t)-\frac{2n}{d_G}\beta(t)\big),\quad t>0,&\\
 \dot{\beta}(t)=(m-1)d_G\beta(t)\big(P_M-\alpha(t)-\frac{2n(m-1)+4}{d_G(m-1)}\beta(t)\big),&\\
\alpha(0):=\alpha_0>0,\quad \beta(0):=\beta_0>0.&\\
 \end{cases}
 \end{equation}
\begin{theorem}[Large time asymptotic]\label{t1} Assume that $(\alpha(t),\beta(t))$ is the strong solution to the initial value problem~\eqref{dy} of the Lotka-Volterra system
 with the initial data $(\alpha_0, \beta_0)$ satisfying
\begin{equation}\label{ini1}
P_M-\frac{2n}{d_G}\beta_0>\alpha_0> P_M-\frac{2n(m-1)+4}{d_G(m-1)}\beta_0.
 \end{equation}
Then,  there exists a constant $C>0$ depending on $P_M,\alpha_0,\beta_0,m,d_G$ such that the optimal time decay rates of $(\alpha(t),\beta(t))$ hold as
\begin{equation*}
C^{-1}(1+t)^{-1}\leq  P_M-\alpha(t), \beta(t)\leq C(1+t)^{-1},\quad t>0. \end{equation*}
Furthermore, $(\alpha(t)-\beta(t)|x|^2)_+$ is a weak sub-solution to the Cauchy problem for the pressure equation~\eqref{pe} with  $g(\rho)$ satisfying~\eqref{G}. In addition, assume that the initial pressure $P_0(x)$ satisfies $(\alpha_0-\beta_0|x|^2)_+\leq P_0(x)\leq P_M$, $x\in \mathbb{R}^n$, then it holds
\begin{equation*}
\big(\alpha(t)-\beta(t)|x|^2\big)_+\leq P(x,t)\leq P_M,\quad (x,t)\in\mathbb{R}^n\times \mathbb{R}_+.
\end{equation*}
Furthermore, for any $R>0$, there exists a constant $C>0$ depending on $P_M,\alpha_0,\beta_0,m,d_G,R$ such that
\begin{equation*}
0\leq P_M-P(x,t)\leq C(1+t)^{-1},\quad (x,t)\in B_R\times \mathbb{R}_+.
\end{equation*}
\end{theorem}

\begin{remark}
Although~\eqref{dy} is a Lotka-Volterra system of competition, the results of the above theorem show the  cooperation of the nonlinear diffusion and the logistic growth.
\end{remark}
\begin{remark}
The method of Theorem~\ref{t1} can be extended to deal with $0\leq P_0\leq P_H$ for $P_H\geq P_M$. We still consider $(\alpha(t)-\beta(t)|x|^2)_+$ (Theorem~\ref{t1}) as a subsolution, and regard the function $f(t)$ satisfying $\frac{d f}{dt}=(m-1)fG(f)$ with $f(0)=P_H$ as a super-solution, then $f(t)\to P_M$ as $t\to\infty$. More precisely, for any $R>0$,  it holds for some constant $C>0$ depending on $P_M,P_H, \alpha_0,\beta_0,m,d_G,R$ that $\|P(t)-P_M\|_{L^\infty(B_R)}\leq |f(t)-P_M|+\|P_M-(\alpha(t)-\beta(t)|x|^2)_+\|_{L^\infty(B_R)}\leq \max\{(1+t)^{-1}, |f(t)-P_M|\},\ t\geq0$. For example of the tumor growth $G(f)=P_M-f$, by solving the ODE, we have $f(t)-P_M=\frac{e^{-(m-1)P_Mt}}{\frac{P_H}{P_H-P_M}-e^{-(m-1)P_Mt}}$, $t\geq0$. Hence, for any $R>0$, one can get large time behavior like
\[\|P(t)-P_M\|_{L^\infty(B_R)}\leq C\max\{(1+t)^{-1}, |f(t)-P_M|\}\leq C(1+t)^{-1},\quad t\geq0,\]
where $C>0$ is a constant depending on $P_M,P_H, \alpha_0,\beta_0,m,d_G,R$.
\end{remark}
 
 Next, we will prove Theorem~\ref{t1}. We first analyze the time decay rates  of the strong solution to the initial value problem for this Lotka-Volterra system.  
 \begin{lemma}[Lotka-Volterra system]\label{ce}
 Let $(\alpha(t),\beta(t))$ be the strong solution to the initial value problem of the Lotka-Volterra system \eqref{dy}
 with the initial data $\alpha_0>0,\beta_0>0$ satisfying~\eqref{ini1}.
 Then, it holds
 \begin{equation}\label{lts}
 \begin{cases}
 \alpha_0<\alpha(t)<P_M,\quad 0<\beta(t)<\beta_0,&\\
 P_M-\frac{2n}{d_G}\beta(t)>\alpha(t)> P_M-\frac{2n(m-1)+4}{d_G(m-1)}\beta(t),\quad t>0,&\\
 \alpha(t)\to P_M,\quad \beta(t)\to0,\quad\text{as }t\to\infty.&
 \end{cases}
 \end{equation}
 Moreover, the optimal time decay rates are obtained as follows: 
\begin{equation*}
C^{-1}(t+1)^{-1}\leq P_M-\alpha(t), \beta(t)\leq C(t+1)^{-1}, \quad t>0,
\end{equation*}
for some constant $C>0$ depending on $P_M,\alpha_0,\beta_0,m,d_G$.
 \end{lemma}
 \begin{proof}
 By the standard phase-plane analysis of the Lotka-Volterra system, \eqref{lts} evidently holds.

 By direct computations, it holds for $\eqref{dy}_2$ that
 \begin{equation}\label{dy22}
 \begin{aligned}
 \dot{\beta}(t)&=(m-1)d_G\beta(t)(P_M-\alpha(t)-\frac{2n}{d_G}\beta(t))-(m-1)d_G\beta^2(t)\\
 &\leq(m-1)d_G\beta(t)(P_M-\alpha(t)-\frac{2n}{d_G}\beta(t)).
 \end{aligned}
 \end{equation}
 Let $f_1(t):=\alpha(t)+\frac{2n}{d_G}\beta(t)$ with $f_1(0)=f_{1,0}:=\alpha_0+\frac{2n}{d_G}\beta_0<P_M$, it holds after~$\eqref{dy}_1+\frac{2n}{d_G}\times\eqref{dy22}$ that
 \begin{equation}\label{f_1}
 \dot{f}_1(t)\leq (m-1)d_Gf_1(t)(P_M-f_1(t)).
 \end{equation}
 Then, we obtain
 \begin{equation*}
 \dot{f}_1(t)(\frac{1}{f_1(t)}+\frac{1}{P_M-f_1(t)})\leq P_Md_G(m-1).
 \end{equation*}
It follows from the derivative of compound function that
\begin{equation*}
\frac{d}{dt}\{\log\frac{f_1(t)}{P_M-f_1(t)}\}\leq P_Md_G(m-1).
\end{equation*}
We integrate the above equation on $[0,t]$ for any $t>0$ and get
\begin{equation*}
\log\frac{f_1(t)}{P_M-f_1(t)}\leq\log\frac{f_{1,0}}{P_M-f_{1,0}}+P_Md_G(m-1)t.
\end{equation*}
It holds by taking the inverse operator $\log^{-1}=e^{*}$ action on the above inequality  that
\begin{equation*}
\frac{f_1(t)}{P_M-f_1(t)}\leq \frac{f_{1,0}}{P_M-f_{1,0}}e^{d_GP_M(m-1)t}.
\end{equation*}
Hence, we obtain
\begin{equation}\label{lowcon}
\alpha(t)+\frac{2n}{d_G}\beta(t)=f_{1}(t)\leq \frac{P_Mf_{1,0}}{(P_M-f_{1,0})e^{-P_Md_G(m-1)t}+f_{1,0}}.
\end{equation}

Similarly, it follows from $\eqref{dy}_1$ that
\begin{equation}\label{dy11}
\begin{aligned}
\dot{\alpha}(t)&=d_G(m-1)\alpha(t)(P_M-\alpha(t)-\frac{2n(m-1)+4}{(m-1)d_G}\beta(t))+\frac{4}{d_G(m-1)}\beta^2(t)\\
&\geq d_G(m-1)\alpha(t)(P_M-\alpha(t)-\frac{2n(m-1)+4}{(m-1)d_G}\beta(t)).
\end{aligned}
\end{equation}
Define $f_2(t)=\alpha(t)+\frac{2n(m-1)+4}{d_G(m-1)}\beta(t)$ with $f_{2,0}=f_2(0):=\alpha_0+\frac{2n(m-1)+4}{d_G(m-1)}\beta_0>P_M$, then it holds after $\eqref{dy11}+\frac{2n(m-1)+4}{d_G(m-1)}\times \eqref{dy}_2$ that
\begin{equation*}
\dot{f}_2(t)\geq d_G(m-1)f_2(t)(P_M-f_2(t)).
\end{equation*}
Similar to \eqref{f_1}, we solve directly the above inequality equation  and get
\begin{equation}\label{upcon}
\alpha(t)+\frac{2n(m-1)+4}{d_G(m-1)}\beta(t)=f_2(t)\geq \frac{P_Mf_{2,0}}{(f_{2,0}-P_M)e^{-P_Md_G(m-1)t}+f_{2,0}}.
\end{equation}

Since
\begin{equation}\label{expand}
\begin{aligned}
&\frac{d}{dt}\big\{\frac{\alpha(t)}{\beta(t)}\big\}=\frac{\dot{\alpha}(t)}{\beta(t)}-\frac{\alpha(t)\dot{\beta}(t)}{\beta^2(t)}\\
=&\frac{(m-1)d_G\alpha(t)}{\beta(t)}[(P_M-\alpha(t)-\frac{2n}{d_G}\beta(t)\big)-(P_M-\alpha(t)-\frac{2n(m-1)+4}{d_G(m-1)}\beta(t))]\\
=&\frac{(m-1)d_G\alpha(t)}{\beta(t)}\frac{4\beta(t)}{d_G(m-1)}\\
=&4\alpha(t)>0\\
\end{aligned}
\end{equation}
and $\alpha_0\leq \alpha(t)\leq P_M$ hold, we integrate~\eqref{expand} on $[0,t]$ for any $t>0$ and obtain
\begin{equation*}
\alpha_0t+\frac{\alpha_0}{\beta_0}\leq \frac{\alpha(t)}{\beta(t)}\leq P_Mt+\frac{\alpha_0}{\beta_0},\quad t>0.
\end{equation*}
Hence, the time decay rate of $\beta(t)$ is obtained as
\begin{equation}\label{betacon}
 \frac{\alpha_0}{P_Mt+\frac{\alpha_0}{\beta_0}}\leq \beta(t)\leq \frac{P_M}{\alpha_0t+\frac{\alpha_0}{\beta_0}},\quad t>0.
 \end{equation}

Inserting~\eqref{betacon} into~\eqref{lowcon}, we get
\begin{equation}\label{conrat1}
\begin{aligned}
P_M-\alpha(t)\geq& P_M-\frac{P_Mf_{1,0}}{(P_M-f_{1,0})e^{-P_Md_G(m-1)t}+f_{1,0}}+\frac{2n}{d_G}\beta(t)\\
\geq&\frac{P_M(P_M-f_{1,0})e^{-P_Md_G(m-1)t}}{(P_M-f_{1,0})e^{-P_Md_G(m-1)t}+f_{1,0}}+\frac{2n}{d_G}\frac{\alpha_0}{P_Mt+\frac{\alpha_0}{\beta_0}}.
\end{aligned}
\end{equation}
Similarly, we take~\eqref{betacon} into~\eqref{upcon} and attain
\begin{equation}\label{conrat2}
\begin{aligned}
P_M-\alpha(t)&\leq P_M-\frac{P_Me^{-P_Md_G(m-1)t}}{(f_{2,0}-P_M)e^{-P_Md_G(m-1)t}+f_{2,0}}+\frac{2n(m-1)+4}{d_G(m-1)}\beta(t)\\
&\leq \frac{P_M(f_{2,0}-P_M)e^{-P_Md_G(m-1)t}}{(f_{2,0}-P_M)e^{-P_Md_G(m-1)t}+f_{2,0}}+\frac{2n(m-1)+4}{d_G(m-1)}\frac{P_M}{\alpha_0t+\frac{\alpha_0}{\beta_0}}.
\end{aligned}
\end{equation}
Combining\eqref{betacon}-\eqref{conrat2}, there exists a constant $C>0$ depending on $P_M,\alpha_0,\beta_0,m,d_G$ such that
\begin{equation*}
C^{-1}(t+1)^{-1}\leq  P_M-\alpha(t), \beta(t)\leq C(t+1)^{-1},\quad t>0,
\end{equation*}
and the proof is completed.
 \end{proof}
 \begin{lemma}[Sub-solution]\label{subs}
 Let $Q(x,t):=\big(\alpha(t)-\beta(t)|x|^2\big)_+$ with $\big(\alpha(t),\beta(t)\big)$ being the strong solution to the initial value problem~\eqref{dy}  for the Lotka-Volterra system from Lemma~\ref{ce}. Then, $Q(x,t)$ is a weak sub-solution to the Cauchy problem for the pressure equation~\eqref{pe} if $g(\rho)\& G(P)$ satisfies~\eqref{G}.
 \end{lemma}
\begin{proof}
 For any given $P\in[0,P_M]$, thanks to the differential mean value theorem and the biological assumption~\eqref{G}, there exists $\xi\in[P,P_M]$ such that
 \begin{equation}
 G(P)=G(P)-G(P_M)=-G'(\xi)(P_M-P)\geq d_G(P_M-P).
 \end{equation}
By direct calculations in the positive parts of $Q$, we have following results:
\begin{equation*}
\partial_t Q(x,t)=\dot\alpha(t)-\dot\beta(t)|x|^2,
\end{equation*}
\begin{equation*}
\begin{aligned}
-(m-1)Q\Delta Q=&-(m-1)(\alpha(t)-\beta(t)|x|^2)(-2n\beta(t))\\
=&(m-1)2n\alpha(t)\beta(t)-(m-1)2n\beta^2(t)|x|^2,\\
\end{aligned}
\end{equation*}
\begin{equation*}
-|\nabla Q|^{2}=-|2\beta(t)x|^2=-4\beta^2(t)|x|^2,
\end{equation*}
and 
\begin{equation*}
\begin{aligned}
-(m-1)QG(Q)\leq& -(m-1)d_GQ(P_M-Q)\\
=&-(m-1)d_G(\alpha(t)-\beta(t)|x|^2)(P_M-\alpha(t)+\beta(t)|x|^2)\\
=&-(m-1)d_G\alpha(t)(P_M-\alpha(t))+(m-1)d_G\beta(t)(P_M-\alpha(t))|x|^2\\
&-(m-1)d_G\beta(t)|x|^2(\alpha(t)-\beta(t)|x|^2)\\
\leq&-(m-1)d_G\alpha(t)(P_M-\alpha(t))+(m-1)d_G\beta(t)(P_M-\alpha(t))|x|^2.
\end{aligned}
\end{equation*}
It concludes by combining the above results that
\begin{equation*}
\begin{aligned}
&\partial_tQ-(m-1)Q\Delta Q-|\nabla Q|^2-(m-1)QG(Q)\\
\leq& \big[\dot{\alpha}(t)-(m-1)d_G\alpha(t)(P_M-\alpha(t)-\frac{2n}{d_G}\beta(t))\big]\\
&-|x|^2\big[\dot{\beta}(t)-(m-1)d_G\beta(t)(P_M-\alpha(t)-\frac{2n(m-1)+4}{d_G(m-1)}\beta(t))\big]\\
=&0-0|x|^2\\
=&0\quad\text{in }\{(x,t)\in\mathbb{R}^n\times \mathbb{R}_+:\ Q(x,t)>0\},
\end{aligned}
\end{equation*}
and the proof is completed.
\end{proof}
\begin{proposition}[Large time asymptotic]\label{lta}Let $\alpha_0>0,\beta_0>0$ satisfy the assumptions~\eqref{ini1} and $P$ be the weak solution to the pressure equation~\eqref{pe} as $g(\rho)\& G(P)$ satisfies~\eqref{G}, and assume that the initial pressure $P_0$ satisfies $(\alpha_0-\beta_0|x|^2)_+\leq P_0(x)\leq P_M$ for $x\in\mathbb{R}^n$. Then, it holds 
\begin{equation*}
Q(x,t)\leq P(x,t)\leq P_M\quad\text{for } (x,t)\in \mathbb{R}^n\times \mathbb{R}_+,
\end{equation*}
where $Q(x,t):=(\alpha(t)-\beta(t)|x|^2)_+$ is from Lemma~\ref{subs}. Furthermore, $P$ locally converges to $P_M$ as the time goes to infinity. To be more precise,  there exists a constant $C>0$ depending on  $P_M,\alpha_0,\beta_0,m,d_G, R$ such that
\begin{equation*}
0\leq P_M-P(x,t)\leq C(1+t)^{-1},\quad (x,t)\in B_R(0)\times \mathbb{R}_+.
\end{equation*}
\end{proposition}
\begin{proof}
Due to Lemma~\ref{subs}, we conclude by the comparison principle that
\begin{equation}\label{ineq1}
Q(x,t)\leq P(x,t)\leq P_M,\quad (x,t)\in \mathbb{R}^n\times \mathbb{R}_+.
\end{equation}
From~\eqref{expand},  we have
\begin{equation} 
\begin{aligned}
\frac{d}{dt}\big\{\frac{\alpha(t)}{\beta(t)}\big\}
=4\alpha(t)>0,
\end{aligned}
\end{equation}
which means that the diameter of the support of $Q(x,t)$, $diam(\{x\in\mathbb{R}^n:Q(x,t)>0\})=\sqrt{\frac{\alpha(t)}{\beta(t)}}$,  expands strictly as the time increases, where $diam(A)$ represents the diameter of the set $A$.

Furthermore, for any given $R>\sqrt{\frac{\alpha_0}{\beta_0}}$, there exists $T>0$ such that $\sqrt{\frac{\alpha(T)}{\beta(T)}}=R$. Then, we have
\begin{equation}\label{decay}
Q(x,t)=(\alpha(t)-\beta(t)|x|^2)_+=\alpha(t)-\beta(t)|x|^2>0,\quad\text{for any }|x|<R\text{ and }t>T.
\end{equation}
On account of \eqref{lts}, we get
\begin{equation*}
Q(x,t)\to P_M\text{ uniformly in }B_R(0),\quad\text{as }t\to\infty.
\end{equation*}
Therefore, it holds by~\eqref{ineq1} that
\begin{equation*}
P(x,t)\to P_M\text{ uniformly in }B_R(0),\quad\text{as }t\to\infty.
\end{equation*}

In addition, it follows from~\eqref{decay} that
\begin{equation*}
0\leq P_M-P(x,t)\leq P_M-\alpha(t)+\beta(t)R^2\quad\text{for }|x|\leq R,\ t>T.
\end{equation*}
Hence, thanks to Lemma~\ref{ce}, there exists a constant $C>0$ depending on  $P_M,\alpha_0,\beta_0,m,d_G, R$ such that 
\begin{equation*}
0\leq P_M-P(x,t)\leq C(t+1)^{-1}, \quad (x,t)\in B_R(0)\times \mathbb{R}_+,
\end{equation*}
and the proof is completed.
\end{proof}

\begin{proof}[\underline{\textbf{Proof of Theorem~\ref{t1}}}] Combining  Lemmas~\ref{ce}-\ref{subs} and Proposition~\ref{lta}, then  the proof of Theorem~\ref{t1} is completed.
\end{proof}

\paragraph{Support expansion rate.}We show the time support expansion rate of $\rho$ as the time goes to infinity. When the initial pressure $P_0$ is strictly less than $P_M$ and is compactly supported, and $g(\rho)$ satisfies~\eqref{sup}, the following theorem is to estimate the support $\{x: P(x,t)>0\}$. To this end, we need the additional upper bound hypothesis of the derivative of $-G$
\begin{equation}\label{sup}
D_G:=\sup\limits_{\rho\in[0,\rho_M]}-g'(\rho)\frac{\rho^{2-m}}{m-1}=\sup\limits_{P\in[0,P_M]}-G'(P)>0,
\end{equation}
which holds for $G(P)=P_M-P$ from the tumor growth model~\cite{PERTHAME2014} and $g(\rho):=1-\rho$ from the Fisher-KPP equation with $1<m\leq2$~\cite{DQZ2020}.  For this purpose, we introduce the initial value problem of the another Lotka-Volterra system~\cite[Section 12]{KM2001B} as follows
\begin{equation}\label{lowdy}
 \begin{cases}
\dot{\lambda}(t)=D_G(m-1)\lambda(t)(P_M-\lambda(t)-\frac{2n}{D_G}\kappa(t)),\quad t>0,&\\
\dot{\kappa}(t)=D_G(m-1)\kappa(t)(P_M-2\lambda(t)-\frac{2n(m-1)+4}{D_G(m-1)}\kappa(t)),&\\
\lambda(0):=\lambda_0>0,\quad \kappa(0):=\kappa_0>0.&\\
 \end{cases}
 \end{equation}

\begin{theorem}[Support expansion rate]\label{t4}
Suppose that $(\lambda(t),\kappa(t))$ is a strong solution to the initial value problem~\eqref{lowdy} for the Lotka-Volterra  system
 with the initial data $\lambda_0>0,\kappa_0>0$ satisfying
 \begin{equation}\label{lowini}
 P_M-\lambda_0-\frac{2n}{D_G}\kappa_0>0>P_M-2\lambda_0-\frac{2n(m-1)+4}{D_G(m-1)}\kappa_0.
 \end{equation}
 Then, there exists a constant $C>0$ depending on $P_M,D_G,m,\lambda_0,\kappa_0$ such that it holds
  \begin{equation*}
C^{-1}e^{-[D_G(m-1)P_M+4\kappa_0]t}\leq P_M-\lambda(t),\ \kappa(t)\leq Ce^{-D_G(m-1)\lambda_0t},\quad t>0.
 \end{equation*}
In addition, $(\lambda(t)-\kappa(t)|x|^2)_+$ is a weak sup-solution to the Cauchy problem for the pressure equation~\eqref{pe}  when $g$ satisfies~\eqref{sup}. Under the initial assumption $P_0(x)\leq (\lambda_0-\kappa_0|x|^2)_+<P_M$, $x\in\mathbb{R}^n$, it holds
\begin{equation*}
0\leq P(x,t)\leq (\lambda(t)-\kappa(t)|x|^2)_+<P_M,\quad(x,t)\in\mathbb{R}^n\times\mathbb{R}_+,
\end{equation*}
and the support of $P(\cdot,t)$ satisfies
\begin{equation*}
supp(P(\cdot,t))\subset B_{R(t)}(0)\quad\text{with }R(t):=\sqrt{\frac{\lambda_0}{\kappa_0}}e^{\frac{D_GP_M(m-1)+4\kappa_0}{2}t},\quad t>0.
\end{equation*}
 Furthermore, there exists a constant $C>0$ depending on $P_M,D_G,m,\lambda_0,\kappa_0$ such that
 \begin{equation*}
P_M-P(x,t)\geq C^{-1}e^{-[P_MD_G(m-1)+4\kappa_0]t},\quad(x,t)\in\mathbb{R}^n\times \mathbb{R}_+.
\end{equation*}

\end{theorem}
The left of this section is devoted to proving Theorem~\ref{t4}. We first obtain the time convergence rates of Lotka-Volterra system~\eqref{lowdy}.
 \begin{lemma}\label{lowce}
 Let $(\lambda(t),\kappa(t))$ be a solution to the initial value problem for the Lotka-Volterra system \eqref{lowdy}
 under the initial assumption~\eqref{lowini}, then it holds
 \begin{equation}\label{llts}
 \begin{cases}
 \lambda_0<\lambda(t)<P_M,\quad 0<\kappa(t)<\kappa_0,\quad t>0,&\\
 P_M-\lambda(t)-\frac{2n}{D_G}\kappa(t)>0>P_M-2\lambda(t)-\frac{2n(m-1)+4}{D_G(m-1)}\kappa(t),&\\
 \lambda(t)\to P_M,\quad \kappa(t)\to0,\quad\text{as }t\to\infty.&
 \end{cases}
 \end{equation}
 Furthermore, there exists a constant $C>0$ depending on $P_M,D_G,m,\lambda_0,\kappa_0$ such that it holds
  \begin{equation}\label{kapparate}
C^{-1}e^{-[D_G(m-1)P_M+4\kappa_0]t}\leq P_M-\lambda(t),\  \kappa(t)\leq Ce^{-D_G(m-1)\lambda_0t},\quad t>0.
 \end{equation}

 \end{lemma}
 \begin{proof}
 By the standard phase-plane analysis of the Lotka-Volterra system, \eqref{llts} evidently holds.

Using~\eqref{lowdy}, we have
 \begin{equation}\label{expand1}
 \begin{aligned}
\frac{d}{dt}\{\frac{\lambda(t)}{\kappa(t)}\}&=\frac{\dot{\lambda}(t)\kappa(t)-\dot{\kappa}(t)\lambda(t)}{\kappa^2(t)}
=\frac{\lambda(t)}{\kappa(t)}[D_G(m-1)\lambda(t)+4\kappa(t)].
 \end{aligned}
 \end{equation}
 On account of~$\eqref{llts}_1$, we integrate~\eqref{expand1} on $[0,t]$ for any $t>0$ and obtain
 \begin{equation}\label{rate}
\frac{\lambda_0}{\kappa_0}e^{D_G(m-1)\lambda_0t}\leq \frac{\lambda(t)}{\kappa(t)}\leq \frac{\lambda_0}{\kappa_0}e^{[D_G(m-1)P_M+4\kappa_0]t}.
 \end{equation}
Naturally, the time exponential decay rate of $\kappa(t)$ is obtained as
 \begin{equation*}
 \kappa_0e^{-[D_G(m-1)P_M+4\kappa_0]t}\leq \kappa(t)\leq \frac{\kappa_0P_M}{\lambda_0}e^{-D_G(m-1)\lambda_0t}\quad\text{for }0\leq t<\infty,
 \end{equation*}
 so $\eqref{kapparate}_2$ holds.

 Let us define $g_1(t)=\lambda(t)+\frac{2n}{D_G}\kappa(t)$ with the initial data $g_{1,0}=g_1(0)=\lambda_0+\frac{2n}{D_G}\kappa_0<P_M$. It concludes by~$\eqref{lowdy}_1+\eqref{lowdy}_2\times \frac{2n}{D_G}$ that
 \begin{equation*}
 \dot{g}_1(t)\leq D_G(m-1)g_1(P_M-g_1).
 \end{equation*}
By directly solving the above inequality equation similar to~\eqref{f_1}, we have
 \begin{equation}\label{g1}
 g_1(t)\leq \frac{P_Mg_{1,0}}{(P_M-g_{1,0})e^{-D_GP_M(m-1)t}+g_{1,0}}.
 \end{equation}
 Similarly, we set $g_2(t)=\lambda(t)+\frac{2n(m-1)+4}{D_G(m-1)}\kappa(t)$ with the initial data $g_{2,0}=g_2(0)=\lambda_0+\frac{2n(m-1)+4}{D_G(m-1)}\kappa_0>P_M$. Then, it follows from~$\eqref{lowdy}_1+\eqref{lowdy}_2\times \frac{2n(m-1)+4}{D_G(m-1)}$ that
 \begin{equation*}
 \dot{g}_2(t)\geq D_G(m-1)g_2(P_M-g_2).
 \end{equation*}
 It holds by the similar computations with~\eqref{f_1} that
 \begin{equation}\label{g2}
 g_2(t)\geq \frac{P_Mg_{2,0}}{(g_{2,0}-P_M)e^{-D_GP_M(m-1)t}+g_{2,0}}.
 \end{equation}
 Therefore, it holds by combining \eqref{g1} and \eqref{g2} that
 \begin{equation*}
\begin{aligned}
P_M-\lambda(t)&=P_M-g_1(t)+\frac{2n}{D_G}\kappa(t)\\
&\geq\frac{P_M(P_M-g_{1,0})e^{-D_GP_M(m-1)t}}{(P_M-g_{1,0})e^{-D_GP_M(m-1)t}+g_{1,0}}+\frac{2n}{D_G}\kappa(t),\\
P_M-\lambda(t)&=P_M-g_2(t)+\frac{2n(m-1)+4}{D_G(m-1)}\kappa(t)\\
&\leq\frac{P_M(g_{2,0}-P_M)e^{-D_GP_M(m-1)t}}{(g_{2,0}-P_M)e^{-D_GP_M(m-1)t}+g_{2,0}}+\frac{2n(m-1)+4}{D_G(m-1)}\kappa(t),
\end{aligned}
\end{equation*}
so $\eqref{kapparate}_1$ holds. The proof is completed.
 \end{proof}

\begin{lemma}\label{lows}
Let $(\lambda(t),\kappa(t))$ be the strong solution to the Lotka-Volterra system~\eqref{lowdy} under the initial assumption~\eqref{lowini}. Then, $(\lambda(t)-\kappa(t)|x|^2)_+$ is a weak sup-solution to the Cauchy problem for the pressure equation~\eqref{pe} as  $g(\rho)\& G(P)$ satisfies~\eqref{sup}. Furthermore, under the initial assumption $P_0(x)\leq (\lambda_0-\kappa_0|x|^2)_+<P_M$ for $x\in\mathbb{R}^n$, it holds for the pressure $P$ that
\begin{equation}\label{comsup}
0\leq P(x,t)\leq (\lambda(t)-\kappa(t)|x|^2)_+<P_M\quad\text{for }(x,t)\in\mathbb{R}^n\times\mathbb{R}_+,
\end{equation}
and the support of $P(\cdot,t)$ can be estimated as
\begin{equation*}
supp(P(\cdot,t))\subset B_{R(t)}(0)\quad\text{with }R(t):=\sqrt{\frac{\lambda_0}{\kappa_0}}e^{\frac{D_GP_M(m-1)+4\kappa_0}{2}t}.
\end{equation*}
 Furthermore, there exists a constant $C>0$ depending on $P_M,D_G,m,\lambda_0,\kappa_0$ such that
 \begin{equation*}
P_M-P(x,t)\geq P_M-(\lambda(t)-\kappa(t)|x|^2)_+\geq C^{-1}e^{-[P_MD_G(m-1)+4\kappa_0]t},\quad (x,t)\in\mathbb{R}^n\times \mathbb{R}_+.
\end{equation*}

\end{lemma}
\begin{proof}
Due to the differential mean value theorem and the additional biological assumption~\eqref{sup}, there exists $\xi\in[P,P_M]$ for $\forall P\in[0,P_M]$ such that
 \begin{equation}
 G(P)=G(P)-G(P_M)=-G'(\xi)(P_M-P)\leq D_G(P_M-P).
 \end{equation}
Set $Q(x,t):=(\lambda(t)-\kappa(t)|x|^2)_+$, by direct calculations in the positive parts of $Q$, we have following results:
\begin{equation*}
\partial_t Q(x,t)=\dot\lambda(t)-\dot\kappa(t)|x|^2,
\end{equation*}
\begin{equation*}
\begin{aligned}
-(m-1)Q\Delta Q
=(m-1)2n\lambda(t)\kappa(t)-(m-1)2n\kappa^2(t)|x|^2,
\end{aligned}
\end{equation*}
\begin{equation*}
-|\nabla Q|^{2}=-|2\kappa(t)x|^2=-4\kappa^2(t)|x|^2,
\end{equation*}
and
\begin{equation*}
\begin{aligned}
-(m-1)QG(Q)\geq& -(m-1)D_GQ(P_M-Q)\\
=&-(m-1)D_G(\lambda(t)-\kappa(t)|x|^2)(P_M-\lambda(t)+\kappa(t)|x|^2)\\
=&-(m-1)D_G\lambda(t)(P_M-\lambda(t))+(m-1)D_G\kappa(t)(P_M-\lambda(t))|x|^2\\
&-(m-1)D_G\kappa(t)|x|^2\lambda(t)+(m-1)D_G\kappa^2(t)|x|^4\\
\geq&-(m-1)D_G\lambda(t)(P_M-\lambda(t))+(m-1)D_G\kappa(t)(P_M-2\lambda(t))|x|^2.
\end{aligned}
\end{equation*}
Combining the above results, it follows
\begin{equation*}
\begin{aligned}
&\partial_tQ-(m-1)Q\Delta Q-|\nabla Q|^2-(m-1)QG(Q)\\
\geq& [\dot{\lambda}(t)-(m-1)d_G\lambda(t)(P_M-\lambda(t)-\frac{2n}{d_G}\kappa(t))]\\
&-|x|^2[\dot{\kappa}(t)-(m-1)d_G\kappa(t)(P_M-2\lambda(t)-\frac{2n(m-1)+4}{d_G(m-1)}\kappa(t))]\\
=&0-0|x|^2\\
=&0\quad\text{in }\{(x,t)\in\mathbb{R}^n\times \mathbb{R}_+:\ Q(x,t)>0\}
\end{aligned}
\end{equation*}

By means of the conclusions of Lemma~\ref{lowce} and the comparison principle, \eqref{comsup} naturally holds.

Similar to Lemma~\ref{lts}, it holds by a direct computation that
\begin{equation*}
P_M-P(x,t)\geq P_M-(\lambda(t)-\kappa(t)|x|^2)_+\geq P_M-\lambda(t)\geq C^{-1}e^{-[P_MD_G(m-1)+4\kappa_0]t}
\end{equation*} for $(x,t)\in \mathbb{R}^n\times \mathbb{R}_+$, where the constant $C>0$ depends on $P_M,\lambda_0,\kappa_0,D_G,R,m$. The proof is completed.
\end{proof}

\begin{proof}[\underline{\textbf{Proof of Theorem~\ref{t4}}}] Combining  Lemmas~\ref{lowce}-\ref{lows}, the proof of Theorem~\ref{t4} is completed.
\end{proof}
\section{Lipschitz's conitinuity}\label{LC}
Lipschitz's continuity of the pressure to the Cauchy problem~\eqref{de} for  the porous medium type reaction-diffusion equation is an important topic and  is crucial for proving the regularity of free boundary. The Aronson-B\'enilan estimate  originating from~\cite{AB1979}  plays an important role in proving the regularity of solutions to the porous medium type equation~\eqref{PME}. Indeed, the Aronson-B\'enilan estimate derives the semi-harmonicity property i.e. $\Delta P >-\infty$ for $t>0$.

\begin{lemma}[Aronson-B\'enilan estimate]\label{lAB}
 Let $\rho$ be the weak solution to the Cauchy problem~\eqref{de} with $m>1$ and the initial assumption $0\leq P_0\leq P_H$ for some $P_H\geq P_M$ satisfying $\inf\limits_{P\in[0,P_H]}G(P)-PG'(P)\geq 0$, then it holds
\begin{equation}\label{ab2}
\Delta P+G(P)\geq -\frac{1}{(m-1)t}\quad \text{in }\mathcal{D}'(\mathbb{R}^n\times \mathbb{R}_+).
\end{equation}
\end{lemma}
\begin{proof}Let $\Delta \times \eqref{pe}+G'(P)\times\eqref{pe}$ and set $\omega:=\Delta P+G(P)$, then we have
\begin{equation}\label{ab1}
\begin{aligned}
\partial_t \omega=&(m-1)P\Delta \omega+2m\nabla P\cdot\nabla \omega+(m-1)\omega^2-(G(P)-PG'(P))(m-1)\omega\\&-G'(P)|\nabla P|^2+2\nabla^2 P:\nabla^2 P\\
\geq&(m-1)P\Delta \omega+2m\nabla P\cdot\nabla \omega+(m-1)\omega^2-(G(P)-PG'(P))(m-1)\omega.
\end{aligned}
\end{equation}
Set $|f(x)|_-:=-\min\{f(x),0\}$ and  multiplying~\eqref{ab1} by $sign(|\omega|_-)$,  thanks to Kato's inequality,  it holds in the sense of distribution that 
\begin{equation}
\begin{aligned}
\partial_t |\omega|_-\leq& (m-1)P\Delta |\omega|_-+2m\nabla P\cdot\nabla |\omega|_--(m-1)|\omega|_-^2\\
&-(G(P)-PG'(P))(m-1)|\omega|_-\\
\leq & (m-1)P\Delta |\omega|_-+2m\nabla P\cdot\nabla |\omega|_--(m-1)|\omega|_-^2.
\end{aligned}
\end{equation}

Since $f(t)=\frac{1}{(m-1)t}$ with $f(0^+)=\infty$ solves $\frac{d}{dt} f=-(m-1)f^2$,  we conclude by the comparison principle that $|\omega(\cdot,t)|_-\leq f(t)$  in the sense of distribution, which means that~\eqref{ab2} holds. The proof is completed.
\end{proof}
\begin{remark}
It should be  emphasized that the additional assumption on $G$, i.e., $\min\limits_{p\in [0,P_M]}[G(P)-PG'(P)]>0$ in \cite{PERTHAME2014}, is not necessary anymore for the porous medium type reaction-diffusion equation~\eqref{de}.
\end{remark}
\begin{lemma}[Cone monotonicity]\label{conemon} Let $x_0,x_1\in \mathbb{R}^n$, with $|x_1|$, $|x_0|>R_0$, satisfy
\begin{equation*}
\cos\langle x_1-x_0,x_0\rangle\geq\frac{R_0}{|x_0|},
\end{equation*}
where $R_0:=\inf\{R>0:supp(P_0)\subset B_{R}(0)\}$. Then, it holds for any $t>0$ that
\begin{equation*}
P(x_1,t)\leq P(x_0,t).
\end{equation*}
\end{lemma}
\begin{proof} Parallel to Alexander's reflection principle of~\cite{AC1983}, the detailed proof is omitted.
\end{proof}
\begin{lemma}\label{dgg}
 Let $\rho$ be the weak solution to the Cauchy problem~\eqref{de} with $m>1$. Set $\Omega(t):=\{x\in\mathbb{R}^n: P(x,t)>0\}$, then it holds \begin{equation}\label{dg}
 \Omega(t)\subset \Omega(t+s) \quad\text{ for any } t>0,\ s\geq0.
 \end{equation}
\end{lemma}
\begin{proof}
From the pressure equation \eqref{pe}, we have
\begin{equation*}
\partial_t P=(m-1)P(\Delta P+G(P))+\nabla P\cdot\nabla P \geq(m-1)P(\Delta P+G(P)).
\end{equation*}
 Due to the Aronson-B\'enilan estimate (Lemma~\ref{lAB}), it follows
 \begin{equation*}
 \partial_t P(x,t)\geq-\frac{P(x,t)}{t}.
 \end{equation*}
 After a direct computation, we get
 \begin{equation*}
 \partial_t[tP(x,t)\big]\geq0.
 \end{equation*}
It holds by integrating above inequality on $[t,t+s]$ for $t>0,s\geq0$ that
 \begin{equation*}
(t+s)P(x,t+s)\geq tP(x,t)\quad\text{ for }x\in\mathbb{R}^n,
 \end{equation*}
  which implies
 \begin{equation*}
 \Omega(t)\subset \Omega(t+s)\text{ for any }t>0,\ s\geq0.
 \end{equation*}
The proof of Lemma~\ref{dgg} is completed.
\end{proof}
Nowadays, we focus on Lipschitz's continuity of the pressure for the Cauchy problem \eqref{de}-\eqref{pe}. As in~\cite{CVW1987}, 
the linearly scaled pressue
 \begin{equation}\label{scall}
P_{\epsilon}=\frac{1}{1+\epsilon}P\big((1+\epsilon)x,(1+\epsilon)t+t_0\big)\quad\text{for any }\epsilon>0
 \end{equation}
is also a solution to the corresponding pressure equation for the PME~\eqref{PME} as follows:
\begin{equation*}
\partial_t P=(m-1)P\Delta P+|\nabla P|^2.
 \end{equation*}Evidently, the scaling strategy~\eqref{scall} is not suitable to the case of the pressure equation~\eqref{pe}. Hence, new nonlinear scaling function and strategy are needful to get Lipschitz's  continuity of pressure for the porous medium type reaction-diffusion equation~$\eqref{de}_1$.

The Lotka-Volterra source term $PG(P)$ is non-monotonic, we need to introduce a new invariable and make a transformation to achieve our goal. Inspired by~\cite{MELLET2017}, we  set $u:=e^{-G(0)t}\rho(x,t)$ which satisfies
\begin{equation*}
\partial_t u=e^{(m-1)G(0)t}\Delta u^{m}+u[G(\frac{m}{m-1}e^{(m-1)G(0)t}u^{m-1})-G(0)].
\end{equation*}
\begin{remark}
We notice that the source term $u[G(\frac{m}{m-1}e^{(m-1)G(0)t}u^{m-1})-G(0)]$ of the above equation is monotonic on u, which helps overcome the difficulties that we will face in the next.
\end{remark}

We define the corresponding pressure $v:=\frac{m}{m-1}u^{m-1}$ and the degenerate parabolic operator as following:
\begin{equation}\label{npe}
\mathcal{L}_1(v):=\partial_t v-e^{(m-1)G(0)t}\big[(m-1)v\Delta v+|\nabla v|^2\big]-(m-1)v[G(e^{(m-1)G(0)t}v)-G(0)]=0.
\end{equation}
\begin{lemma}\label{subs2}
Let $\beta(\epsilon,t):=\frac{1}{(m-1)G(0)}\log\big[1+(e^{(m-1)G(0)t}-1)(1+\epsilon)\big]$ with $\epsilon>0$ and $\beta(\epsilon,0)=0$ satisfying
\begin{equation}\label{est6}
\frac{\partial_t\beta(\epsilon,t)}{1+\epsilon}e^{(m-1)G(0)\beta(\epsilon,t)}-e^{(m-1)G(0)t}=0.
\end{equation}
Then, it follows
\begin{equation*}
\mathcal{L}_1\big(\frac{1}{1+\epsilon}v((1+\epsilon)x,\beta(\epsilon,t)+t_0)\big)\leq 0\quad\text{for }(x,t)\in\mathbb{R}^n\times \mathbb{R}_+\text{ and }t_0\geq0.
\end{equation*}
\end{lemma}
\begin{proof}For simplicity, we set $v_\epsilon:=\frac{1}{1+\epsilon}v((1+\epsilon)x,\beta(\epsilon,t)+t_0)$ and $v:=v((1+\epsilon)x,\beta(\epsilon,t)+t_0)$.

We first show the time derivative of $v_{\epsilon}\geq 0$ as
\begin{equation}\label{est1}
\begin{aligned}
\partial_tv_\varepsilon=
&\partial_t\big(\frac{1}{1+\epsilon}v((1+\epsilon)x,\beta(\epsilon,t)+t_0)\big)\\
=&\frac{\partial_t\beta(\epsilon,t)}{1+\epsilon}D_tv((1+\epsilon)x,\beta(\epsilon,t)+t_0)\\
=&\frac{\partial_t\beta(\epsilon,t)}{1+\epsilon}e^{(m-1)G(0)(\beta(\epsilon,t)+t_0)}[(m-1)v\Delta_Xv+|\nabla_X v|^2]\\
&+\frac{\partial_t\beta(\epsilon,t)}{1+\epsilon}(m-1)v[G(e^{(m-1)G(0)(\beta(\epsilon,t)+t_0)}v)-G(0)],
\end{aligned}
\end{equation}
where $D_t,\nabla_X,\Delta_X$ are derivatives of temporal and spatial positions.

We turn to the spatial derivative terms and the source term:
\begin{equation}\label{est2}
e^{(m-1)G(0)(t+t_0)}[(m-1)v_\epsilon\Delta v_\epsilon+|\nabla v_\epsilon|^2]=e^{(m-1)G(0)(t+t_0)}[(m-1)v\Delta_X v+|\nabla_X v|^2],
\end{equation}
\begin{equation}\label{est3}
(m-1)v_\epsilon[G(e^{(m-1)G(0)(t+t_0)}v_\epsilon)-G(0)]= (m-1)\frac{v}{1+\epsilon}[G(e^{(m-1)G(0)(t+t_0)}\frac{v}{1+\epsilon})-G(0)].
\end{equation}
To this end, we need to investigate some preliminary properties of $\beta(\epsilon,t)$ as follows:
\begin{equation}\label{est4}
\begin{aligned}
\partial_t\beta(\epsilon,t)=&[e^{(m-1)G(0)t}e^{-(m-1)G(0)\beta(\epsilon,t)}](1+\epsilon)\\
=&\frac{e^{(m-1)G(0)t}(1+\epsilon)}{1+(e^{(m-1)G(0)t}-1)(1+\epsilon)}\\
\geq&1,
\end{aligned}
\end{equation}
\begin{equation}\label{est5}
\begin{aligned}
\frac{e^{(m-1)G(0)t}}{1+\epsilon}-e^{(m-1)G(0)\beta(\epsilon,t)}=&\frac{e^{(m-1)G(0)t}}{1+\epsilon}-1-(e^{(m-1)G(0)t}-1)(1+\epsilon)\\
\leq&0.
\end{aligned}
\end{equation}
Combining \eqref{est1}-\eqref{est3}, it follows
\begin{equation}\label{est7}
\begin{aligned}
&\mathcal{L}_1(v_\epsilon)\\
=&e^{(m-1)G(0)t_0}\big[\frac{\partial_t\beta(\epsilon,t)e^{(m-1)G(0)\beta(\epsilon,t)}}{1+\epsilon}-e^{(m-1)G(0)t}\big]\big[(m-1)v\Delta_X v+|\nabla_X v|^2\big]\\
&+\frac{\partial_t\beta(\epsilon,t)}{1+\epsilon}(m-1)v[G(e^{(m-1)G(0)(\beta(\epsilon,t)+t_0)}v)-G(0)]\\
&-(m-1)\frac{v}{1+\epsilon}[G(e^{(m-1)G(0)(t+t_0)}\frac{v}{1+\epsilon})-G(0)]\\
=&\frac{\partial_t\beta(\epsilon,t)}{1+\epsilon}(m-1)v[G(e^{(m-1)G(0)(\beta(\epsilon,t)+t_0)}v)-G(0)]\\
&-(m-1)\frac{v}{1+\epsilon}[G(e^{(m-1)G(0)(t+t_0)}\frac{v}{1+\epsilon})-G(0)]\\
\leq &\frac{\partial_t\beta(\epsilon,t)}{1+\epsilon}(m-1)v\big[G(e^{(m-1)G(0)(\beta(\epsilon,t)+t_0)}v)-G(e^{(m-1)G(0)(t+t_0)}\frac{v}{1+\epsilon})\big],
\end{aligned}
\end{equation}
where the fifth line is derived by \eqref{est6} and the last inequality is due to \eqref{est4}.
 Furthermore, it holds by \eqref{est5} and the monotonic decreasing property of $G$ that
 \begin{equation}\label{est8}
 G(e^{(m-1)G(0)(\beta(\epsilon,t)+t_0)}v)-G(e^{(m-1)G(0)(t+t_0)}\frac{v}{1+\epsilon})\leq 0.
 \end{equation}

 Therefore, we conclude from \eqref{est7}-\eqref{est8} that
 \begin{equation*}
 \mathcal{L}_1\big(\frac{1}{1+\epsilon}v((1+\epsilon)x,\beta(\epsilon,t)+t_0)\big)\leq 0,
 \end{equation*}
and the proof is completed.
\end{proof}

Based on Lemma~\ref{subs2}, the next two preliminary lemmas are ready to verify that the initial data at any time $t_0>T$  with a sufficiently large $T>T_0$ satisfies
\begin{equation*}
\frac{1}{1+\epsilon}v((1+\epsilon)x,t_0)\leq v(x,t_0),\quad x\in \mathbb{R}^n.
\end{equation*}
\begin{lemma}\label{l2est}
Under the assumptions as in Theorem~\ref{t3}, let $(\rho,P)$ be a weak solution to the Cauchy problem~\eqref{de}-\eqref{pe}. Then, for $ R_0:=\inf\{R>0:supp(P_0)\subset B_{R}(0)\}$ and
$T_0:=\inf\{t>0:B_{R_0}(0)\subset supp\big({P(t)\big)}\}$, there exists $T_1>T_0$ such that it holds
\begin{equation}\label{res1}
\frac{\rho_M}{2}\leq\rho(x,t)\leq\rho_H,\quad (x,t)\in B_{R_0+2}\times[T_1,\infty),
\end{equation}
\begin{equation}\label{res3}
\begin{aligned}
\|\nabla\rho\|_{L^2(B_{R_0+2}\times[t,t+1])}^2\leq \gamma(t), \quad t\geq T_1,
 \end{aligned}
 \end{equation}
where $\gamma(t)\to0$, as $t\to\infty$.  Moreover, there exists $T_2>T_1$ such that
\begin{equation}\label{res2}
\||\nabla\rho|^{s+1}\|_{L^2(B_{R_0+1}\times[t+\frac{1}{4},t+1])}^2\leq C(T_1,s,R_0)\|\nabla\rho\|_{L^2(B_{R_0+2}\times[t,t+1])}^2
\end{equation}for  $s\in\mathbb{N_+}\text{ and }t\geq T_2$.
\end{lemma}
\begin{proof}
\emph{Step 1.} Thanks to the assumption on the large time behavior, there exists a sufficiently large time $T_1>T_0$ such that \eqref{res1} holds.

 We rewrite the density equation~\eqref{de} as
 \begin{equation*}
 \partial_t\rho=\rho(\Delta P+G(P))+\nabla\rho\cdot\nabla P.
 \end{equation*}
 Inserting the Aronson-B\'enilan estimate (Lemma~\ref{lAB}) into the above equation, it follows
 \begin{equation*}
 \partial_t\rho\geq -\frac{\rho}{(m-1)t}+\nabla\rho\cdot\nabla P.
 \end{equation*}
 We integrate the above inequality on the time interval $[t,t+1]$ and get
 \begin{equation*}
 \begin{aligned}
 \int_{t}^{t+1}\nabla\rho\cdot\nabla Pds\leq \rho(t+1)-\rho(t)+\frac{\rho_M}{m-1}\int_{t}^{t+1}\frac{1}{s}ds.
 \end{aligned}
 \end{equation*}
 Taking both $(x,t)\in B_{R_0+2}\times[T_1,\infty)$ and Proposition~\ref{lts} into consideration, we obtain
 \begin{equation*}
 \begin{aligned}
 \int_{t}^{t+1}\hskip-6pt\int_{B_{R_0+2}}|\nabla\rho(x,s)|^2dxds\leq \|\rho(t+1)-\rho(t)\|_{L^\infty(B_{R_0+2})}
 +\frac{2^{m-1}|B_{R_0+2}|}{m(m-1)\rho_M^{m-3}}\log \big[\frac{t+1}{t}\big].
 \end{aligned}
 \end{equation*}
 Thus, \eqref{res3} holds if we set
 \begin{equation}\label{gamma}
 \begin{aligned}
 \gamma(t):=\|\rho(t+1)-\rho(t)\|_{L^\infty(B_{R_0+2})}
 +\frac{2^{m-1}|B_{R_0+2}|}{m(m-1)\rho_M^{m-3}}\log \big[\frac{t+1}{t}\big]\to0,\quad \text{as }t\to\infty.
 \end{aligned}
 \end{equation}

\noindent\emph{Step 2.}The porous medium type reaction-diffusion equation~\eqref{de} is rewritten as
\begin{equation}\label{dde}
\partial_t\rho=\nabla\cdot(m\rho^{m-1}\nabla\rho)+\rho G(P).
\end{equation}

Let $s\in\mathbb{N_+}$ and $\zeta\in C_{0}^{\infty}(B_{R_0+2}\times(t,t+1))$ be smooth test function with $\zeta=1$ in $B_{R_0+1}\times[t+\frac{1}{4},t+1]$. We multiply all terms in~\eqref{dde} by $\nabla\cdot(|\nabla\rho|^{2s}\zeta^2\nabla\rho)$ and integrate on $B_{R_0+2}$, then we get
\begin{equation}\label{l21}
\begin{aligned}
\int_{B_{R_0+2}}&\partial_t\rho\nabla\cdot(|\nabla\rho|^{2s}\zeta^2\nabla\rho)dx\\
=&\frac{-1}{2s+2}\frac{d}{dt}\int_{B_{R_0+2}}|\nabla\rho|^{2s+2}\zeta^2dx\underbrace{+\frac{1}{s+1}\int_{B_{R_0+2}}|\nabla\rho|^{2s+2}\zeta\zeta_tdx}_{\uppercase\expandafter{\romannumeral1}},\\
\end{aligned}
\end{equation}
\begin{equation}\label{l22}
\begin{aligned}
\int_{B_{R_0+2}}&\nabla\cdot(m\rho^{m-1}\nabla\rho)\nabla\cdot(|\nabla\rho|^{2s}\zeta^2\nabla\rho)=\int_{B_{R_0+2}}(m\rho^{m-1}\rho_{x_k})_{x_l}(|\nabla\rho|^{2s}\zeta^2\rho_{x_l})_{x_k}dx\\
=&\int_{B_{R_0+2}}\big((m-1)m\rho^{m-2}\rho_{x_l}\rho_{x_k}+m\rho^{m-1}\rho_{x_lx_k}\big)\\
&\times\big(|\nabla\rho|^{2s}\zeta^2\rho_{x_kx_l}+2|\nabla\rho|^{2s}\zeta\zeta_{x_k}\rho_{x_l}+2s|\nabla\rho|^{2s-2}\zeta^2\rho_{x_l}\rho_{x_j}\rho_{x_lx_j}\big)dx\\
=&\int_{B_{R_0+2}}m\rho^{m-1}|\nabla\rho|^{2s}\zeta^2(\rho_{x_lx_k})^2dx+2sm\int_{B_{R_0+2}}\rho^{m-1}|\nabla\rho|^{2s-2}\zeta^2\sum_{j=1}^{n}(\rho_{x_k}\rho_{x_kx_j})^2dx\\
&+\underbrace{\int_{B_{R_0+2}}2m\rho^{m-1}|\nabla\rho|^{2s}\rho_{x_kx_l}\zeta\zeta_{x_k}\rho_{x_l}dx}_{\uppercase\expandafter{\romannumeral2}}\\
&+\underbrace{\int_{B_{R_0+2}}(2s+1)(m-1)m\rho^{m-2}\rho_{x_l}\rho_{x_k}|\nabla\rho|^{2s}\zeta^2\rho_{x_kx_l}dx}_{\uppercase\expandafter{\romannumeral3}}\\
&+\underbrace{\int_{B_{R_0+2}}2(m-1)m\rho^{m-2}\rho_{x_l}\rho_{x_k}|\nabla\rho|^{2s}\zeta\zeta_{x_k}\rho_{x_l}dx}_{\uppercase\expandafter{\romannumeral4}},
\end{aligned}
\end{equation}
and 
\begin{equation}\label{l23}
\int_{B_{R_0+2}}\rho G(P)\nabla\cdot(|\nabla\rho|^{2s}\zeta^2\nabla\rho)dx=-\underbrace{\int_{B_{R_0+2}}(G(p)+m\rho^{m-2}G'(P))|\nabla\rho|^{2s+2}\zeta^2dx}_{\uppercase\expandafter{\romannumeral5}}.
\end{equation}
Combining \eqref{l21}-\eqref{l23}, we obtain
\begin{equation}\label{ls4}
\begin{aligned}
\frac{1}{2s+2}&\frac{d}{dt}\int_{B_{R_0+2}}|\nabla\rho|^{2s+2}\zeta^2dx+\int_{B_{R_0+2}}m\rho^{m-1}|\nabla\rho|^{2s}\zeta^2(\rho_{x_lx_k})^2dx\\
&+2sm\int_{B_{R_0+2}}\rho^{m-1}|\nabla\rho|^{2s-2}\zeta^2\sum_{j=1}^{n}(\rho_{x_k}\rho_{x_kx_j})^2dx=\uppercase\expandafter{\romannumeral1}+\uppercase\expandafter{\romannumeral2}+\uppercase\expandafter{\romannumeral3}+\uppercase\expandafter{\romannumeral4}+\uppercase\expandafter{\romannumeral5}.
\end{aligned}
\end{equation}
By means of Young's inequality, we have
\begin{equation}\label{ls5}
\begin{aligned}
\uppercase\expandafter{\romannumeral2}\leq\frac{m}{4}\int_{B_{R_0+2}}\rho^{m-1}(\rho_{x_kx_l})^2|\nabla\rho|^{2s}\zeta^2dx+4m\int_{B_{R_0+2}}\rho^{m-1}|\nabla\rho|^{2s+2}|\nabla\zeta|^2dx,
\end{aligned}
\end{equation}
\begin{equation}\label{ls6}
\begin{aligned}
\uppercase\expandafter{\romannumeral3}&=- \int_{B_{R_0+2}}(2s+1)m(m-1)\rho^{\frac{m-1}{2}}\rho^{\frac{m-3}{2}}\rho_{x_kx_l}\rho_{x_k}\rho_{x_l}\zeta^2|\nabla \rho|^{2s}dx\\
\leq \frac{m}{4}&\int_{B_{R_0+2}}\rho^{m-1}(\rho_{x_kx_l})^2|\nabla \rho|^{2s}\zeta^2dx+(2s+1)^2\int_{B_{R_0+2}}m(m-1)^2\rho^{m-3}|\nabla \rho|^{2s+4}\zeta^2dx,
\end{aligned}
\end{equation}
and
\begin{equation}\label{ls7}
\begin{aligned}
\uppercase\expandafter{\romannumeral4}=&-2\int_{B_{R_0+2}}(m-1)m\rho^{m-2}\rho_{x_k}2|\nabla\rho|^{2s+2}\zeta\zeta_{x_k}dx\\
\leq&\int_{B_{R_0+2}}m(m-1)\rho^{m-2}|\nabla\rho|^{2s+4}\zeta^2dx+\int_{B_{R_0+2}}m(m-1)\rho^{m-2}|\nabla\rho|^{2s+2}|\nabla\zeta|^2dx.
\end{aligned}
\end{equation}
Taking \eqref{ls5}-\eqref{ls7} into \eqref{ls4}, it follows
\begin{equation}\label{l29}
\begin{aligned}
\frac{1}{2s+2}&\frac{d}{dt}\int_{B_{R_0+2}}|\nabla\rho|^{2s+2}\zeta^2dx+\frac{m}{2}\int_{B_{R_0+2}}\rho^{m-1}|\nabla\rho|^{2s}\zeta^2(\rho_{x_lx_k})^2dx\\
&+2sm\int_{B_{R_0+2}}\rho^{m-1}|\nabla\rho|^{2s-2}\zeta^2\sum_{j=1}^{n}(\rho_{x_k}\rho_{x_kx_j})^2dx\\
\leq&\int_{B_{R_0+2}}\big[m(2s+1)^2(m-1)^2\rho^{m-3}+m(m-1)\rho^{m-2}\big]|\nabla\rho|^{2s+4}\zeta^2dx\\
+&\int_{B_{R_0+2}}\big[\frac{\zeta|\zeta_t|}{s+1}+m(m-1)\rho^{m-2}|\nabla\zeta|^2+\big(G(P)+m\rho^{m-2}G'(P)\big)\zeta^2\big]|\nabla\rho|^{2s+2}dx.
\end{aligned}
\end{equation}
Let $C(\rho_M,s):=4m(2s+1)^2(m-1)^2\rho_M^{m-3}+2m(m-1)\rho_M^{m-2}$ and $C(G(0),\rho_M,s):=\frac{1}{s+1}+2m(m-1)\rho_M^{m-2}+G(0)$, it follows from~\eqref{l29} and the fact $G'(P)\leq 0$ for $P\in(0,P_H)$  that
\begin{equation}\label{l210}
\begin{aligned}
&\frac{1}{2s+2}\frac{d}{dt}\int_{B_{R_0+2}}|\nabla\rho|^{2s+2}\zeta^2dx+\frac{m\rho_M^{m-1}}{2^{m-1}}\int_{B_{R_0+2}}|\nabla\rho|^{2s}\zeta^2(\rho_{x_lx_k})^2dx\leq C(\rho_M,s)\times\\
&\int_{B_{R_0+2}}|\nabla\rho|^{2s+4}\zeta^2dx+C(G(0),\rho_M,s)\int_{B_{R_0+2}}(\zeta|\zeta_t|+|\nabla\zeta|^2+\zeta^2)|\nabla\rho|^{2s+2}dx.
\end{aligned}
\end{equation}
We estimate the term $\int_{B_{R_0+2}}|\nabla\rho|^{2s+4}\zeta^2dx$ as follows:
\begin{equation*}
\begin{aligned}
\int_{B_{R_0+2}}&|\nabla\rho|^{2s+4}\zeta^2dx=-\int_{B_{R_0+2}}|\nabla\rho|^{2s+2}\rho_{x_l}(\rho_M-\rho)_{x_l}\zeta^2dx\\
=&\int_{B_{R_0+2}}(2s+2)|\nabla\rho|^{2s}\rho_{x_j}\rho_{x_lx_j}\rho_{x_l}(\rho_M-\rho)\zeta^2dx+\int_{B_{R_0+2}}|\nabla\rho|^{2s+2}\rho_{x_lx_l}(\rho_M-\rho)\zeta^2dx\\
&+2\int_{B_{R_0+2}}|\nabla\rho|^{2s+2}\rho_{x_l}\zeta_{x_l}(\rho_M-\rho)\zeta dx\\
\leq& \frac{3}{4}\int_{B_{R_0+2}}|\nabla\rho|^{2s+4}\zeta^2dx+(2s+2)^2\int_{B_{R_0+2}}|\nabla\rho|^{2s}(\rho_{x_jx_l})^2\zeta^2(\rho_M-\rho)^2dx\\
&+\int_{B_{R_0+2}}|\nabla\rho|^{2s}(\rho_{x_lx_l})^2(\rho_M-\rho)^2\zeta^2dx+4\int_{B_{R_0+2}}|\nabla\rho|^{2s+2}|\nabla\zeta|^2(\rho_M-\rho)^2dx.
\end{aligned}
\end{equation*}
Hence, we have
\begin{equation}\label{l211}
\begin{aligned}
&\int_{B_{R_0+2}}|\nabla\rho|^{2s+4}\zeta^2dx\leq 4(2s+2)^2\int_{B_{R_0+2}}|\nabla\rho|^{2s}(\rho_{x_jx_l})^2\zeta^2(\rho_M-\rho)^2dx\\
&+4\int_{B_{R_0+2}}|\nabla\rho|^{2s}(\rho_{x_lx_l})^2(\rho_M-\rho)^2\zeta^2dx+16\int_{B_{R_0+2}}|\nabla\rho|^{2s+2}|\nabla\zeta|^2(\rho_M-\rho)^2dx.
\end{aligned}
\end{equation}
Inserting \eqref{l210} into \eqref{l211} and integrating on $(t,t+1)$ for $t>T_1$ determined later, it follows
\begin{equation*}
\begin{aligned}
&\int_{t}^{t+1}\hskip-8pt\int_{B_{R_0+2}}|\nabla\rho|^{2s+4}\zeta^2dxd\sigma\leq \gamma'(t)\frac{C(\rho_M,s)2^{m+1}(2s+3)}{m\rho_M^{m-1}}\int_{t}^{t+1}\hskip-8pt\int_{B_{R_0+2}}|\nabla\rho|^{2s+4}\zeta^2dxd\sigma\\
&+\gamma'(t)\frac{C(G(0),\rho_M,s)2^{m+1}(2s+3)}{m\rho_M^{m-1}}\int_{t}^{t+1}\hskip-8pt\int_{B_{R_0+2}}(\zeta|\zeta_t|+|\nabla\zeta|^2+\zeta^2)|\nabla\rho|^{2s+2}dxd\sigma\\
&+16\rho_M^2\int_{t}^{t+1}\hskip-8pt\int_{B_{R_0+2}}|\nabla\rho|^{2s+2}|\nabla\zeta|^2dxd\sigma,
\end{aligned}
\end{equation*}
where $\gamma'(t)=\|\rho_M-\rho\|_{L^{\infty}(B_{R_0+2}\times(t,t+1))}^2$ and $\gamma'(t)\to 0$ as $t\to\infty$ resulting from  Proposition~\ref{lta}. We choose $T_2>T_1>0$ such that $\gamma'(t)\frac{C(\rho_M,s)2^{m+1}(2s+3)}{m\rho_M^{m-1}}\leq \frac{1}{2}$ for $t\geq T_2$, thus it concludes
\begin{equation*}
\begin{aligned}
\int_{t}^{t+1}&\hskip-8pt\int_{B_{R_0+2}}|\nabla\rho|^{2s+4}\zeta^2dxd\sigma\\
\leq& (32\rho_M^2+\gamma'(t)\frac{C(G(0),\rho_M,s)2^{m+2}(2s+3)}{m\rho_M^{m-1}})\int_{t}^{t+1}\hskip-8pt\int_{B_{R_0+2}}(\zeta|\zeta_t|+|\nabla\zeta|^2+\zeta^2)|\nabla\rho|^{2s+2}dxd\sigma\\
\leq& (32\rho_M^2+\frac{2C(G(0),\rho_M,s)}{C(\rho_M,s)})\int_{t}^{t+1}\hskip-8pt\int_{B_{R_0+2}}(\zeta|\zeta_t|+|\nabla\zeta|^2+\zeta^2)|\nabla\rho|^{2s+2}dxd\sigma,
\end{aligned}
\end{equation*}
so \eqref{res2} holds if we  choose suitable test function $\zeta$ and make $s$-step iterations. The proof is completed.
\end{proof}

Let $u:=\rho_M-\rho$, then $u$ satisfies the following equation
\begin{equation}\label{ude}
\partial_tu=\nabla\cdot(m\rho^{m-1}\nabla u)-\rho G(P).
\end{equation}
In the following, we use De Giorgi's iteration technique  to obtain the $L^\infty$-estimate of the density gradient $\nabla \rho$ in a cylindrical region.
\begin{lemma}\label{lip}
Under the initial assumptions of Lemma~\ref{l2est}, let $u:=\rho_M-\rho$ be the weak solution solution to the Cauchy problem for~\eqref{ude}. Then, for any piecewise smooth test function $\zeta$ satisfying
\begin{equation*}
0\leq\zeta\leq1 \text{ and }supp(\zeta)\subset B_{r}(x_0)\times(t,t+1) \text{ with } |B_r(x_0)|\leq |B_1|\text{ and } B_r\subset B_{R_0+1}\quad\text{for } t  \geq T_2,
\end{equation*}
there exist a universal constant $C>0$ such that, it holds for $k=1,...,n$ that
\begin{equation}\label{lipres1}
\begin{aligned}
\sup\limits_{s\in[t,t+1]}&\|(u_{x_k}-l)_+\zeta\|_{L^2(B_r)}^2+\int_{t}^{t+1}\hskip-8pt\int_{B_r}|\nabla(u_{x_k}-l)_+|^2\zeta^2dxds\\
\leq& C\int_{t}^{t+1}\hskip-8pt\int_{B_r}(u_{x_k}-l)_+^2(|\nabla\zeta|^2+\zeta|\zeta_t|)dxds+\gamma_1(t)(\int_{t}^{t+1}\hskip-8pt(\int_{A_{l}(s)}\zeta(x,s)dx)^{\frac{r}{q}}ds)^{\frac{2(1+\chi)}{r}},
\end{aligned}
\end{equation}
where $r=q=2+\frac{4}{n}$ satisfies the relationship $\frac{1}{r}+\frac{n}{2q}=\frac{n}{4}$, $A_{l}(s):=B_r\cap \{u(\cdot,s)>l\}$ for $s\in(t,t+1)$, $\chi=\frac{1}{n}$, and $\gamma_1(t)\to0$ as $t\to\infty$. Furthermore, let $\gamma_2(t):=2\sqrt{\gamma_1(t)}$, then there exists $T_3\geq T_2$ such that 
\begin{equation}\label{lipres2}
\|\nabla \rho\|_{L^\infty(B_{R_0+\frac{1}{2}}\times[t+\frac{1}{2},t+1])}\leq \gamma_2(t)\quad\text{for }t\geq T_3.
\end{equation}
\end{lemma}
\begin{proof}
Let $u_k:=u_{x_k}:=\partial_{x_k}u$, we take spatial partial derivative $\partial_{x_k}$ for $k=1,...,n$ action on~\eqref{ude} and obtain
\begin{equation*}
\partial_t u_k=\nabla\cdot(m\rho^{m-1}\nabla u_k)-\nabla\cdot(m(m-1)\rho^{m-2}u_k\nabla u)-(\rho G(P))_{x_k}.
\end{equation*}
Multiplying above equation by $(u_k-l)_+\zeta^2$ and integrating on $B_{r}$, it holds
\begin{equation}\label{lip11}
\begin{aligned}
\frac{1}{2}\frac{d}{dt}&\int_{B_r}(u_k-l)_+^2\zeta^2dx-\int_{B_r}(u_k-l)_+^2\zeta\zeta_tdx\\
=&-\int_{B_r}m\rho^{m-1}|\nabla(u_k-l)_+|^2\zeta^2dx-2\int_{B_{r}}m\rho^{m-1}\nabla(u_k-l)_+\cdot\nabla\zeta (u_k-l)_+\zeta dx\\
&-\int_{B_r}m(m-1)\rho^{m-2}u_k\nabla u\cdot\nabla (u_k-l)_+\zeta^2dx+\int_{B_r}\rho G(P) (u_k-l)_{+x_k}\zeta^2dx\\
&-2\int_{B_r}m(m-1)\rho^{m-2}u_k\nabla u\cdot\nabla\zeta (u_k-l)_+\zeta dx+2\int_{B_r}\rho G(P)\zeta_{x_k} (u_k-l)_+\zeta dx.
\end{aligned}
\end{equation}
By means of direct computations and H\"older's inequality, we have
\begin{equation}\label{lip22}
\begin{aligned}
-2\int_{B_{r}}&m\rho^{m-1}\nabla(u_k-l)_+\cdot\nabla\zeta (u_k-l)_+\zeta dx\\
\leq& \frac{1}{4}\int_{B_r}m\rho^{m-1}|\nabla(u_k-l)_+|^2\zeta^2dx+4\int_{B_r}m\rho^{m-1}(u_k-l)_+^2|\nabla\zeta|^2dx,
\end{aligned}
\end{equation}
\begin{equation}\label{lip3}
\begin{aligned}
-\int_{B_r}&m(m-1)\rho^{m-2}u_k\nabla u\cdot\nabla (u_k-l)_+\zeta^2dx\\
\leq& \frac{1}{4}\int_{B_r}m\rho^{m-1}|\nabla(u_k-l)_+|^2\zeta^2dx+\int_{A_l(t)}m(m-1)^2\rho^{m-3}|\nabla u|^4\zeta^2dx,
\end{aligned}
\end{equation}
\begin{equation}\label{lip4}
\begin{aligned}
-2\int_{B_r}&m(m-1)\rho^{m-2}u_k\nabla u\cdot\nabla\zeta (u_k-l)_+\zeta dx\\
\leq& \int_{B_r}m^2(m-1)^2\rho^{2m-2}|\nabla u|^4\zeta^2dx+\int_{B_r}(u_k-l)_+^2|\nabla\zeta|^2dx,
\end{aligned}
\end{equation}
\begin{equation}\label{lip5}
\int_{B_r}\rho G(P) (u_k-l)_{+x_k}\zeta^2dx\leq \frac{1}{4}\int_{B_r}m\rho^{m-1}|\nabla(u_k-l)_+|^2\zeta^2dx+\int_{A_l(t)}\frac{\rho^2G^2(P)\zeta^2}{m\rho^{m-1}}dx,
\end{equation}
and 
\begin{equation}\label{lip6}
2\int_{B_r}\rho G(P)\zeta_{x_k}(u_k-l)_+\zeta dx\leq \int_{B_r}(u_k-l)_+^2|\nabla\zeta|^2dx+\int_{A_l(t)}\rho^2G^{2}(P)\zeta^2dx.
\end{equation}
Taking \eqref{lip22}-\eqref{lip6} into~\eqref{lip11}, we obtain
\begin{equation}\label{lip7}
\begin{aligned}
\frac{1}{2}\frac{d}{dt}\int_{B_r}&(u_k-l)_+^2\zeta^2dx+\frac{1}{4}\int_{B_r}m\rho^{m-1}|\nabla(u_k-l)_+|^2\zeta^2dx\\
\leq&(2+4m\rho_M^{m-1})\int_{B_r}(u_k-l)_+^2(|\nabla\zeta|^2+\zeta|\zeta_t|)dx\\
&+2\|G(P(t))\|_{L^\infty(B_{R_0+1})}^2\max\{\frac{\rho_M^22^{m-1}}{m\rho_M^{m-1}},\rho_M^2\}\int_{A_l(t)}\zeta^2dx\\
&+2 m(m-1)^2\max\{4\rho_M^{m-3},m\rho_M^{2m-2}\}\int_{A_l(t)}|\nabla u|^4\zeta^2dx.
\end{aligned}
\end{equation}
We integrate~\eqref{lip7} on $[t+\frac{1}{4},t+1]$ and get
\begin{equation}\label{lip8}
\begin{aligned}
\sup\limits_{t\in[t+\frac{1}{4},t+1]}&\|(u_k-l)_+\zeta\|_{L^2(B_{r})}^2+\|\nabla(u_k-l)_+\zeta\|_{L^2(B_r\times[t,t+1])}^2\\
\leq& C(\rho_M,m)\int_{t+\frac{1}{4}}^{t+1}\hskip-6pt\int_{B_r}(u_k-l)_+^2(|\nabla\zeta|^2+\zeta|\zeta_t|)dxds+\gamma_3(t)\int_{t+\frac{1}{4}}^{t+1}\hskip-6pt\int_{A_l(s)}\zeta^2dxds\\
&+C(\rho_M,m)\int_{t+\frac{1}{4}}^{t+1}\hskip-6pt\int_{A_l(s)}|\nabla u|^4\zeta^2dxds,
\end{aligned}
\end{equation}
where $\gamma_3(t):=2\|G(P)\|_{L^\infty(B_{R_0+1}\times[t,t+1])}^2\max\{\frac{\rho_M^22^{m-1}}{m\rho_M^{m-1}},\rho_M^2\}\to0$ as $t\to\infty$.

Using H\"older's inequality again,  we find
\begin{equation}\label{lip9}
\begin{aligned}
\gamma_3(t)\int_{t+\frac{1}{4}}^{t+1}\hskip-6pt\int_{A_l(s)}\zeta^2dxds&\leq \gamma_3(t)|B_r\times[t+\frac{1}{4},t+1]|^{\frac{1}{n+2}}\big(\int_{t+\frac{1}{4}}^{t+1}\hskip-6pt\int_{A_l(s)}\zeta^{\frac{2(n+2)}{n+1}}dxds\big)^{\frac{n+1}{n+2}}\\
&\leq \gamma_3(t)|B_r\times[t+\frac{1}{4},t+1]|^{\frac{1}{n+2}}\big(\int_{t+\frac{1}{4}}^{t+1}\hskip-6pt(\int_{A_l(s)}\zeta dx)^{\frac{r}{q}}ds\big)^{\frac{2(1+\chi)}{r}},
\end{aligned}
\end{equation}
\begin{equation}\label{lip10}
\begin{aligned}
C(\rho_M,m)&\int_{t+\frac{1}{4}}^{t+1}\hskip-6pt\int_{A_l(s)}|\nabla u|^4\zeta^2dxds\\
\leq& C(\rho_M,m)\big(\int_{t+\frac{1}{4}}^{t+1}\hskip-6pt\int_{B_r}|\nabla u|^{4n+8}dxds\big)^{\frac{1}{n+2}}\big(\int_{t+\frac{1}{4}}^{t+1}\hskip-6pt\int_{A_l(s)}\zeta^{\frac{2(n+2)}{n+1}}dxds\big)^{\frac{n+1}{n+2}}\\
\leq&\gamma_4(t)\big(\int_{t+\frac{1}{4}}^{t+1}\hskip-6pt(\int_{A_l(s)}\zeta dx)^{\frac{r}{q}}ds\big)^{\frac{2(1+\chi)}{r}},
\end{aligned}
\end{equation}
where $r=q=2+\frac{4}{n}$ satisfies $\frac{1}{r}+\frac{n}{2q}=\frac{n}{4}$ and $\chi=\frac{1}{n}$, and $\gamma_4(t):=C(\rho_M,m)\big(\int_{t+\frac{1}{4}}^{t+1}\int_{B_r}|\nabla u|^{4n+8}dxds\big)^{\frac{1}{n+2}}$ $\to0$ as $t\to\infty$ thanks to~\eqref{res3}-\eqref{res2}.

Inserting~\eqref{lip9}-\eqref{lip10} into~\eqref{lip8}, and set $\gamma_1(t):=\max\{\gamma_3(t)|B_r\times[t,t+1]|^{\frac{1}{n+2}},\gamma_4(t)\}$, then $\gamma_1(t)\to0$ as $t\to\infty$ and \eqref{lipres1} holds. On account of~\cite[Chapter \uppercase\expandafter{\romannumeral3}, Lemma 6.1]{LSU1988}, there exists $T_3\geq T_2$ such that, if we set $l=\sqrt{\gamma_1(t)}(2-\frac{1}{2^N})$ for $N\in\mathbb{N}$, it follows from~\eqref{lipres1} that
\begin{equation*}
\|\nabla \rho\|_{L^\infty(B_{\frac{r}{2}}\times[t+\frac{1}{2},t+1])}\leq 2\sqrt{\gamma_1(t)},\quad t\geq T_3,
\end{equation*}
so~\eqref{lipres2} holds for $t\geq T_3$. The proof is completed.
\end{proof}

\begin{remark} Indeed,  based on the time asymptotic behaviors of Theorme~\ref{t1}, we can investigate the time convergence rate of Lipschitz constant of the density (or the pressure) in any fixed ball when the time is large. From Lemma~\ref{lip}, we know, for any $R_0>0$, $$\|\nabla P(t)\|_{L^\infty(B_{R_0})}\sim\|\nabla \rho(t)\|_{L^\infty(B_{R_0})}\leq \sqrt{\gamma_1(t)}$$  with  \begin{equation}\label{aaa}\gamma_1(t):=\max\{\gamma_3(t)|B_r\times[t,t+1]|^{\frac{1}{n+2}},\gamma_4(t)\},\quad t\gg1.\end{equation} In the proof process of Lemma~\ref{lip} and Lemma~\ref{l2est}, one can learn 
$$\gamma_4(t):=C(\rho_M,m)\big(\int_{t+\frac{1}{4}}^{t+1}\int_{B_r}|\nabla u|^{4n+8}dxds\big)^{\frac{1}{n+2}}\leq [\gamma(t)]^{\frac{1}{n+2}}$$ with $$ \gamma(t):=\|\rho(t+1)-\rho(t)\|_{L^\infty(B_{R_0+2})}
 +\frac{2^{m-1}|B_{R_0+2}|}{m(m-1)\rho_M^{m-3}}\log \big[\frac{t+1}{t}\big],\quad t\gg1,$$ and $$\gamma_3(t):=2\|G(P)\|_{L^\infty(B_{R_0+1}\times[t,t+1])}^2\max\{\frac{\rho_M^22^{m-1}}{m\rho_M^{m-1}},\rho_M^2\},\quad t\gg1.$$
By means of the time convergence rates of Theorem~\ref{t1} for both $G(P):=P_M-P$ from the tumor growth model~\cite{PERTHAME2014} and $g(\rho)=1-\rho$  from the Fisher-KPP equation with $m\geq2$~\cite{DQZ2020}, we have 
$\|\rho(t)-\rho_M\|_{L^\infty(B(R_0+2))}\leq C(1+t)^{-1}$ for $t\gg1$, which yields
$$\gamma_3(t)\leq C(1+t)^{-2}\text{ and }\gamma_4(t)\leq [\gamma(t)]^{\frac{1}{n+2}}\leq C(1+t)^{-1}+C[\log(1+\frac{1}{t})]^{\frac{1}{n+2}}\leq C(1+t)^{-\frac{1}{n+2}},\quad t\gg1.$$ Hence, inserting the above two results on the decay rate into \eqref{aaa}, there exists a constant $C>0$ depends on $P_M,\alpha_0,\beta_0,m,d_G,R_0$ as in Theorem \ref{t1} such that 
$$\|\nabla \rho(t)\|_{L^\infty(B_{R_0})}\leq C(1+t)^{-\frac{1}{2(n+2)}},\quad t\gg 1.$$
 
\end{remark}

In the next two lemmas,  we  will show Lipschitz's continuity of the pressure for the Cauchy problem~\eqref{de} after some time $T_0>0$.
\begin{lemma}\label{lip1}
Let $\rho$ be the weak solution to the Cauchy problem~\eqref{de}, then there exists $T_4>T_3$ such that $P$ is Lipschitz continuous in $\mathbb{R}^n\times(T_4+\delta,\infty)$ for any $\delta\in(0,1)$. More precisely, it holds for some constant $C>0$ depending on $G(0),\delta,T_4,m$ that
\begin{align}
&|\nabla P(x,t)|\leq C\max\{1,|x|\},\label{reslip1}\\
&|\partial_t P(x,t)|\leq C\max\{1,|x|^2\}\label{reslip2},\quad (x,t)\in\mathbb{R}^n\times(T_4+\delta,\infty).
\end{align}
\end{lemma}
\begin{proof}
Set $v(x,t):=e^{-(m-1)G(0)t}P(x,t)$, we first claim that there exists $T_4\geq T_3$ such that

\begin{equation}\label{ini2}
\frac{1}{1+\epsilon}v((1+\epsilon)x,t_0)\leq v(x,t_0)\quad\text{ for }x\in\mathbb{R}^n,\ t_0\geq T_4.
\end{equation}
By the Alexander reflection principle (Lemma~\ref{conemon}), it holds for any $|x|\geq R_0$ and $t_0\geq T_0$ that
\begin{equation}\label{init11}
\begin{aligned}
\frac{1}{1+\epsilon}v((1+\epsilon)x,t_0)\leq v((1+\epsilon)x,t_0)\leq v(x,t_0).
\end{aligned}
\end{equation}
For $|x|\leq R_0$ and small $\epsilon>0$, it holds for $t_0\geq T_3$ and by Lemma~\ref{lip} that
\begin{equation}\label{init2}
\begin{aligned}
\frac{1}{1+\epsilon}&v((1+\epsilon)x,t_0)-v(x,t_0)\\
\leq&e^{-(m-1)G(0)t_0}[-\frac{\epsilon}{1+\epsilon}P((1+\epsilon)x,t_0)+P((1+\epsilon)x,t_0)-P(x,t_0)]\\
\leq&e^{-(m-1)G(0)t_0}[-\frac{\epsilon}{2}\frac{m}{m-1}(\frac{\rho_M}{2})^{m-1}+\sup\limits_{|x|\leq R_0+\delta}|\nabla P(x,t_0)|\epsilon R_0]\\
\leq &e^{-(m-1)G(0)t_0}\epsilon[-\frac{1}{2}\frac{m}{m-1}(\frac{\rho_M}{2})^{m-1}+\gamma_2(t_0)2m\rho_{M}^{m-2}R_0].
\end{aligned}
\end{equation}
If we choose a large time $T_4>T_3$ on account of Lemma~\ref{lip} such that
 \begin{equation}\label{init3}
 -\frac{1}{2}\frac{m}{m-1}(\frac{\rho_M}{2})^{m-1}+\gamma_2(t_0)2m\rho_{M}^{m-2}R_0\leq 0\quad\text{ for }t_0\geq T_4.
 \end{equation}
Combining \eqref{init11}-\eqref{init3}, \eqref{ini2} holds.

Thanks to Lemma~\ref{subs2}, it holds by the comparison principle and the definition of $v$ that
\begin{equation*}
e^{-(m-1)G(0)(t_0+\beta(\epsilon,t))}\frac{1}{1+\epsilon}P((1+\epsilon)x,\beta(\epsilon,t)+t_0)\leq e^{-(m-1)G(0)(t+t_0)}P(x,t+t_0)
\end{equation*} for sufficiently small $\delta>0$ and $0\leq \epsilon<\delta$. 
Also because
\begin{equation*}
e^{-(m-1)G(0)(t_0+\beta(\epsilon,t))}\frac{1}{1+\epsilon}P((1+\epsilon)x,\beta(\epsilon,t)+t_0)\big|_{\epsilon=0}= e^{-(m-1)G(0)(t+t_0)}P(x,t+t_0),
\end{equation*}
which means
\begin{equation}\label{td1}
\frac{d}{d\epsilon}\big|_{\epsilon=0}\big[e^{-(m-1)G(0)(t_0+\beta(\epsilon,t))}\frac{1}{1+\epsilon}P((1+\epsilon)x,\beta(\epsilon,t)+t_0)\big]\leq0.
\end{equation}
Thanks to the definition of $\beta(\epsilon,t)$ in Lemma~\ref{subs2}, it holds
\begin{equation}\label{td2}
\frac{d}{d\epsilon}\big|_{\epsilon=0}\beta(\epsilon,t)=\frac{d}{d\epsilon}\big|_{\epsilon=0}\big\{\frac{\log\big[1+(e^{(m-1)G(0)t}-1)(1+\epsilon)\big]}{(m-1)G(0)}\big\}=\frac{1-e^{-(m-1)G(0)t}}{(m-1)G(0)}.
\end{equation}
Combining~\eqref{td1} and~\eqref{td2}, we have
\begin{equation*}
\begin{aligned}
&e^{-(m-1)G(0)(t+t_0)}\big[\frac{e^{-(m-1)G(0)t}-1}{(m-1)G(0)}P(x,t+t_0)-P(x,t+t_0)+\nabla P\cdot x\\
&+\frac{1-e^{-(m-1)G(0)t}}{(m-1)G(0)}\partial_tP(x,t+t_0)\big]\leq 0.
\end{aligned}
\end{equation*}
It holds after rearrangement for $t\geq0$ that
\begin{equation}\label{tdepsi}
\begin{aligned}
&\partial_tP(x,t+t_0)\leq \frac{(m-1)G(0)\big(|\nabla P(x,t+t_0)||x|+P(x,t)\big)}{1-e^{-(m-1)G(0)t}}+P(x,t).
\end{aligned}
\end{equation}

Because of the pressure equation
\begin{equation*}\label{pe1}
\partial_t P=(m-1)P(\Delta P+G(P))+\nabla P\cdot\nabla P
\end{equation*}
 and the Aronson-B\'enilan estimate (Lemma~\ref{lAB})
\begin{equation*}\label{abe}
\Delta P+G(P)\geq -\frac{1}{(m-1)t},
\end{equation*}
we obtain
\begin{equation}\label{peab}
\partial_t P(x,t+t_0)\geq -\frac{P(x,t+t_0)}{t+t_0}+|\nabla P(x,t+t_0)|^2.
\end{equation}
Combining \eqref{tdepsi} and \eqref{peab}, it follows
\begin{equation*}
\begin{aligned}
&|\nabla P(x,t+t_0)|^2-\frac{(m-1)G(0)|x|}{1-e^{-(m-1)G(0)t}}|\nabla P(x,t+t_0)|-P(x,t+t_0)\\
&\times\big(\frac{1}{t+t_0}+1+\frac{G(0)(m-1)}{1-e^{-(m-1)G(0)t}}\big)\leq0.
\end{aligned}
\end{equation*}
Hence, we have
\begin{equation}\label{gradiente}
\begin{aligned}
&|\nabla P(x,t+t_0)|\leq \frac{(m-1)G(0)|x|}{2-2e^{-(m-1)G(0)t}}+\big\{\frac{G^2(0)(m-1)^2|x|^2}{4(1-e^{-(m-1)G(0)t})^2}\\
&+\big[\frac{1}{t+t_0}+1+\frac{G(0)(m-1)} {1-e^{-(m-1)G(0)t}}\big]P(x,t+t_0)\big\}^{\frac{1}{2}}.
\end{aligned}
\end{equation}
From \eqref{peab}, we directly get
\begin{equation}\label{ltl}
\partial_t P(x,t+t_0)\geq -\frac{P(x,t+t_0)}{t+t_0}.
\end{equation}
In addition, combining \eqref{tdepsi} and \eqref{gradiente}, it concludes
\begin{equation}\label{tds}
\begin{aligned}
&\partial_t P(x,t+t_0)\leq P(x,t+t_0)(1+\frac{(m-1)G(0)}{1-e^{-(m-1)G(0)t}})\\
&+\frac{(m-1)^2G^2(0)|x|^2}{2(1-e^{-(m-1)G(0)t})^2}+\frac{(m-1)G(0)|x|}{1-e^{-(m-1)G(0)t}}\times
\big\{\frac{(m-1)^2G^2(0)|x|^2}{4(1-e^{-(m-1)G(0)t})^2}\\
&+\big[\frac{1}{t+t_0}+1+\frac{G(0)(m-1)}{1-e^{-(m-1)G(0)t}}\big]P(x,t+t_0)\big\}^{\frac{1}{2}}.
\end{aligned}
\end{equation}

In one word, \eqref{gradiente} supports~\eqref{reslip1} and~\eqref{ltl}-\eqref{tds} imply that~\eqref{reslip2} holds. The proof  is completed.
\end{proof}
\begin{lemma}\label{lip2} Let $\rho$ be the weak solution to the Cauchy problem \eqref{de} with the additional assumption $G(\cdot)\in \mathcal{C}^1(0,P_H]$. Then $P$ is Lipschitz continuous in $\mathbb{R}^n\times[T_0+\delta, T_3+1]$ for any $\delta>0$. To be more precise, there exists a constant $C>0$ depending on $\delta,G(0),m,P_M,R_0,T_0,T_4$ such that
\begin{equation*}
|\nabla P(x,t)|\leq C\max\{1,|x|\},\quad |\partial_tP(x,t)|\leq C\max\{1,|x|^2\},\quad (x,t)\in\mathbb{R}^n\times[T_0+\delta,T_4+1].
\end{equation*}
\end{lemma}
\begin{proof}
Set $R_{T_0}^{\delta}:=\sup\{R>0, B_R\subset supp(P(\cdot,T_0+\frac{\delta}{8}))\}>R_0$, thanks to Theorem~\ref{vhc}, $\Omega(t):=\{x: P(\cdot,t)\}$ strictly expands as the time increases. Therefore, $B_{R_{T_0}^\delta}$ is strictly contained in $\Omega(t)$ for $t\in[T_0+\frac{\delta}{4},T_3+1]$.
In other words, there exist $0<c'<C'<\infty$ such that
\begin{equation*}
c'\leq P(x,t)<C'\text{ for }(x,t)\in B_{R_{T_0}^\delta}\times[T_0+\frac{\delta}{2},T_4+1].
\end{equation*}
On the account of interior uniform regularity estimate for the uniform parabolic equation~\eqref{de}, we obtain
\begin{equation*}
|\nabla P|, |\partial_tP|\leq C_1\text{ for }(x,t)\in B_{R_0}\times[T_0+\frac{\delta}{2},T_4+1].
\end{equation*}

Next, we consider the domain $\{x:|x|\geq  R_0\}\times[T_0+\frac{\delta}{2},T_4+1]$ and~$v(x,t):=P(x,t)e^{-(m-1)G(0)t}$. Similar to \eqref{init11} by Alexander reflection principle, it holds for $|x|\geq R_0$ that
\begin{equation*}
\frac{1}{1+\epsilon}v\big((1+\epsilon)x,t'\big)\leq v(x,t')\ \text{ for any }t'\in[T_0+\frac{\delta}{2},T_4+1].
\end{equation*}
On the boundary $\{x: |x|=R_0\}$ (different from Lemma~\ref{lip1}), we have
\begin{equation*}
\begin{aligned}
\frac{1}{1+\epsilon}&v((1+\epsilon)x,t'+\beta(\epsilon,t))-v(x,t'+t)\\
=&e^{-(m-1)G(0)t'}[e^{-(m-1)G(0)\beta(\epsilon,t)}\frac{1}{1+\epsilon}P((1+\epsilon)x,\beta(\epsilon,t)+t')-e^{-(m-1)G(0)t}P(x,t+t')]\\
\leq&e^{-(m-1)G(0)(t'+t)}[-\frac{\epsilon}{1+\epsilon}P((1+\epsilon)x,\beta(\epsilon,t)+t')+P((1+\epsilon)x,\beta(\epsilon,t)+t')\\
&-P(x,\beta(\epsilon,t)+t')+P(x,\beta(\epsilon,t)+t')-P(x,t+t')]\ (\text{cone monocinity})\\
\leq &e^{-(m-1)G(0)(t'+t)}[-\frac{\epsilon c}{2}+P(x,\beta(\epsilon,t)+t')-P(x,t+t')]\\
\leq&e^{-(m-1)G(0)(t'+t)}[-\frac{\epsilon c}{2}+|\partial_t P(x,\beta(\epsilon,t)\theta_1+t')|(\beta(\epsilon,t)-t)]\\
\leq &e^{-(m-1)G(0)(t'+t)}\big[-\frac{\epsilon c}{2}+C_1\epsilon\frac{e^{(m-1)G(0)t}-1}{(m-1)G(0)[e^{(m-1)G(0)t}+(e^{(m-1)G(0)t}-1)\theta_2\epsilon]}\big]\\
\leq&e^{-(m-1)G(0)(t'+t)}\epsilon[-\frac{ c'}{2}+C_1 t]\leq 0
\end{aligned}
\end{equation*} for $0<t\leq2\delta':=\min\{\frac{c'}{2C_1},\frac{\delta}{2}\}$,
in which $\theta_1,\theta_2\in[0,1]$ .

From Lemma~\ref{subs2} and by means of the comparison principle, it holds
\begin{equation*}
\frac{1}{1+\epsilon}v((1+\epsilon)x,t'+\beta(\epsilon,t))\leq v(x,t'+t),\quad |x|\geq R_0.
\end{equation*}
Same arguments as \eqref{gradiente}-\eqref{tds} in Lemma~\ref{lip1}, it concludes
\begin{equation*}
\begin{aligned}
&|\nabla P(x,t+t')|\leq \frac{(m-1)G(0)|x|}{2-2e^{-(m-1)G(0)\delta'}}+\big\{\frac{G^2(0)(m-1)^2|x|^2}{4(1-e^{-(m-1)G(0)\delta'})^2}\\
&+\big[\frac{1}{\delta'+t'}+1+\frac{G(0)(m-1)} {1-e^{-(m-1)G(0)\delta'}}\big]P(x,t+t')\big\}^{\frac{1}{2}}.
\end{aligned}
\end{equation*}
\begin{equation*}
\partial_t P(x,t+t')\geq -\frac{1}{t+t'},
\end{equation*}
and
\begin{equation*}
\begin{aligned}
&\partial_t P(x,t+t')\leq P(x,t+t')(1+\frac{(m-1)G(0)}{1-e^{-(m-1)G(0)\delta'}})\\
&+\frac{(m-1)^2G^2(0)|x|^2}{2(1-e^{-(m-1)G(0)\delta'})^2}+\frac{(m-1)G(0)|x|}{1-e^{-(m-1)G(0)\delta'}}\times
\big\{\frac{(m-1)^2G^2(0)|x|^2}{4(1-e^{-(m-1)G(0)\delta'})^2}\\
&+\big[\frac{1}{\delta'+t'}+1+\frac{G(0)(m-1)}{1-e^{-2(m-1)G(0)\delta'}}\big]P(x,t+t')\big\}^{\frac{1}{2}}.
\end{aligned}
\end{equation*}
for $x\in\mathbb{R}^n$, $\delta'\leq t\leq 2\delta'$, and $t'\in[T_0+\frac{\delta}{2},T_4+1]$. Therefore, there exists a general constant $C>0$ depending on $\delta,G(0),m,R_0,T_0,T_4, c', C'$ such that
\begin{equation*}
|\nabla P(x,t)|\leq C\max\{1,|x|\},\quad |\partial_tP(x,t)|\leq C\max\{1,|x|^2\}\quad \text{for }(x,t)\in\mathbb{R}^n\times[T_0+\delta,T_3+1],
\end{equation*}
and the proof is completed.
\end{proof}
 In the end of this section, we give the proof of Theorem~\ref{t3}.
\begin{proof}[\underline{\textbf{Proof of Theorem~\ref{t3}}}]
By direct observations,  Lemmas~\ref{lip1}-\ref{lip2} support Lipschitz's continuity of the pressure  in Theorem~\ref{t3}.  
\end{proof}
\section{$C^{1,\alpha}$ regularity of free boundary}\label{fbr}

We mainly follow the approach of~\cite{CW1990,KZ2021} to prove the $C^{1,\alpha}$ regularity of the free boundary $\partial\{(x,t): \rho(x,t)>0\}$ for the Cauchy problem~\eqref{de} of the porous medium type reaction-diffusion equation.  The differences lie in the logistic growth effect caused by the Lotka-Volterra source term $\rho g(\rho)$ from~\eqref{de}. 

Some new notations are needful in the rest of this section as follows:
\begin{definition}
\begin{equation*}
\oint_{B_{R}(x_0)}w(x)dx=\frac{1}{|B_{R}(x_0)|}\int_{B_{R}(x_0)}w(x)dx,
\end{equation*}
\begin{equation*}
\Omega:=\{(x,t)\in\mathbb{R}^n\times \mathbb{R}_+: \rho(x,t)>0\},\quad \Omega(t):=\{x\in\mathbb{R}^n: \rho(x,t)>0\},
\end{equation*}
\begin{equation*}
\Gamma:=\partial \Omega,\quad \Gamma(t):=\partial \Omega(t),
\end{equation*}
\begin{equation*}
A_{r}(x,t):=B_{r}(x)\times[t-r,t+r],\quad (x,t)\in\mathbb{R}^n\times \mathbb{R}_+.
\end{equation*}
\begin{definition}
Due to the AB estimate (Lemma~\ref{lAB}), there exists $C_0>0$ such that, for $t\geq \frac{T_0}{4}>0$,
\begin{equation*}
\Delta P+G(P)>-C_0,
\end{equation*} 
where $T_0$ is given by~\eqref{T_0}. If we set $C_0'=C_0+G(0)$, it holds for $t\geq \frac{T_0}{4}>0$ that
\begin{equation*}
\Delta P\geq -C_0'.
\end{equation*}
\end{definition}
\begin{definition}
Let us define the spatial cone of the $\mu$-direction as follows:
\begin{equation*}
W_{\theta,\mu}=\{y\in\mathbb{R}^n\ |\ |\frac{y}{|y|}-\mu|\leq 2\sin\frac{\theta}{2}\}\
\end{equation*}
with any axis $\mu\in\mathcal{S}^{n-1}$ and $\theta\in(0,\frac{\pi}{2}]$. If $P(\cdot,t)$ is non-increasing along directions in $W_{\theta,\mu}$, we say that $P(\cdot,t)$ is monotone with respect to $W_{\theta,\mu}$.
\end{definition}
\end{definition}
\begin{definition}
In the rest of this section, $\sigma>0$ represents the universal constant.
\end{definition}
 \paragraph{\large Fundamental lemmas.}
 Lemma \ref{l18} tells us  a fact that a considerable amount of  population has entered the ball $B_{R}(x_0)$ in the time $t_0+\sigma$
 when there was no population in the ball $B_{R}(x_0)$ at time $t_0$, then the population must reach the ball $B_{R/6}(x_0)$ at time $t_0+\sigma$.
 Lemma \ref{l19} says that a round area congest enough population, then the center of the round area must congest considerable population after a finite time.
 
 The detailed proofs of Lemmas~\ref{l18}-\ref{l19} are almost same as~\cite{CF1980,KZ2021},  we give in Appendix~\ref{app A} for completeness. 
 \begin{lemma}\label{l18}
 For any given $T_0>0$, there exist positive constants $\eta$, $c_0$ depending only on $m$, $n$, $T_0$ such that the following is true: Let $x_0\in\ \mathbb{R}^n$, $t_0\geq T_0$, $R>0$, $0<\sigma<\eta$. If
 \begin{equation*}
 P(x,t_0)\equiv 0,\quad x\in B_{R}(x_0)
 \end{equation*}
 and
 \begin{equation*}
 \oint_{B_{R}(x_0)}P(x,t_0+\sigma)dx\leq \frac{c_0R^2}{\sigma}
 \end{equation*}
 hold, then we have
 \begin{equation*}
 P(x,t_0+\sigma)\equiv 0,\quad x\in B_{R/6}(x_0).
 \end{equation*}
 \end{lemma}
\begin{lemma}\label{l19}
For any $t_0\geq T_0$ and $\nu>0$, there exist $\lambda, c, \eta>0$ depending on $\nu, T_0$ and universal constants such that the following holds: For any $R>0$ and $0< \sigma\leq \eta$, if
\begin{equation*}
\oint_{B_R(x_0)}\rho^m(x,t_0)dx\geq \nu\frac{R^2}{\sigma},
\end{equation*}
then it holds
\begin{equation}\label{a14}
\rho^m(x_0,t_0+\lambda\sigma)\geq c\frac{R^2}{\sigma},
\end{equation}
where
\begin{equation}\label{a22}
0<\varepsilon\lambda \leq 1,\ 0<\varepsilon^{\delta}\ll 1,\ 0<c\ll1,\ \lambda^m\nu^{m-1}\gg 1\quad\text{with }\varepsilon=\eta\max\{(m-1)C_0,\ C_0',1\}.
\end{equation}\end{lemma}

We show the relation between the free boundary and vertical line, where the vertical line of any $(x_0,t_0)\in \mathbb{R}^n\times (T_0,\infty)$ is defined by
\begin{equation*}
 \sigma(x_0,t_0)=\{(x,t):x=x_0,\ 0\leq t< t_0\}.
\end{equation*}
\begin{theorem}\label{vhc}
For a given point $(x_0,t_0)\in\Gamma$ and $t_0\geq T_0$, the following is true:
\begin{itemize}
\item[$(1)$] either $(\RNum{1})\ \sigma(x_0,t_0)\subset\Gamma$ or $(\RNum{2})\ \sigma(x_0,t_0)\cap\Gamma=\varnothing$ holds.
\item[$(2)$] If $(\RNum{2})$ holds, then there exist positive constants $\hat{C}, \beta, h$ such that for all $s\in(0,h)$
    \begin{align*}
    &\rho(x,t-s)=0\ \text{ if }\ |x-x_0|\leq \hat{C} s^{\beta};\\
    &\rho(x,t+s)>0\ \text{ if }\ |x-x_0|\leq \hat{C} s^{\beta}.
    \end{align*}
\end{itemize}
\end{theorem}
\begin{proof}
The proof of this lemma is a direct application of Lemmas \ref{l18}-\ref{l19} and is parallel to \cite[Theorems 3.1-3.2]{CF1980}, we omit detailed processes.
\end{proof}
\begin{remark} When the initial pressure statisfies \eqref{i1}-\eqref{i2}, the vertical line in Theorem \ref{vhc} is type-$\RNum{2}$ on account of~\cite{ACK1983}.
\end{remark}

\paragraph{\large Non-degeneracy near the free boundary.}

The behaviors of the pressure near the free boundary for the porous medium type reaction-diffusion equation~\eqref{de} is about to be analyzed. Similar to the PME~\eqref{PME}, the pressure~\eqref{de} is linearly degenerate near the free boundary after a finite time, which supports the fact that Lipschitz's continuity is the optimal (sharp) regularity for the pressure. We mainly argue by the inf-convolution technique as in \cite{KZ2021} to achieve our goal. The difference lies in dealing with differences caused by the logistic source term for the Cauchy problem~\eqref{de}.

We recall some basic properties of the inf-convolution of smooth functions. Let $\psi,h\in C^{\infty}(B_2)$ with $0<\psi<\frac{1}{2}$ and $h\geq0$. Define
\begin{equation}\label{inf-con}
f(x):=\inf\limits_{y\in B(x,\psi(x))}h(y),
\end{equation}
which is Lipschitz continuous in $B_1$. The following lemma, from \cite[Lemma 5.3]{KZ2021}, is concerned on Laplace's operator $\Delta$ working on a class of inf-convolution function.
\begin{lemma}\label{inf-del}
Let $h$ and $f$ be from~\eqref{inf-con}. Furthermore, suppose $\Delta h\geq -C$ for some constant  $C>0$ and $\|\nabla\psi\|_{L^\infty}\leq1$. Then,  there exist two constants $\sigma_1>0$ and $\sigma_2\geq3$ such that if $\psi>0$ satisfies
\begin{equation*}
\Delta\psi\geq\frac{\sigma_1|\nabla\psi|^2}{\psi}\quad\text{in }B_2,
\end{equation*}
it holds
\begin{equation*}
\Delta f(\cdot)-(1+\sigma_2\|\nabla\psi\|_{L^\infty})\Delta h(y(\cdot))\leq \sigma_2\|\nabla\psi\|_{L^\infty}C\quad\text{in }\mathcal{D}'(B_1),
\end{equation*}
where $y(\cdot)$ satisfies $f(\cdot)=h(y(\cdot))$ a.e. in $B_1$.
\end{lemma}

For the weak solution $\rho$ to the Cauchy problem~\eqref{de} in $A_2$, the smooth approximation $\{\rho_k\}_{k\in\mathbb{N}}$ of $\rho$ is needful as in \cite{KZ2021,PMEs}. 

Let $\rho_k$ be a weak solution to the following modified porous medium type reaction-diffusion equation~$\eqref{de}_1$:
\begin{equation}\label{mde}
\partial_t\rho_k=\Delta\rho_k^m+\rho_k g(\rho_k)\quad\text{for }(x,t)\in\mathbb{R}^n\times\mathbb{R}_+
\end{equation}
with the initial data
\begin{equation}\label{mini}
\rho_{k,0}=\max\{\rho_0,\frac{1}{k}\}.
\end{equation}
The corresponding pressure is given by 
\begin{equation*}
P_k=\frac{m}{m-1}\rho_k^{m-1},\quad (x,t)\in\mathbb{R}^n\times\mathbb{R}_+.
\end{equation*}

\begin{lemma}\label{smoothappro}
Let $\rho_k$ be the weak solution to the Cauchy problem~\eqref{mde}-\eqref{mini}. Under the initial assumption $0\leq \rho_0(x)\leq \rho_H$ for $x\in\mathbb{R}^n$, then it holds
\begin{equation*}
\frac{1}{k}\leq \rho_k(x,t)\leq \rho_H,\quad(x,t)\in\mathbb{R}^n\times\mathbb{R}_+,
\end{equation*}
where $k\geq K_M$ with $\frac{1}{K_M}\leq \frac{\rho_H}{2}$ and $P_H=\frac{m}{m-1}\rho_H^{m-1}$. Furthermore, $\rho_k$ is smooth in $\mathbb{R}^n\times \mathbb{R}_+$. In addition, after extracting the subsequence, $\{\rho_k\}_{k\geq K_N}$ uniformly locally converges to $\rho$ that is the weak solution to the Cauchy problem~\eqref{de}.
\end{lemma}
\begin{proof}By direct computations, we have
\begin{equation*}
\begin{aligned}
&\mathcal{L}\big(\frac{1}{k}\big)=-\frac{1}{k}g(\frac{1}{k})\leq0,\\
&\mathcal{L}\big(\rho_M\big)=\rho_H g(\rho_H)=0.
\end{aligned}
\end{equation*}
which implies that $\frac{1}{k}$ and $\rho_H$ are a weak sub-solution and a  weak sup-solution to~\eqref{mde} respectively.

Thanks to the initial assumption~\eqref{mini} and the fact that $\tilde{\mathcal{L}}(\rho_k)=0$, it concludes by the comparison principle that
\begin{equation*}
\frac{1}{k}\leq \rho_k(x,t)\leq \rho_H\quad\text{for }(x,t)\in\mathbb{R}^n\times\mathbb{R}_+.
\end{equation*}
Hence, by the standard interior uniformly parabolic estimate, $\rho_k$ is smooth in $\mathbb{R}^n\times(0,\infty)$. And after extracting the subsequence as in~\cite[Lemma 9.5]{PMEs} for the porous medium equation, $\{\rho_k\}_{k\geq K_N}$ uniformly locally converges to $\rho$, and $\rho$ is the weak solution to the Cauchy problem~\eqref{de}-\eqref{pe}, and the proof is completed.
\end{proof}

Let $\varphi:\ \mathbb{R}^n\to(0,\infty)$ be a smooth function and $\sigma_1, \sigma_2$ be from Lemma \ref{inf-del}. For some constants $N_0, M_0\geq 1$ to be determined later, we define
\begin{align}
&\omega_{k}(x,t)=e^{N_0\varepsilon t}\inf\limits_{y\in B(x, R_{\varepsilon}(x))}P_k(y+r\varepsilon \mu, P_{\varepsilon}(t)),\label{k11}\\
&\omega(x,t)=e^{N_0\varepsilon t}\inf\limits_{y\in B(x, R_{\varepsilon}(x))}P(y+r\varepsilon \mu, P_{\varepsilon}(t)),\label{k12}
\end{align}
and
\begin{equation}\label{k13}
\eta_{k}:=e^{N_1\varepsilon t}\inf\limits_{y\in B(x, R_{\varepsilon}(x))}\rho_k(y+r\varepsilon \mu, P_{\varepsilon}(t))\quad \text{with }N_1=\frac{N_0}{m-1},
\end{equation}
where
\begin{align}
&R_{\varepsilon}(x):=\varepsilon \varphi,\label{k14}\\
&P_{\varepsilon}(t):=(1+\sigma_2M_0\varepsilon)(\frac{e^{N_0\varepsilon t}-1}{N_0\varepsilon}).\label{k15}
\end{align}
Due to~\eqref{inf-con},
 $\omega_k$ is Lipschitz continuous and
\begin{equation}\label{etak}
\eta_{k}:=(\frac{m-1}{m}\omega_{k})^{1/(m-1)} \text{ is Lipschitz continuous }.
\end{equation}
Thus, it is sufficient to show that $\eta=(\frac{m-1}{m}\omega)^{1/(m-1)}$ is a sup-solution for $\mathcal{L}(f)=0$, where \begin{equation}\label{para-oper}
\mathcal{L}(f):=\partial_t f-\Delta f^m-fg(f).
\end{equation}
\begin{proposition}\label{p1}
Let $P_k$, $\rho_k$ be defined in Lemma~\ref{smoothappro}, and suppose that $P_k$ satisfies $\Delta P_k\geq -C_0':= -C_0-G(0)$ in $A_2$. Fix any $M_0\geq 1$ and consider $\varphi: \mathbb{R}^n\to \mathbb{R}$ such that

\begin{equation}\label{k16}
\begin{cases}
\Delta \varphi =\frac{\sigma_1|\nabla\varphi|^2}{\varphi},&\\
\frac{r}{M_0}\leq \varphi(\cdot)\leq rM_0,\quad\|\nabla\varphi\|_{L^{\infty}(\mathbb{R}^n)}\leq M_0\quad \text{for some }r\in (0,1).&\\
\end{cases}
\end{equation}
Then, there exist positive constants $N_0, \tau$ depending on $M_0$ and the universal constant such that for all $\varepsilon <\frac{1}{M_0}$, $\eta$  (given by \eqref{etak}) is a weak sup-solution for $\mathcal{L}$, i.e.,
\begin{equation*}
\mathcal{L}(\eta)\geq 0\quad\text{weakly in } B_{r}\times(0,\tau).
\end{equation*}
\end{proposition}
\begin{proof}

Below we estimate each term of $\mathcal{L}(\eta_k)$ in $B_r\times (0,\tau)$ by means of $\rho_k$. We begin with some preliminary estimates on $\eta_k$.

Since $\rho_k$ is smooth and positive, $P_k$ is also smooth and positive. From the definition of the inf-convolution, it follows that $\eta_k$ is Lipschitz continuous. Since $\Delta P_k\geq -C_0'$, direct computation yields
\begin{equation}\label{k17}
\Delta \rho_k^m\geq -C_0' \rho_k.
\end{equation}

Let us define the constant
\begin{equation}\label{k18}
N_0:=\sigma_3 M_0(1+C_0')
\end{equation}
for some constant $\sigma_3> \sigma_2$ to be determined later, and
\begin{equation}\label{k19}
\tau:=\min\big\{\frac{1}{2N_1},\frac{1}{2N_0},\frac{1}{\sigma_2M_0}, \frac{1}{(m-1)N_1\varepsilon}\log(\frac{N_1-\sigma_2M_0C_0'}{2\sigma_2M_0C_0'})\big\}.
\end{equation}

By the definition of $\eta_k$, there exists $z(x,t)$ satisfying
\begin{equation}\label{k20}
|z(x,t)-x|\leq |R_{\varepsilon}(x)|+r\varepsilon\leq 2M_0r\varepsilon
\end{equation}
such that
\begin{equation}\label{k21}
\eta_{k}(x,t)=\kappa(t)\rho_{k}(z(x,t),P_{\varepsilon}(t)),\quad \kappa(t):=e^{N_1\varepsilon t}.
\end{equation}

It follows from the definition of $P_{\varepsilon}(t)$ in \eqref{k15} that
\begin{align}
&P_{\varepsilon}'(t)=(1+\sigma_2M_0\varepsilon)\kappa^{m-1}(t)\label{k22}.
\end{align} For simplicity, $P_{\varepsilon},\eta_k$ denotes the values of them at $(x,t)$, and $\rho_k, \partial_t \rho_k, \Delta \rho_k^m$ denotes the values of them at $(z(x,t), P_{\varepsilon}(t))$.

 Since $\eta_k$ is Lipschitz continuous,  it holds by direct computations and in the sense of distribution that
 \begin{equation}\label{k24}
 \partial_{t}\eta_k\geq N_1\varepsilon \kappa(t) \rho_k+(P_{\varepsilon}')\kappa(t) \partial_t\rho_k\quad\text{in }B_{r}\times (0,\tau).
 \end{equation}
 Applying \eqref{k22} to \eqref{k24}, it holds
 \begin{equation}\label{k25}
  \partial_{t}\eta_k\geq N_1\varepsilon \eta_k+(1+\sigma_2M_0\varepsilon)\kappa^{m}(t)\partial_t\rho_k.
 \end{equation}

  Thanks to the definition of $\varphi$, we have  $\|R_{\varepsilon}\|_{L^{\infty}(\mathbb{R}^n)}\leq rM_0\varepsilon$ and $\|\nabla R_{\varepsilon}\|_{L^{\infty}(\mathbb{R}^n)}\leq M_0\varepsilon$ hold. We now apply Lemma \ref{inf-del} with $h=\rho_{k}^m(\cdot+r\varepsilon\mu,P_{\varepsilon})$ and $\psi=R_{\varepsilon}$. From \eqref{k17} i.e. $\Delta h\geq -C_0'\rho_k$, then the following holds:
 \begin{equation}\label{k26}
 \begin{aligned}
 -\Delta \eta_k^{m}&\geq -(1+\sigma_2\|R_{\varepsilon}\|_{L^{\infty}(\mathbb{R}^n)})\kappa^m(t)\Delta \rho_{k}^m(t)-\sigma_2\|\nabla R_{\varepsilon}\|_{L^{\infty}(\mathbb{R}^n)} C_0'\rho_k\\
 &\geq -(1+\sigma_2M_0\varepsilon)\kappa^m(t)\Delta \rho_k^m-\sigma_2M_0C_0'(1+2r\kappa^m(t))\varepsilon \rho_k
 \end{aligned}
 \end{equation}
 in the sense of distribution. 
Combining \eqref{k25}-\eqref{k26}, it concludes by~\eqref{k18}-\eqref{k19} and choosing large $\sigma_3>\sigma_2$ and small $\tau>0$ that
\begin{equation}\label{k27}
\begin{aligned}
\mathcal{L}(\eta_k)\geq& (N_1-\sigma_2 M_0C_0'-2\sigma_2 M_0C_0'\kappa^{m-1}(\tau))\varepsilon\eta_k\\
&+(1+\sigma_2M_0\varepsilon)\kappa^m(t)(\partial_t\rho_k-\Delta \rho_k^m)-\eta_kg(\eta_k) \\
\geq& (1+\sigma_2M_0\varepsilon)\kappa^m(t)\rho_k g(\rho_k)-\eta_k g(\eta_k):=\mathcal{A}.
\end{aligned}
\end{equation}
According to the monotonic  decreasing property of $g$, we have
\begin{equation}\label{k28}
\begin{aligned}
\mathcal{A}:&=(1+\sigma_2M_0\varepsilon)\kappa^m(t)\rho_k g(\rho_k)-\eta_k g(\eta_k)\\
&\geq \eta_k[(1+\sigma_2M_0\varepsilon)\kappa^{m-1}(t) g(\rho_k)- g(\eta_k)]\\
&\geq\eta_k[g(\rho_k)-g(\eta_k)]\geq 0.
\end{aligned}
\end{equation}
Therefore, we conclude by combining~\eqref{k27}-\eqref{k28} that
\begin{equation*}
\mathcal{L}(\eta_k)\geq 0\quad \text{in }B_{r}\times(0,\tau).
\end{equation*}

It is not hard to see that the choice of $N_0, \tau$ is independent of $r$ and $k$.   Since $\eta_k$ pointwisely converges to $\eta$ as $k$ tends to infinity, it holds in the weak sense that
\begin{equation*}
\mathcal{L}(\eta)\geq 0\quad \text{in }B_{r}\times(0,\tau),
\end{equation*}  and the proof is completed.
\end{proof}
We turn to prove the linear growth rate near the free (support) boundary of the pressure for the Cauchy problem~\eqref{de}.
\begin{lemma}\label{non-dege}
Let $\rho$ solve \eqref{de} in $A_2$ with $\Delta P\geq -C_0'$, there exist $C, \varepsilon_0>0$ both depending on $C_0', \theta, \hat{C}, h, \beta $, and the universal constants such that
\begin{equation}\label{k8}
P(x-\varepsilon\mu, t+C\varepsilon)>0\quad \text{for}\ (x,t)\in\Gamma\cap N_1\ \text{and }0<\varepsilon<\varepsilon_0.
\end{equation}
\end{lemma}
\begin{proof}
Let $\sigma_1$ be given in Lemma~\ref{inf-del}, and let $\Phi$ be the unique solution of
\begin{equation*}
\begin{cases}
\Delta (\Phi^{-\sigma_1+1})=0&\text{in }B_{1/2}\backslash B_{\sin\theta/10},\\
\Phi=A_{n,\theta}&\text{on }\partial B_{\sin\theta/10},\\
\Phi=\sin\theta& \text{on }\partial B_{1/2}.
\end{cases}
\end{equation*}
Here, $A_{n,\theta}$ is chosen sufficiently large such that
\begin{equation}\label{k29}
\Phi(y+\frac{\mu}{5})\geq 3\quad \text{for all }y\in B_{1/10}.
\end{equation}
Then, for some $M_0(\theta,n)\geq 1$,
\begin{equation*}
\frac{1}{M_0}\leq \Phi\leq M_0,\quad |\nabla\Phi|\leq M_0\quad\text{ in }B_{1/2}\backslash B_{\sin\theta/10}.
\end{equation*}

 Let $N_0,\tau$ be given in Proposition \ref{p1}, let $\hat{C}, h, \beta$ be given in Theorem \ref{vhc}, and fix any $(\hat{x},\hat{t})\in A_1\cap \Gamma$.

Let $\delta:=\delta(\theta, C_0)>0$ be determined later and define
\begin{align}
&t_{\delta}:=\min\{\tau, h, \delta\},\label{k30}\\
&r_{\delta}:=\min\{\hat{C}t_{\delta}^{\beta},\frac{1}{4}\}.\label{k31}
\end{align}
Due to Theorem \ref{vhc}, we have
\begin{equation}\label{k32}
\rho(x,\hat{t}-t_\delta)=0\quad\text{for }x\in B(\hat{x}, r_{\delta}).
\end{equation}
After translation, we assume $(\hat{x}, \hat{t}-t_\delta)=(0,0)$.
 Since $\mathcal{L}(\rho)=0$ weakly in $A_{1/2}$, where $\mathcal{L}$ is given by \eqref{para-oper}. It follows from \eqref{k32} that
 \begin{equation}\label{kk33}
 P(x,0)=0\quad \text{in }B_{r_\delta}.
 \end{equation}

  Let $E:=-\frac{r_\delta}{5}\mu$ and $\varphi:=r_{\delta}\Phi(\frac{x-E}{r_\delta})$, let $\omega$ be defined as in \eqref{k12} and $r:=r_\delta$,
 \begin{equation*}
 \omega(x,t):=e^{N_0\varepsilon t}\inf\limits_{B(x,\varepsilon\varphi)}P(y+r_\delta\varepsilon\mu, P_{\varepsilon}(t)).
 \end{equation*}
The cylindrical domain is denoted as
 \begin{equation}\label{k34}
 \Sigma:=(B(E,\frac{r_\delta}{2})\backslash B(E,r_{\theta}))\times[0,t_\delta],
 \end{equation}
 where $r_{\theta}:=\frac{r_\delta}{10}\sin\theta$.

 Our next aim is to show 
\begin{equation}\label{k35}
\omega\geq P\quad\text{in }\Sigma.
\end{equation}
To this end, it suffices to show that $\omega\geq v$ on the parabolic boundary of $\Sigma$ by means of the comparison principle.

  Due to  \eqref{kk33}, we have
  \begin{equation*}
\omega(x,0)\geq 0=P(x,0)\quad\text{ in }B(E,r_\delta/2).
  \end{equation*}
  Since $P(0,t_\delta)=0$ and due to Theorem~\ref{vhc}, we obtain
  \begin{equation*}
  P(0,t)=0\quad\text{for }t\in[0,t_\delta].
  \end{equation*}
Thanks to the cone monotonicity (Lemma~\ref{conemon}), it holds
\begin{equation*}
\omega\geq P=0 \quad\text{in }B(E,r_\theta)\times[0,t_\delta]\subset B(E,\frac{r_\delta}{5}\sin\theta)\times[0,t_\delta].
\end{equation*}
Hence to justify \eqref{k35}, it remains to show 
\begin{equation}\label{k41}
\omega\geq P\quad\text{ on }\partial B(E,\frac{r_\delta}{2})\times[0,t_\delta].
\end{equation}

By the definition of $\varphi$, we have $\varphi(x)=r_\delta\sin\theta$ on $\partial B(E,\frac{r_\delta}{2})$. From  \eqref{k18},  we have
\begin{equation}\label{k37}
\begin{aligned}
\omega(x,t)&= e^{N_0\varepsilon t}\inf\limits_{y\in B(x,\varepsilon\varphi)}P(y+r_\delta\varepsilon\mu, P_\varepsilon(t))\\
&=:e^{N_0\varepsilon t}V_1(x,t)\quad\text{on }\partial B(E, \frac{r_\delta}{2}).
\end{aligned}
\end{equation}
In view of the cone monotonicity (Lemma~\ref{conemon}), we have
\begin{equation}\label{k38}
\inf\limits_{y\in B(x,r_\delta\varepsilon\sin\theta)}P(y+r_\delta\varepsilon\mu,t)\geq P(x,t)\quad\text{in }A_1.
\end{equation}
It remains to show
\begin{equation}\label{k39}
e^{N_0\varepsilon t}V_1(x,\cdot)\geq \inf\limits_{B(\cdot,r_\delta\varepsilon\sin\theta)}P(y+r_\delta\varepsilon\mu,\cdot)\quad\text{on }\partial B(P,\frac{r_\delta}{2})\times[0,t_\delta].
\end{equation}

By the definition of $P_{\varepsilon}(t)$, we have
\begin{equation}\label{k40}
\begin{aligned}
P_\varepsilon(t)-t&=(1+\sigma_2M_0\varepsilon)(\frac{e^{N_0\varepsilon t}-1}{N_0\varepsilon})-t\\
&\leq (1+\sigma_2M_0\varepsilon)te^{ N_0\varepsilon t}-t\\
&\leq t[\sigma_2M_0\varepsilon+(1+\sigma_2M_0\varepsilon)(e^{N_0\varepsilon\delta}-1)]\\
&\leq t[\sigma_2M_0\varepsilon+(1+\sigma_2M_0\varepsilon)N_0\varepsilon\delta e^{N_0\varepsilon\delta}]\\
&=t\varepsilon[\sigma_2M_0+(1+\sigma_2M_0)N_0\delta e^{N_0\delta}].
\end{aligned}
\end{equation}

Due to Lemma~\ref{lAB}, there exits $C:=C_0'(m-1)$ such that
\begin{equation*}
\partial_t P\geq -CP\quad \text{for }(x,t)\in \mathbb{R}^n\times \mathbb{R}^+.
\end{equation*}
Then, it holds
\begin{equation*}
P(x,t)\geq e^{-C(t-t_0)}P(x,t_0)\quad \text{for any }t\geq t_0>0.
\end{equation*}
In view of \eqref{k37}, we drive
\begin{equation*}
\begin{aligned}
&\omega(x,t)\geq e^{N_0\varepsilon t}e^{-C(P_\varepsilon(t)-t)}\inf\limits_{B(x,r_\delta\varepsilon\sin\theta)}P(y+r_\delta\varepsilon\mu,t)\\
&\geq e^{\big(\sigma_3M_0(1+C_0)-C_0\sigma_2M_0(m-1)-(1+\sigma_2M_0)(m-1)N_0C_0\delta e^{N_0\delta}\big)\varepsilon t}\inf\limits_{B(x,r_\delta\varepsilon \sin\theta)}P(y+r_\delta\varepsilon\mu,t)
\end{aligned}
\end{equation*}
If we choose $\sigma_3>2C_0\sigma_2(m-1)$ in Proposition~\ref{p1}, there exists small $\delta'>0$ such that 
\begin{equation*}
\sigma_3M_0(1+C_0)-C_0\sigma_2M_0(m-1)-(1+\sigma_2M_0)(m-1)N_0C_0\delta' e^{N_0\delta'}\geq 0.
\end{equation*}
Hence, we fix $\delta=\delta'>0$ in~\eqref{k30} and take \eqref{k38} into consideration, then \eqref{k41} holds.

 We conclude with $\omega\geq P$ on the parabolic boundary of $\Sigma$, and the claim \eqref{k35} holds by the comparison principle.

   Using \eqref{k29}, we have
   \begin{equation*}
   \varphi(x)=r_{\delta}\Phi(\frac{x}{r_\delta}+\frac{\mu}{5})\geq 3r_\delta\quad \text{for } x\in B_{r_\delta/10}(0).
   \end{equation*}
From this, it follows
\begin{equation}\label{k36}
-r_\delta\varepsilon\mu\in B(x,\frac{12}{5}r_\delta\varepsilon)+r_\delta\varepsilon\mu\subset B(x,\varepsilon\varphi)+r_\delta\varepsilon\mu\text{ for all }|x|\leq \frac{r_\delta\varepsilon}{5}\leq \frac{r_\delta}{10}.
\end{equation}
Thanks to \eqref{k35} and the definition of $\omega$, we get for $|x|\leq \frac{r_{\delta}\varepsilon}{5}$,
\begin{align*}
e^{N_0\varepsilon t_\delta}P(-r_{\delta}\varepsilon\mu, P_{\varepsilon}(t_\delta))&\geq e^{N_0\varepsilon t_\delta}\inf\limits_{B(x,\varepsilon\varphi)}P(y+r_{\delta}\varepsilon\mu, P_{\varepsilon}(t_\delta))\\
&= w(x,t_\delta)\geq P(x,t_\delta).
\end{align*}
From \eqref{k15}, it follows that $P_{\varepsilon}=t_\delta+c\varepsilon$ for some $c=c(t_\delta, \delta)$ independent of $\varepsilon$. Thus
\begin{equation*}
P(-r_\delta\varepsilon\mu, t_\delta+c\varepsilon)\geq e^{-N_0\varepsilon t_\delta}\sup\limits_{|x|\leq r_\delta\varepsilon/5}P(x,t_\delta).
\end{equation*}
Recall that $(0,t_\delta)=(\hat{x},\hat{t})\in \Gamma$. We have
\begin{equation*}
P(-r_\delta\varepsilon\mu+\hat{x}, \hat{t}+c\varepsilon)>0.
\end{equation*}
Hence, it holds
\begin{equation*}
P(\cdot-r_\delta\varepsilon\mu,\cdot+c\varepsilon)>0\quad \text{on }\Gamma\cap A_1,
\end{equation*}
and the proof is completed.
\end{proof}
Inspired by~\cite{KZ2021}, the following proposition shows that the pressure $P$ for the Cauchy problem~\eqref{de} grows faster than a  linear growth near the free boundary.
\begin{proposition}\label{lindegs}
Under the initial assumptions of Lemma~\ref{non-dege}, there exists $\varepsilon_0,\ \kappa_*$ depending on $C_0',\ \theta,\ \hat{C},\ h$, and universal constants such that, for all $\varepsilon\in(0,\varepsilon_0)$, it holds
\begin{equation}\label{k5}
P(x+\varepsilon\mu,t)\geq \kappa_*\varepsilon\quad \text{for }(x,t)\in\Gamma\cap A_1.
\end{equation}
\end{proposition}
\begin{proof}
Let $c_0$ be from Lemma \ref{l18}, C be from Lemma~\ref{non-dege}, and set $\kappa:=\frac{c_0\sin^2\theta}{4C}$. We first prove that, for all $0<\varepsilon<\varepsilon_0$,
\begin{equation}\label{k1}
\sup\limits_{y\in B(x,\varepsilon)}P(y,t)\geq\kappa\varepsilon\quad \text{for all } (x,t)\in\Gamma\cap A_1.
\end{equation}

We argue by contradiction. Suppose that the above claim~\eqref{k1} is false. Then, for any $\varepsilon_0>0$, there exist $\varepsilon\in(0,\varepsilon_0]$ and $(\hat{x},\hat{t})\in\Gamma\cap Q_1$ such that \eqref{k1} fails.

Set $t_1:=\hat{t}-C\varepsilon$, since the positive set of $P$ is strictly expansive as the time increases (Theorem~\ref{vhc}), we have
\begin{equation*}
P(x,\hat{t})>0,\quad\text{for}\ x\in A_1\cap\Gamma_{t_1}.
\end{equation*}
Using the cone monotonicity (Lemma~\ref{conemon}) and the fact that $P(\hat{x},\hat{t})=0$, it follows
\begin{equation*}
(\hat{x}+\mathbb{R}^{+}\mu)\cap\Gamma_{t_1}\neq \varnothing.
 \end{equation*}
 Therefore, there exists $x_1\in \Gamma_{t_1}$ such that
\begin{equation}\label{k2}
x_1=\hat{x}+C_1\varepsilon \mu\quad \text{for some } C_1>0.
\end{equation}
Due to the cone monotonicity (Lemma~\ref{conemon}), we have
\begin{equation}\label{k3}
d(x_1-c\varepsilon\mu,\Gamma_{t_1})\geq \frac{c\varepsilon\sin\theta}{2}\quad\text{for all }c\geq0.
\end{equation}
In view of Lemma~\ref{non-dege}, it holds for sufficiently small $\varepsilon>0$ that
\begin{equation*}
P(x_1-\varepsilon\mu,t_1+C\varepsilon)>0.
\end{equation*}
Therefore, combining with the fact that
\begin{equation*}
P(x_1-\varepsilon C_1\mu,t_1+C\varepsilon)=P(\hat{x},\hat{t})=0,
\end{equation*}
we obtain
\begin{equation}\label{c1}
C_1>1.
\end{equation}

Using \eqref{k2}-\eqref{k3} and~\eqref{c1}, if $0<\varepsilon\leq\varepsilon_0$ is sufficiently small compared with $C$, it follows
\begin{equation*}
d(\hat{x},\Gamma_{t_1})\geq \frac{C_1\varepsilon\sin\theta}{2}\geq \frac{\varepsilon\sin\theta}{2}:=R,
\end{equation*}
which yields
\begin{equation}\label{kk3}
P(\cdot,t_1)=0\text{ in } B(\hat{x},R).
\end{equation}

Note that $t_1+C\varepsilon=\hat{t}$, the failure of \eqref{k1} implies
\begin{equation}\label{k4}
\oint_{B(\hat{x},R)}P(x,\hat{t})dx\leq \kappa\varepsilon=\frac{c_0R^2}{C\varepsilon},
\end{equation}
where we used $\kappa=\frac{c_0\sin^2\theta}{4C}$ in the last inequality.

With \eqref{kk3} and \eqref{k4} at hands, we are able to apply Lemma \ref{l18} and get
\begin{equation*}
P(x,\hat{t})=0\quad \text{in}\ B(\hat{x},R/6),
\end{equation*}
which contradicts with $(\hat{x},\hat{t})\in \Gamma$.  

We turn to justify that~\eqref{k1} holds. The left of proof is parallel to~\cite[Corollary 5.7]{KZ2021} and we sketch the proof. Let us take small  $\gamma\in(0,1)$ depending only on $\theta$ such that $B(\mu,\gamma)\subset W_{\theta,\mu}$.
Fix any $(x,t)\in\Gamma\cap A_1$ and set $\kappa_*:=\kappa\gamma$, we conclude as in~\cite[Corollary 5.7]{KZ2021} that
\begin{equation*}
P(x+\varepsilon\mu,t)\geq\kappa_*\varepsilon\quad \text{for any }(x,t)\in \Gamma\cap A_1 \text{ and } \varepsilon\in(0,\varepsilon_0],
\end{equation*}
and the proof is completed.
\end{proof}
In the following lemma, we give the linear degeneracy  estimate of the weak solution to the Cauchy problem~\eqref{de} near the free boundary. The $n+1$-dimensional unit vector is given by 
\begin{equation*}
\hat{\mu}:=\frac{1}{\sqrt{2}}<\mu,e_{n+1}>,
\end{equation*}
where $e_{n+1}$ is the unit vector of $(n+1)th$-direction.
\begin{lemma}[Degeneracy rate]\label{deg-rate}Under the initial assumptions of Lemma~\ref{non-dege}, we set $(0,0)\in\Gamma$ without loss of generality, then $\Gamma\cap A_1$ is a n-dimensional Lipschitz continuous surface. Furthermore, there exist $0<\delta_2<\frac{1}{2}$ and $c_2>0$ such that
\begin{equation}\label{lin-degenerate}
\hat{\nabla}_{\hat{\mu}} P\geq c_2\quad\text{in }A_{\delta_2}\cap \Omega,
\end{equation}
where $\hat{\nabla}:=<\partial_t,\nabla>$.
\end{lemma}
\begin{proof}
 
 \emph{Step~1.}  On account of Lipschitz's continuity (Lemmas~\ref{lip1}-\ref{lip2}) and being parallel to \cite[Lemma 6.2]{KZ2021}, it is easy to prove that $\Gamma\cap A_1$ is a $n$-dimensional locally time-space Lipschitz continuous surface.

\noindent\emph{Step~2.} We turn to \eqref{lin-degenerate} and first show that there exist $ c_1>0$ such that
\begin{equation}\label{space-lin-deg}
\nabla_{\mu} P\geq c_1\quad\text{in }A_{\delta_1}\cap \Omega.
\end{equation}

The approach mainly comes from~\cite[Thorem 6]{CVW1987}, that is to use the method of Harnack's inequality of the uniform parabolic equation in the positive set. 

 Let $C$ be the Lipschitz constant of the surface $\Gamma \cap A_{\frac{1}{2}}$ and $C_1=C+1$.  We pick 
$(\hat{x},\hat{t})\in \Omega\cap A_{\delta_1}$ with $d(\hat{x},\Gamma_{\hat{t}})=h<\delta_1$ and small $\delta_1$ being determined later.  Set $s=\hat{t}-h$, then it holds 
\begin{equation*}
d(\hat{x},\Gamma_{\hat{t}-h})\leq Ch.
\end{equation*}
Furthermore, there exists $(y,s)\in\Gamma $ such that $d(\hat{x},\Gamma_{s})=d(\hat{x},y)\leq Ch$, which means that $B(y,h)\subset B(\hat{x},C_1h)$.

Since $\Gamma_t\cap B_{\delta_1}$  for spatial variable is a Lipschitz's surface in the $\mu$-direction (Lemma~\ref{conemon}), $|\partial B(y,h)\cap \{P>0\}|_{n-1}\geq C'h^{n-1}$ with $C'>0$ independent of $h$.  Thanks to the cone monotonicity in $W_{\theta,\mu}$ and Proposition~\ref{lindegs}, it holds by the divergence theorem that there exists $\delta_1>0$ depending only $\varepsilon_0$ such that
\begin{equation*}
\oint_{B(y,h)\cap \{P>0\}}\nabla_{\mu} P(x,s)dx\geq\frac{\sigma}{h}\oint_{\partial B(y,h)\cap\{P>0\}} P(x,s)\mu\cdot \nu_{x}dx\geq \kappa,
\end{equation*}
where $\kappa$ only depends on $\kappa_*$ and $C_1$,  and $\sigma$ is the universal constant. 

Let \begin{equation*}
\Omega^{r}:=\{(x,t)\in \Omega,\ d((x,t),\partial\Omega)\geq r\},
\end{equation*}
 thus there exists $\gamma\in(0,1)$ only depending on $\kappa$ such that 
\begin{equation*}
\oint_{B(y,h)\cap \Omega^{\gamma h}}\nabla_{\mu}P(x,s)dx\geq\frac{\kappa}{2}.
\end{equation*}
Therefore, it holds by the  mean value theorem of integrals  that there exists 
\begin{equation*}
z\in B(y,h)\cap \Omega^{\gamma h}\subset B(\hat{x},C_1 h)\cap \Omega^{\gamma h} 
\end{equation*}
such that 
\begin{equation*}
\nabla_{\mu} P(z,s)\geq \frac{\kappa}{2}.
\end{equation*}

Due to the cone monotonicity (Lemma~\ref{conemon}), $\phi=\nabla_{\mu}P\geq0$ in $\Omega\cap A_1$ holds. We take the spatial directional derivative operator $\nabla_{\mu}$ action on the pressure equation~\eqref{pe} and obtain
\begin{equation}\label{harnack1}
\begin{aligned}
\partial_t\phi&=(m-1)\phi(\Delta P+G(P))+(m-1)P\Delta\phi+(m-1)PG'(P)\phi+2\nabla P\cdot\nabla\phi\\
&\geq (m-1)P\Delta\phi+2\nabla P\cdot\nabla\phi-(m-1)(C_0+K_G)\phi,
\end{aligned}
\end{equation}
where $K_G$ is defined by~\eqref{KG}. By means of~\eqref{harnack1}, we define a new variable $\tilde{\phi}:=e^{(m-1)(C_0+K_G)t}\phi$ which satisfies the following uniform parabolic inequality:
\begin{equation}\label{harnack2}
\partial_t\tilde{\phi}\geq (m-1)P\Delta\tilde{\phi}+2\nabla P\cdot\nabla\tilde{\phi}.
\end{equation}

 We define 
 \begin{align*}
 &\Sigma_1^{h}:=\Omega^{\gamma h}\cap(B(\hat{x},C_1h)\cap (-h+\hat{t},\hat{t})),\\
 &\Sigma_2^{h}:=\Omega^{\frac{\gamma h}{2}}\cap(B(\hat{x},2C_1h)\cap (-2h+\hat{t},\hat{t})),
 \end{align*}
 it is easy to verify 
 \begin{align*}
 &(\hat{x},\hat{t}), (z,s)\in \Sigma_1^h\subset \Sigma_2^h,\\
 & P(x,t)\geq \frac{\kappa_*\gamma h}{2}\quad\text{for any }(x,t)\in \Sigma_2^h,
 \end{align*}
  where the second line is derived by~\eqref{k5}.
  
 Set $\omega(x,t):=\tilde{\phi}(\hat{x}+hx,s+ht)$,  it follows form $h=\hat{t}-s$ that 
 \begin{equation*}
 \omega(0,1)=\tilde{\phi}(\hat{x},\hat{t})\quad \text{and }\omega(z',0)=\tilde{\phi}(z,s)\quad\text{for }z':=\frac{z-\hat{x}}{h}.
 \end{equation*}
 Denote
\begin{equation*}
\Sigma_1:=\frac{\Sigma_1^h-(\hat{x},s)}{h},\quad \Sigma_2:=\frac{\Sigma_2^h-(\hat{x},s)}{h},
\end{equation*}
$\Sigma_1,\Sigma_2$ are Lipschitz's regions with Lipschitz constant depending only on $C,\sigma$.
 
 In addition, we conclude 
 \begin{align*}
 &(0,1),(z',0)\in \Sigma_1\subset\Sigma_2,\\
 &\Sigma_1+B_{\frac{\gamma}{2}}\subset \Sigma_2.
 \end{align*} 

From~\eqref{harnack2}, we have
\begin{equation*}
\partial_t\omega \geq (m-1)P_h\Delta\omega+2\nabla P_h\cdot\nabla\omega\quad\text{in }\Sigma_2,
\end{equation*}
 where 
 \begin{equation*}
 P_h(x,t)=\frac{P(\hat{x}+hx,s+ht)}{h}\geq \frac{\kappa_*\gamma}{2}\quad\text{in }\Sigma_2.
 \end{equation*}
 Due to Harnack's inequality,  there exists $c$ dependning only on $\kappa_*,\theta,C_1$ such that 
 \begin{equation*}
 \omega(0,1)\geq c \omega(z',0).
 \end{equation*}
By the definition of $\omega$, we obtian
\begin{equation*}
\tilde{\phi}(\hat{x},\hat{t})\geq c\tilde{\phi}(z,s).
\end{equation*}
Hence, by the definition of~$\tilde{\phi}$, it follows
\begin{equation*}
\phi(\hat{x},\hat{t})\geq ce^{(m-1)(C_0+K_G)(s-\hat{t})}\phi(z,s)\geq \frac{c\kappa e^{-(m-1)(C_0+K_G)}}{2}.
\end{equation*}

In conclusion, we set $c_1:=\frac{c\kappa e^{-(m-1)(C_0+K_G)}}{2}$, \eqref{space-lin-deg} holds.

\noindent\emph{Step 3.} We turn to the lower bound of the time derivative near the boundary.
Due to the pressure equation~\eqref{pe}, we have 
\begin{equation}\label{time-lin-degenerate}
\partial_tP(x,t)\geq -P(x,t) C_0+\nabla P\cdot\nabla P
\geq-\delta_2 L C_0+c_1^2\geq c_2,
\end{equation}
where $c_2:=\min\{\frac{c_1^2}{2},c_1\}$ and $\delta_2:=\min\{\frac{c_1^2}{2LC_0},\delta_1\}$.

Combining~\eqref{space-lin-deg} and~\eqref{time-lin-degenerate}, then \eqref{lin-degenerate} holds for any $(x,t)\in A_{\delta_2}\cap\Omega$ as
\begin{equation*}
\hat{\nabla}_{\hat{\mu}}P= \frac{1}{\sqrt{2}}\nabla_{\mu}P+\frac{1}{\sqrt{2}}\partial_t P\geq c_2,
\end{equation*} 
and the proof is completed.
\end{proof}

\paragraph{\large $C^{1,\alpha}$ regularity of free boundary.}
Without loss of generality, we set $(0,0)\in \Gamma$.   To deal with the difficulties caused by the source term of the porous medium type reaction-diffusion equation~\eqref{de}, we  introduce the following scaling: 
\begin{equation*}
P_{\delta_3}(x,t)=\frac{1}{\delta_3}P(\delta_3x,\delta_3t)\quad\text{in }A_1
\end{equation*} 
with $\delta_3\in (0, \delta_2)$ determined later, which satisfies 
\begin{equation*}\label{P_delta}
\partial_t P_{\delta_3}=(m-1)P_{\delta_3}\Delta P_{\delta_3}+\nabla P_{\delta_3}\cdot\nabla P_{\delta_3}+(m-1)\delta_3P_{\delta_3}G(\delta_3P_{\delta_3})\quad\text{in }A_1.
\end{equation*}
Form Lemma~\ref{deg-rate} and Theorem~\ref{t3}, there exists a constant $L>0$ such that
\begin{equation}\label{est-delta}
\nabla_{\mu}P_{\delta_3}, \partial_t P_{\delta_3}\geq c_2, \quad\Delta P_{\delta_3}\geq -\delta_3C_0',\quad  |P_{\delta_3},\  \nabla P_{\delta_3},\ \partial _tP_{\delta_3}|\leq L,\quad\text{in }A_1\cap \{P_{\delta_3}>0\}.
\end{equation}

To prove the $C^{1,\alpha}$ regularity of free boundary,   we introduce iteration function $P_k$ as  
\begin{equation}\label{P_k}
P_k(x,t):=\frac{1}{J^k}P_{\delta_3}(J^kx,J^kt)\quad\text{in }A_1
\end{equation}
with $J\in(0,1)$ determined later,  which is a solution to the following equation:
\begin{equation}\label{P_kE}
\partial_t P_{k}=(m-1)P_{k}\Delta P_{k}+\nabla P_{k}\cdot\nabla P_{k}+(m-1)J^k\delta_3P_{k}G(J^k\delta_3P_{k})\quad\text{in }A_1\text{ and }k\in\mathbb{N}.
\end{equation}
Thanks to~\eqref{est-delta}, we have
\begin{equation}\label{est-k}
\nabla_{\mu}P_{k},\partial_t P_k\geq c_2, \quad \Delta P_{k}\geq -\delta_3J^kC_0',\quad |P_{k},\ \nabla P_{k},\ \partial _tP_{k}|\leq L,\quad\text{in }A_1\cap \{P_{k}>0\}.
\end{equation}

As initiated in~\cite{CW1990}, the $C^{1,\alpha}$ regularity of free boundary comes true under the help of the following three lemmas, the main idea is to propagate the regularity form the positive parts to the  boundary.
\begin{lemma}\label{bound-data}
For $J\in(0,1)$ and $k\geq 1$, let $P_k$ be given in~\eqref{P_k}, and suppose that $P_k$ satisfies~\eqref{est-k}. For any $\varepsilon\in (0,1)$, then there exist $r\leq\frac{\varepsilon}{2L}$ and $C\leq 1$ depending on $\varepsilon, L,\sigma$ such that 
\begin{equation}\label{P_k-eps}
P_k\leq \varepsilon \quad \text{in }B_{2r}\times[-2r,2r].
\end{equation} 
Moreover, for any $\gamma\in(0,\varepsilon)$, $p\in \hat{W}_{\hat{\mu}_k,\theta_k}\cap\textbf{S}^{n}$ and $\tau:=C\varepsilon^{-1}\cos<p,\hat{\nabla}P_k(\mu,-2r)>$, we have
\begin{equation}\label{P_k-ga}
P_k((x,t)+\gamma p)\geq (1+\gamma\tau)P_k(x,t)\quad\text{on }B_{3/4}\times [-2r,2r]\cap \{P_k=\varepsilon\}.
\end{equation}
\end{lemma}

\begin{lemma}\label{incre-rate}
Let $\varepsilon>0$ be small enough depending only on $L,\sigma$, and let $r,P_k,\tau$ be from Lemma~\ref{bound-data}. If $\omega$ is a sup-solution to~\eqref{P_kE} such that $\omega\geq P_k$ in $A_1$, and for some $\gamma\in(0,\varepsilon)$, 
\begin{equation}\label{ep-bo-co}
\omega(x,t)\geq (1+\tau\gamma)P_k(x,t)\quad \text{in }B_{1/2}\times[-2r,2r]\cap \{P_k=\varepsilon\}, 
\end{equation}
then 
\begin{equation}\label{in-ep}
\omega(x,t)\geq (1+\tau\gamma)P_k(x,t)\quad \text{in }B_{1/4}\times[-2r,2r]\cap \{P_k\leq\varepsilon\}.
\end{equation}
\end{lemma}
\begin{lemma}\label{ini-rate}
Let $\varepsilon,r,\tau,\omega,\gamma$ be from Lemma~\ref{incre-rate}. There exist  small $\kappa,\delta_3>0$ depending only on $L,\sigma$ such that the following holds. Consider any smooth function $\phi:\mathbb{R}^n\to \mathbb{R}_+$ such that $\phi$ is supported in $B_{2r}$ and $\phi,|\nabla\phi|,|D^2\phi|\leq \kappa\tau\gamma$. If $P_k\leq \varepsilon$ in $A_{2r}$,  then it holds
\begin{equation*}
\omega(x,t)\geq P_k(x+(t+2r)\phi(x)\mu,t)\quad\text{in }A_{2r}.
\end{equation*}
\end{lemma}
\begin{proof}[\underline{\textbf{Proof of Lemma~\ref{bound-data}}}] To begin with, \eqref{est-k} and $r\leq \frac{\varepsilon}{2L}$ directly yield~\eqref{P_k-eps}.

We argue by Harnack's inequality of the uniform parabolic equation.  Let us define $\phi_k:=\hat{\nabla}_{p}P_k\geq0$ for any $p\in \hat{W}_{\hat{\mu}_{k},\theta_k}$ which satisfies 
\begin{equation*}
\begin{aligned}
\partial_t\phi_k=&(m-1)P_k\Delta \phi_k+(m-1)\phi_k\Delta P_k+2\nabla \phi_k\cdot\nabla P_k\\
&+(m-1)J^k\delta_3\phi_kG(J^k\delta_3 P_k)+(m-1)J^{2k}\delta_3^2P_kG'(\delta_3 J^kP_k) \phi_k\\
\geq &(m-1)P_k\Delta \phi_k+2\nabla \phi_k\cdot\nabla P_k-\phi_k(m-1)J^{k}\delta_3(C_0'+K_G),
\end{aligned}
\end{equation*}
where the last inequality is derived by~\eqref{KG}. Set $\tilde{\phi}_k:=e^{(m-1)\delta_3J^k(C_0'+K_G)t}\phi_k$, it holds by the direct computation that
 \begin{equation*}
 \partial_t\tilde{\phi}_k\geq (m-1)P_k\Delta \tilde{\phi}_k+2\nabla \tilde{\phi}_k\cdot\nabla P_k.
 \end{equation*}

In $(x,t)\in B_{7/8}\times [-3r,3r]\cap \{P_k\geq \frac{1}{2}\varepsilon\}$,  by Harnack's inequality, there exists $C>0$ depending on $L,r,\varepsilon$ such that 
\begin{equation*}
\tilde{\phi}_{k}(x,t)\geq C\tilde{\phi}_k(\mu,-2r)\quad\text{for }(x,t)\in B_{3/4}\times [-2r,2r]\cap \{P_k>\varepsilon\}.
\end{equation*}
In another words, it holds for $\tilde{C}=Ce^{-(m-1)\delta_3 J^k(C_0'+K_G)4r}$ that
\begin{equation}\label{me-va}
\phi_{k}(x,t)\geq \tilde{C}\phi_k(\mu,-2r)\quad\text{for }(x,t)\in B_{3/4}\times [-2r,2r]\cap \{P_k>\varepsilon\}.
\end{equation}
Hence, for any $(x,t)\in B_{3/4}\times [-2r,2r]\cap \{P_k=\varepsilon\}$, it holds by the differential mean value theorem and \eqref{me-va} that
\begin{equation*}
\begin{aligned}
\frac{P_k((x,t)+\gamma p)-P_k(x,t)}{\gamma}&\geq \tilde{C}\phi_k(\mu,-2r)=\tilde{C}|\hat{\nabla}P_k|\cos<p,\hat{\nabla}P_k(\mu,-2r)>\\
&\geq \tilde{C}c_2\cos<p,\hat{\nabla}P_k(\mu,-2r)>\\
&=\tilde{C}c_2\varepsilon^{-1}P_k(x,t)\cos<p,\hat{\nabla}P_k(\mu,-2r)>,
\end{aligned}
\end{equation*}
 which yields~\eqref{P_k-ga} when we set  $C:=\min\{\tilde{C}c_2,1\}$. The proof is completed.

\end{proof}
\begin{proof}[\underline{\textbf{Proof of Lemma~\ref{incre-rate}}}]Let the non-negative $f\in C^1(B_{1/2})$ satisfy 
\begin{equation}\label{ffdef}
f=0\ \text{in }B_{1/4};\quad f=\varepsilon\ \text{on }B_{1/2};\quad |\nabla f|\leq 10\varepsilon;\quad |\Delta f|\leq 10\varepsilon.
\end{equation}
For any given $\alpha\in(-2r,2r)$, we define 
\begin{equation*}
\zeta_k(x,t):=P_k(x,t)+\tau\gamma(P_k(x,t)+\varepsilon(t+\alpha)-f(x))_+\quad \text{ in }B_{1/2}\times [-2r,-\alpha].
\end{equation*}

Our next goal is to prove that $\zeta_k$ is a sub-solution to~\eqref{P_kE} in $B_{1/2}\times [-2r,-\alpha]\cap \{P_k\leq \varepsilon\}$. As in \cite{CW1990},  define $g=\tau\gamma s^+:=\tau\gamma \max \{0,s\}$, then it follows $g'=\tau\gamma \textbf{1}_{s>0}$ and $g''\geq0$ in the sense of distribution. Denote $g_k:=g(P_k+\varepsilon(t+\alpha)-f)$,  we have
\begin{equation}\label{zeta-e}
\begin{aligned}
\partial_t\zeta_k&=\partial_t(P_k+g_k)=(1+g')\partial_t P_k+\varepsilon g',\\
\nabla \zeta_k&=\nabla(P_k+g_k)=(1+g')\nabla P_k-g'\nabla f,\\
\Delta \zeta_k&=\Delta(P_k+g_k)=\Delta P_k+g'(\Delta P_k-\Delta f)+g''|\nabla P_k-\nabla f|^2\\
&\geq (1+g')\Delta P_k-g'\Delta f.\\
\end{aligned}
\end{equation}
Then, we take~\eqref{est-k}, \eqref{ffdef} and~\eqref{zeta-e} into consideration and obtain
\begin{equation*}
\begin{aligned}
\partial_t&\zeta_k-(m-1)\zeta_k\Delta\zeta_k-|\nabla\zeta_k|^2-(m-1)\delta_3J^k\zeta_kG(\delta_3 J^k\zeta_k)\\
\leq& (1+g')(\partial_tP_k-(m-1)P_k\Delta P_k-|\nabla P_k|^2-(m-1)\delta_3J^kP_kG(\delta_3J^kP_k))\\
&+\varepsilon g' -(m-1)g_k(1+g')\Delta P_k+g'\zeta_k\Delta f+2g'(1+g')\nabla P_k\cdot\nabla f-g'^2|\nabla f|^2\\
&-g'(1+g')|\nabla P_k|^2+(m-1)\delta_3J^kP_kG(\delta_3J^kP_k))-(m-1)\delta_3J^k\zeta_kG(\delta_3 J^k\zeta_k)\\
\leq &\varepsilon g'-g'(1+g')|\nabla P_k|^2-(m-1)g'(P_k+\varepsilon(t+\alpha)-f)(1+g')\Delta P_k+2g'(1+g')\nabla P_k\cdot\nabla f\\
&+g'\zeta_k\Delta f+(m-1)\delta_3J^kP_k(G(\delta_3J^kP_k)-G(\delta_3J^k\zeta_k))+(m-1)\delta_3J^k(P_k-\zeta_k)G(\delta_3 J^k\zeta_k).\\
\leq&g'[\varepsilon+20L(1+g')\varepsilon+100g'\varepsilon^2+2\delta_3 J^kC_0'(m-1)(1+g')\varepsilon-c_2^2]+(m-1)K_G\delta_3J^{k}g_k\\
\leq &g'[\varepsilon\big(1+2\delta_3 J^kC_0'(m-1)(1+g')+20L(1+g')+100g'\varepsilon+2(m-1)K_G\delta_3J^{k}L\big)-c_2^2].
\end{aligned}
\end{equation*}
where the seventh line is derived by~\eqref{KG} and the fact $g_k\leq 2g'\varepsilon$ is used in the last line.

Since $\gamma\in(0,\varepsilon)$, we get
\begin{equation*}
g'\leq \tau\gamma \leq C\leq 1.
\end{equation*}
 Choosing $0<\varepsilon\leq\frac{c_2^2}{2+40L}\leq \frac{c_2^2}{1+4\delta_3 J^kC_0'(m-1)+40L+100\varepsilon+2(m-1)K_G\delta_3J^{k}}$ with small $\delta_3,\varepsilon>0$, then it holds in $B_{1/2}\times [-2r,-\alpha]\cap \{P_k\leq \varepsilon\}$ that
 \begin{equation*}
\partial_t\zeta_k-(m-1)\zeta_k\Delta\zeta_k-|\nabla\zeta_k|^2-(m-1)\delta_3J^k\zeta_kG(\delta_3 J^k\zeta_k)\leq0.
\end{equation*}
It follows from \cite[Proposition 2.3]{CW1990} that
 \begin{equation*}
\omega(x,t)\geq \zeta_k(x,t) \quad \text{in }B_{1/4}\times[-2r,-\alpha]\cap \{P_k\leq\varepsilon\}.
\end{equation*}
By the definition of $\zeta_k$ and setting $t=-\alpha$,  we conclude
\begin{equation*}
\omega(x,-\alpha)\geq (1+\gamma\tau)P_k(x,-\alpha)\text{ for }B_{1/4}\cap \{P_k(\cdot,-\alpha)\leq\varepsilon\}\text{ with any }\alpha\in(-2r,2r),
\end{equation*}
and the proof is completed. 
\end{proof}
\begin{proof}[\underline{\textbf{Proof of Lemma~\ref{ini-rate}}}] On account of~\eqref{est-k}, we use the interior estimate of the uniform parabolic equation as in~\cite{CW1990} and obtain
\begin{equation*}
|P_k\nabla_{i,j}P_k|\leq C_4\quad \text{in }A_{2r},
\end{equation*}
where $C_4>0$ is a constant depending on $c_2,L$, the universal constant and Lipschitz's constant of $\Gamma\cap A_{2r}$, and $i,j=1,...,n$.  

To this end,  we define 
\begin{equation*}
\xi_k(x,t):=(1+\tau\gamma)P_k\big(x+(t+2r)\phi\mu,t\big)\quad\text{in }A_{2r}.
\end{equation*}
By means of the definition of $\phi$ and the conclusions in Lemma~\ref{incre-rate}, we have
\begin{align*}
\omega\geq \zeta_k\text{ on the parabolic boundary of }\Sigma:=A_{2r}\cap\{P_k\leq\varepsilon\}.
\end{align*}

Writing $\tau'=\tau\gamma\leq 1$, it completely follows from \cite[Proposition 2.4]{CW1990} that
\begin{align*}
\partial_t\xi_k=&(1+\tau')(\partial_t P_k+\nabla_{\mu}P_k\phi),\\
\nabla\xi_k=&(1+\tau')(\nabla P_k+\nabla_{\mu}P_k(t+2r)\nabla\phi),\\
\Delta \xi_k=&(1+\tau')[\Delta P_k+2(t+2r)\nabla\nabla_{\mu}P_k\cdot\nabla\phi+\nabla_{\mu}\nabla_{\mu}P_k(t+2r)^2|\nabla\phi|^2+\nabla_{\mu}P_k(t+2r)\Delta \phi].
\end{align*}
Taking the above equalities, the fact $\tau'\leq 1$ and $\phi,|\nabla\phi|,|D^2\phi|\leq \kappa\tau'$ into account, we get
\begin{align*}
&\partial_t\xi_k-(m-1)\xi_k\Delta\xi_k-|\nabla\xi_k|^2-(m-1)\delta_3J^k\xi_kG(\delta_3J^k\xi_k)\\
\leq&(1+\tau')[\partial_tP_k-(m-1)P_k\Delta P_k-|\nabla P_k|^2-(m-1)\delta_3J^kP_kG(\delta_3J^kP_k)]\\
&+(1+\tau')\nabla_{\mu}P_k\phi-\tau'(1+\tau')|\nabla P_k|^2-2(1+\tau')^2\nabla_{\mu}P_k(t+2r)\nabla P_k\cdot\nabla\phi\\
&-(1+\tau')^2|\nabla_{\mu}P_k(t+2r)\nabla\phi|^2-\tau'(1+\tau')(m-1)P_k\Delta P_k+(1+\tau')\tau'(m-1)\delta_3J^{k}K_GP_k\\
&-(1+\tau')^2P_k[2(t+2r)\nabla\nabla_{\mu}P_k\cdot\nabla\phi+\nabla_{\mu}\nabla_{\mu}P_k(t+2r)^2|\nabla\phi|^2+\nabla_{\mu}P_k(t+2r)\Delta \phi]\\
\leq&\tau'(1+\tau')[-c_2^2+(m-1)\delta_3J^kC_0'L+(m-1)\delta_3J^{k}K_GL+\kappa C(L+C_4)],
\end{align*}
 where $C>0$ is a universal constant. If $\kappa,\delta_3$ are small, then we get
 \begin{align*}
 \partial_t\xi_k-(m-1)\xi_k\Delta\xi_k-|\nabla\xi_k|^2-(m-1)\delta_3J^k\xi_kG(\delta_3J^k\xi_k)\leq 0\quad\text{in }A_{2r}.
 \end{align*}
 Furthermore, it holds by the comparison principle that
 \begin{align*}
 \omega(x,t)\geq \xi_k(x,t)\geq P_k(x+(t+2r)\phi(x)\mu,t)\quad\text{in }A_{2r},
 \end{align*}
and the proof is completed.
\end{proof}

Based on Lemmas~\ref{bound-data}-\ref{ini-rate}, we show the iteration process similar to that in~\cite{CW1990,KZ2021} which yields the $C^{1,\alpha}$ regularity of free boundary.
\begin{proposition}[Improvement of monotonicity]\label{imp-mon}Suppose that $(0,0)\in \Gamma$. Then, there exist $J,s\in(0,1)$ and $\delta_3\in (0,\delta_2)$ such that 
\begin{equation}\label{-itera}
\hat{\nabla}_{p}P_k\geq \frac{1}{2L}J^k\quad\text{in }A_1\cap\{P_k>0\}
\end{equation}for all $k\in\mathbb{N},\ p\in \hat{W}_{\hat{\mu}_k,\theta_k}\cap \textbf{S}^n\text{ with }\theta_k:=\frac{\pi}{2}-s^k(\frac{\pi}{2}-\theta_0)$.
\begin{proof}
 We prove this by induction. Assume that  it holds for some large $L>0$ and any $k\in\mathbb{N}$ that
 \begin{equation*}
\hat{\nabla}_pP_k\geq \frac{1}{2L}J^k.
  \end{equation*}
  
  We note that $\phi$ is given in Lemma~\ref{ini-rate}, and suppose that 
 \begin{equation*}
 \phi\geq cr^2\kappa\tau\gamma\quad\text{in }B_r\text{ for some small constant } c>0.
 \end{equation*}
Let $\omega(x,t)=P_{k}((x,t)+\gamma p)\text{ for }p\in \hat{W}_{\theta_k,\hat{\mu}_k}\cap \textbf{S}^n$, then it follows from Lemmas~\ref{bound-data}-\ref{ini-rate} that
\begin{equation*}
P_{k}((x,t)+\gamma p)\geq P_k(x+(t+2r)\phi(x)\mu,t)\quad\text{in }A_{2r}.
\end{equation*}
By the differential mean value theorem, it holds for $c'=\frac{cr^3\kappa}{L}$ that 
\begin{equation*}
P_k((x,t)+\gamma p)\geq P_k(x+(t+2r)\phi(x)\mu,t)\geq P_k(x,t)+\frac{t+2r}{L}\phi(x)\geq P_k(x,t)+c'\tau\gamma
\end{equation*}
in $A_r$, which yields
\begin{equation*}
\hat{\nabla}_p P_{k}(x,t)\geq c'\tau\quad\text{in }A_r.
\end{equation*}
We take the definition of $\tau$ in Lemma~\ref{bound-data} and obtain
\begin{equation*}
\hat{\nabla}_{p}P_k(x,t)\geq C'\cos<p,\hat{\nabla}P_k(\mu,-2r)>,
\end{equation*}
in which $C'=c'C\varepsilon^{-1}$.  Hence, we conclude by the small $c>0$ that
\begin{equation*}
\cos<p,\hat{\nabla}P_k(x,t)>\geq \frac{C'}{L}\cos<p,\hat{\nabla}P_k(\mu,-2r)>.
\end{equation*}
In addition, we have
\begin{equation*}
\cos<p,\hat{\nabla}P_k(\mu,-2r)>=\frac{\hat{\nabla}_p P_k(\mu,-2r)}{|\hat{\nabla} P_k(\mu,-2r)|}\geq\frac{J^k}{2L^2}.
\end{equation*}
Let the radius be 
\begin{equation*}
\rho_k(p):=\frac{C'}{8L}\cos<p,\hat{\nabla}P_k(\mu,-2r)>, 
\end{equation*}
then it holds for $q\in\hat{W}_{\hat{\mu}_{k+1},\theta_k+\rho_k(p)}$ that
\begin{equation*}
\begin{aligned}
\cos<&q,\hat{\nabla}P_k(x,t)>\\
=&\cos<q,p>\cos<p,\hat{\nabla}P_k(x,t)>+\sin<q,p>\sin<p,\hat{\nabla}P_k(x,t)>\\
\geq&\frac{1}{2}\cos<p,\hat{\nabla}P_k(x,t)>-2\sin<q,p>\\
\geq&\frac{C'}{2L}\cos<p,\hat{\nabla}P_k(\mu,-2r)>-2\rho_k(p)\\
\geq&\frac{C'}{4L}\cos<p,\hat{\nabla}P_k(\mu,-2r)>,
\end{aligned}
\end{equation*}
 where the fact $0\leq\sin \alpha \leq \alpha$ for $0\leq \alpha\leq \frac{\pi}{2}$ is used in the forth line. Hence, it holds
\begin{equation*}
\begin{aligned}
\hat{\nabla}_qP_k(x,t)=|\hat{\nabla}P_k(x,t)|\cos<q,\hat{\nabla}P_k(x,t)> \geq\frac{c_2C'}{8L^3}J^k\quad\text{in }A_r\cap\{P_k>0\}.
\end{aligned}
\end{equation*}
Choosing $0<J\leq \min\{\frac{c_2 C'}{4L^2},r\}$, then we have
\begin{equation}\label{Pk}
\hat{\nabla}_pP_k(x,t) \geq \frac{1}{2L}J^{k+1}\quad\text{in }A_r\cap\{P_k>0\}\text{ and }p\in \hat{W}_{\hat{\mu}_{k+1},\theta_{k+1}},
\end{equation}
where $\theta_{k+1}=s(\frac{\pi}{2}-\theta_k)+\theta_k=\frac{\pi}{2}-s^{k}(\frac{\pi}{2}-\theta_0)$ for some fixed $s\in(0,1)$.

In conlusion, we set $P_{k+1}=\frac{1}{J}P_k(Jx,Jt)$, it concludes by~\eqref{Pk} that
\begin{equation*}
\hat{\nabla}_pP_{k+1}(x,t) \geq \frac{1}{2L}J^{k+1}\quad\text{in }A_1\cap\{P_{k+1}>0\}\text{ and }p\in \hat{W}_{\hat{\mu}_{k+1},\theta_{k+1}},
\end{equation*}
and the proof is completed.
\end{proof}

\begin{proof}[\underline{\textbf{Proof of Theorem~\ref{t2}}}] Similar arguments as in \cite[Theorem 1]{CW1990},  $\Gamma\cap A_{\delta_3}$ is a $C^{1,\alpha}$ surface. More precisely, there exist $\alpha\in(0,1)$ such that 
\begin{equation}
|\nu(x)-\nu(y)|\leq |x-y|^{\alpha}\quad\text{for any }x,y\in \Gamma\cap A_{\delta_3},
\end{equation}
where $\nu(x)$ is the outer normal vector at the position $x\in\Gamma\cap A_{\delta_3}$. The proof is completed.
\end{proof}

\end{proposition}

\subsection*{Acknowledegements} The author would thank Mr. Zhennan Zhou (BICMR) and Mr. Jiajun Tong (BICMR) for helpful discussion and warm hospitality.

\appendix 
\section*{Appendix: proof of Lemmas~\ref{l18}-\ref{l19} }\label{app A}
\begin{proof}[\underline{\textbf{Proof of Lemma~\ref{l18}}}]
For simplicity, we denote $x-x_0$ by $x$ and $t-t_0$ by $t$. Taking a rescaling, we consider the function
\begin{equation*}
\tilde{P}(x,t)=\frac{\sigma}{R^2}P(Rx,\sigma t),
\end{equation*}
which satisfies the following equation
\begin{equation}\label{a11}
\begin{aligned}
\partial_t\tilde{P}&=(m-1)\tilde{P}\Delta\tilde{P}+\nabla\tilde{P}\cdot\nabla\tilde{P}+(m-1)\sigma\tilde{P}G(\frac{R^2}{\sigma}\tilde{P})\\
&=(m-1)\sigma\tilde{P}(\Delta P+G(P))+|\nabla\tilde{P}|^2.
\end{aligned}
\end{equation}
Furthermore, we have
\begin{align}
&\tilde{P}(x,0)=0,\quad x\in B_{1},\label{a6}\\
&\oint_{B_1}\tilde{P}(x,1)dx\leq c_0.\label{a7}
\end{align}
By the AB estimate (Lemma \ref{lAB}), we obtain
\begin{align}
&\partial_t\tilde{P}\geq -(m-1)\eta C_0\tilde{P}\geq-\varepsilon\tilde{P},\label{a8}\\
&\Delta\tilde{P}=\sigma\Delta P=\sigma\big(\Delta P+G(P)\big)-\sigma G(P)\geq -\eta C_0'\geq -\varepsilon,\label{a9}
\end{align}
where $\varepsilon=\eta\max\{(m-1)C_0,\ C_0'\}$.
The inequality \eqref{a9} implies
\begin{equation*}
\Delta(\tilde{P}+\frac{\varepsilon}{2n}|x|^2)\geq0.
\end{equation*}
Consequently, for any $x\in B_{1/2}$, we have
\begin{align*}
\tilde{P}(x,1)+\frac{\varepsilon}{2n}|x|^2&\leq\oint_{B_{1/2}(x)}(\tilde{P}(\xi,1)+\frac{\varepsilon}{2n}|\xi|^2)d\xi\\
&\leq 2^n\oint_{B_1}\tilde{P}(\xi,1)d\xi+\frac{\varepsilon}{8n}.
\end{align*}
Using \eqref{a7}, we conclude
\begin{equation}\label{a10}
\tilde{P}(x.1)\leq 2^nc_0+\frac{\varepsilon}{A_n},\  x\in B_{1/2}, \  A_n=8n.
\end{equation}
By direct computations, it follows from \eqref{a8} that
\begin{equation}\label{me}
\tilde{P}(x,1)\geq e^{-\varepsilon(1-t)}\tilde{P}(x,t),\quad t\in[0,1].
\end{equation}
Combining \eqref{me} with \eqref{a10}, we get
\begin{equation}\label{a13}
\tilde{P}(x,t)\leq e^{\varepsilon}(2^nc_0+\frac{\varepsilon}{A_n}),\ x\in B_{1/2},\ 0<t<1.
\end{equation}

We construct a special super-solution to Eq. \eqref{a11} to achieve our goal.
Consider the function
\begin{equation*}
V(x,t)=\lambda[\frac{t}{36}+\frac{1}{6}(|x|-\frac{1}{3})]_+.
\end{equation*}
It satisfies
\begin{equation}\label{a12}
\partial_t V\geq (m-1)V(\Delta V+\sigma G(0))+\nabla V\cdot\nabla V
\end{equation}
on the support of $V$, if and only if
\begin{equation*}
1\geq 6(m-1)[\frac{t}{6}+|x|-\frac{1}{3}](\frac{\lambda(n-1)}{6|x|}+\sigma G(0))+\lambda
\end{equation*}
holds. It is easy to know $V(x,t)>0$ if and only if
\begin{equation*}
|x|\geq \frac{1}{3}-\frac{t}{6}.
\end{equation*}
Thus \eqref{a12} is valid provided $\lambda=\lambda(m,n)$ and $\eta(>\sigma)$ are small if $0\leq t\leq1$.

In the following, we verify the initial and boundary conditions such that $V(x,t)$ is a super-solution on the cylinder $B_{1/2}\times [0,1]$.

When $t=0$ and $|x|<\frac{1}{2}$, we have
\begin{equation*}
\tilde{P}=0\leq V,
\end{equation*}
and when $|x|=\frac{1}{2}$ and $0<t<1$, we have
\begin{equation*}
\tilde{P}\leq\frac{\lambda}{36}\leq V,
\end{equation*}
which is valid if $\varepsilon$ and $c_0$ are sufficiently small in~\eqref{a13}. 
 
 Then, it holds by the comparison principle for the parabolic equation~\eqref{a11} that
   \begin{equation*}
   V\geq \tilde{P},\ (x,t)\in B_{1/2}\times[0,1].
   \end{equation*}
   Hence, in particular, we obtain
   \begin{equation*}
   \tilde{P}(x,t)\leq V(x,t)=0\quad\text{for } |x|\leq \frac{1}{6},
   \end{equation*}
and the proof is completed.
\end{proof}

\begin{proof}[\underline{\textbf{Proof of Lemma~\ref{l19}}}]
Without loss of generality, we set $(x_0,t_0)=(0,0)$. We consider re-scaled formula of the density as
\begin{equation*}
\tilde{\rho}(x,t):=(\frac{\sigma}{R^2})^{\frac{1}{m-1}}\rho(Rx,\sigma t),
\end{equation*}
then $\tilde{\rho}$ solves the following re-scaled density equation
\begin{equation*}
\partial_t\tilde{\rho}=\Delta \tilde{\rho}^m+\sigma \tilde{\rho} G(\frac{R^2}{\sigma}\tilde{P})
\end{equation*}
where $\tilde{P}=\frac{m}{m-1}\tilde{\rho}^{m-1}$. The AB estimate on $\tilde{v}$ as in \eqref{a9} implies that
\begin{equation*}
\Delta \tilde{\rho}^m\geq -\varepsilon\tilde{\rho}
\end{equation*}
where $\varepsilon=\eta\max\{(m-1)C_0,\ C_0'\}$.

 Let us define
 \begin{equation*}
 \varphi(t):=\oint_{B_{1}} \tilde{\rho}^m(x,t)dx.
 \end{equation*}
 We show that $\varphi(t)$ for $t\in[0,\lambda]$ stays strictly positive. By direct computations, we have
 \begin{align*}
 \varphi(0)=\frac{\sigma}{R^2}\oint_{B_R}\rho^{m}(x,0)dx\geq \nu.
 \end{align*}
 Due to \eqref{a8}, we have
 \begin{equation}\label{a15}
 \partial_t\tilde{\rho}^m\geq-\varepsilon\tilde{\rho}^m.
 \end{equation}
 Consequently, we obtain
\begin{equation}\label{a16}
\varphi(t)\geq e^{-\varepsilon t} \varphi(0)\geq e^{-\varepsilon \lambda}\nu\quad \text{for}\ t\in[0,\lambda].
\end{equation}

Next, we construct a differential inequality equation to obtain a upper bound for the growth of $\varphi$ over time.

As in \cite{CF1979,CF1980}, we introduce the Green's function in a ball $B_{1}$ so that $G_{1}$ solves
\begin{equation*}
\Delta G_{1}=\alpha_n \textbf{1}_{B_{1}}-\gamma_n\delta(x)\quad \text{and} \quad G_1=|\nabla G_1|=0,\quad\text{on }\partial B_1,
\end{equation*}
where $\gamma_n$ is a positive constant depending only on $n$, $\textbf{1}_A$ is the characteristic function of a given set $A$, $\delta(x)$ is the Dirac measure  and $\alpha_n=n(n-2)$. According to the Green's function on $B_1$, $G_{1}$ as in \cite{CF1980} is defined as
\begin{equation*}
G_{1}(x)=
\begin{cases}
|x|^{2-n}-1+\frac{n-2}{2}(|x|^2-1),&n\geq 3,\\
\log\frac{1}{|x|}-\frac{1}{2}(1-|x|^2),&n=2.\\
\end{cases}
\end{equation*}
Then, for $0<t<1$, we have
\begin{align*}
\int_{B_1}G_1\tilde{\rho}dx\geq& \int_{0}^{t}\int_{B_1}G_1(x)\tilde{\rho}_s(x,s)dxds\\
=&\int_{0}^{t}\int_{B_1}G_1(x)\Delta\tilde{\rho}^mdxds+\sigma\int_{0}^{t}\int_{B_1}G_1(x)\tilde{\rho}G(\frac{R^2}{\sigma}\tilde{P})dxds\\
\geq&\int_{0}^{t}\int_{B_1}G_1(x)\Delta\tilde{\rho}^mdxds+\sigma G(P_H)\int_{0}^{t}\int_{B_1}G_1(x)\tilde{\rho}dxds\\
=&-\gamma_n\int_{0}^{t}\tilde{\rho}^m(0,s)ds+\alpha_n\int_{0}^{t}\int_{B_1}\tilde{\rho}^m(x,s)dxds+\sigma G(P_H)\int_{0}^{t}\int_{B_1}G_1(x)\tilde{\rho}dxds,
\end{align*}where $G(P)\geq G(P_H)$ on $[0,P_H]$ defined in \eqref{G} is used.
Hence, it follows
\begin{equation*}
\begin{aligned}
\int_{B_1}&G_1(x)\tilde{\rho}(x,t)dx-\sigma G(P_H)\int_{0}^{t}\int_{B_1}G_1(x)\tilde{\rho}dxds\\
&\geq-\gamma_n\int_{0}^{t}\tilde{\rho}^m(0,s)ds+\alpha_n\int_{0}^{t}\int_{B_1}\tilde{\rho}^{m}(x,s)dxds.
\end{aligned}
\end{equation*}

Let us recall the definition of $\varphi(t)$ as
\begin{equation*}
\varphi(t)=\int_{B_1}\tilde{\rho}^m(x,t)dx.
\end{equation*}
 Parallel to \cite[Lemma 2.3]{CF1980} with the dimension $n\geq2$ and the diffusion exponent $m>1$, we have 
 $$\int_{B_1}G_1(x)\tilde{\rho}(x,t)dx\leq \tilde{C}_1\varepsilon^{\delta}+\tilde{C}_2(\varphi(t))^{\frac{1}{m}},$$
then it holds by Young's inequality that 
\begin{equation}\label{a17}
\int_{0}^{t}\varphi(s)ds\leq C_1\int_{0}^{t}\tilde{\rho}^m(0,s)ds+C_2[\varepsilon^{\delta}+\varepsilon\lambda(\varepsilon^\delta+1)]+C_3(\varphi(t))^{\frac{1}{m}},
\end{equation}
where $\tilde{C}_1$, $\tilde{C}_2$, $C_1$, $C_2$, $C_3$ and $\delta$ are positive constants depending only on $m,n$.

   Now, assume that \eqref{a14} is not satisfied, that is
   \begin{equation*}
   \tilde{\rho}^m(0,\lambda)\leq c.
   \end{equation*}
   It follows from \eqref{a15} that
   \begin{equation*}
   \tilde{\rho}^{m}(0,t)\leq e^{\varepsilon \lambda}c\quad\text{for}\ 0\leq t\leq\lambda.
   \end{equation*}
   Set
   \begin{equation*}
   \tilde{C}:=\frac{(C_1 \lambda c e^{\varepsilon\lambda}+C_2[\varepsilon^{\delta}+\varepsilon\lambda(\epsilon^\delta+1)])e^{\varepsilon\lambda/m}}{\nu^{1/m}},
   \end{equation*}
   by means of \eqref{a16} and \eqref{a17}, we conclude
   \begin{equation*}
   \int_{0}^{t}\varphi(s)ds\leq \frac{1}{B}(\varphi(t))^{1/m}\quad\text{with }\frac{1}{B}=C_3+\tilde{C}.
   \end{equation*}

    Let
   \begin{equation*}
   \Psi(t):=\int_{0}^{t}\varphi(s)ds,
   \end{equation*}
we have
   \begin{equation}\label{a18}
   \Psi'(t)\geq(B\Psi(t))^{m}.
   \end{equation}

   From \eqref{a16}, it follows
   \begin{equation*}
   \Psi(t)\geq At,\ A=\nu e^{-\varepsilon\lambda}.
   \end{equation*}
   We can construct a sub-function $\chi(t)$ of $\Psi(t)$ by solving the following ODE:
   \begin{equation}\label{a19}
   \chi'(t)=(B\chi(t))^{m},\ t>t_0,\ \chi(t_0)=A t_0,
   \end{equation}
it holds by the comparison principle that
\begin{equation}\label{a20}
\Psi(t)\geq\chi(t),\ t_0\leq t\leq\lambda.
\end{equation}
By directly solving ODE \eqref{a19},  we obtain
\begin{equation*}
(m-1)\chi^{m-1}(t)=\frac{1}{C-B^mt}
\end{equation*}
where $C$ is a constant determined by
\begin{equation*}
 \frac{1}{C-B^m t_0}=(m-1)A^{m-1}t_0^{m-1}.
 \end{equation*}
 Since
 \begin{equation*}
 \chi(t)\to\infty,\ as\ t\to\frac{C}{B^m},
 \end{equation*}
 but \eqref{a20} implies
 \begin{equation*}
 \Psi(t)\to\infty,\ as\ t\to\frac{C}{B^m}
 \end{equation*}
 under the assumption $\lambda\geq \frac{C}{B^m}$, which is contradicted with the bound of $\Psi$. Thus, \eqref{a14} must be valid provided
 \begin{equation}\label{a21}
 \lambda\geq \frac{C}{B^m}=\frac{e^{\varepsilon\lambda(m-1)}}
 {(m-1)\nu^{m-1}t_0^{m-1}B^m}+t_0
 \end{equation}
 where the definitions of $A,C$ were used.

 In addition, choosing $t_0=\frac{\lambda}{2}$, we  verify by the definition of $B,c$ that \eqref{a21} satisfies
 \begin{equation*}
 \lambda^{m}\nu^{m-1}\geq\frac{2^{m}e^{m-1}}{m-1}(\frac{(C_1 \lambda c e^{\varepsilon\lambda}+C_2[\varepsilon^{\delta}+\varepsilon\lambda(\varepsilon^\delta+1)])e^{\varepsilon\lambda/m}}{\nu^{1/m}}+C_3)^{m}.
 \end{equation*}
Therefore, \eqref{a22} guarantees that~\eqref{a21} holds. The proof is completed.
\end{proof}

%

\end{document}